\nonstopmode
\documentclass[10pt, reqno]{amsart}
\usepackage{graphicx}
\usepackage{latexsym}
\usepackage{fancyhdr}
\usepackage{amsmath, amssymb, amsthm}
\usepackage[all]{xy}
\usepackage{pdflscape}
\usepackage{longtable}
\usepackage{rotating}
\usepackage{verbatim}
\usepackage%[pagebackref]
{hyperref}
\hypersetup{
    linkcolor = blue,
    citecolor = blue,
    urlcolor  = blue,
    colorlinks = true,
}

\usepackage{etoc}
\etocsettocstyle
  {\medskip}          % no heading
  {\bigskip}

\etocsetstyle{section}                                          % start (before first entry)
  {}
  {\leavevmode\leftskip 0pt\relax}            % prefix (before each entry)
  {\etocnumber \etocname          % contents
   \dotfill\etocpage\par\smallskip}
  {}                                          % finish (after last entry)

\etocsetstyle{subsection}
  {}
  {\leavevmode\leftskip 1.5em\relax}
  {\etocnumber\ \etocname\dotfill\etocpage\par\smallskip}
  {}

\etocsettocstyle
  {\section*{\contentsname}\small}           % TOC heading
  {\bigskip}                                 % after TOC

\usepackage{cleveref}
\usepackage{subfigure}
\usepackage{mathrsfs}
\usepackage{mdwlist}
\usepackage{dsfont}

\usepackage{mathtools}
\mathtoolsset{showonlyrefs=true}

\usepackage{float}
\usepackage{color}
\usepackage{stmaryrd}
\usepackage{parskip}

\usepackage{tkz-base}
\usepackage{pgfplots}
\pgfplotsset{compat=newest}

\usepackage{tkz-euclide}
\usepackage{etoolbox}

\definecolor{teal}{rgb}{0.0, 0.5, 0.5}

\newcounter{mnotecount}[section]

\newcommand{\rmnote}[1]{}%{\mnote{#1}}

\allowdisplaybreaks

\DeclareFontFamily{U}{mathb}{\hyphenchar\font45}
\DeclareFontShape{U}{mathb}{m}{n}{
      <5> <6> <7> <8> <9> <10> gen * mathb
      <10.95> mathb10 <12> <14.4> <17.28> <20.74> <24.88> mathb12
      }{}
\DeclareSymbolFont{mathb}{U}{mathb}{m}{n}
\DeclareFontSubstitution{U}{mathb}{m}{n}

\theoremstyle{plain}
\newtheorem*{theorem*}{Theorem}
\newtheorem{theorem}{Theorem}[section]
\newtheorem*{lemma*}{Lemma}
\newtheorem{lemma}[theorem]{Lemma}
\newtheorem*{assumption*}{Assumption}

\newtheorem*{proposition*}{Proposition}
\newtheorem{proposition}[theorem]{Proposition}
\newtheorem*{corollary*}{Corollary}
\newtheorem{corollary}[theorem]{Corollary}
\newtheorem*{claim*}{Claim}

\newtheorem*{conjecture*}{Conjecture}

\newtheorem*{question*}{Question}

\newtheorem{problem}[theorem]{Problem}
\newtheorem*{result*}{Result}

\theoremstyle{definition}
\newtheorem*{definition*}{Definition}
\newtheorem{definition}[theorem]{Definition}
\newtheorem*{example*}{Example}
\newtheorem{example}[theorem]{Example}

\newtheorem*{algorithm*}{Algorithm}
\newtheorem*{remark*}{Remark}
\newtheorem*{remarks*}{Remarks}
\newtheorem{remark}[theorem]{Remark}

\newtheorem*{convention*}{Convention}

\Crefname{l}{Lemma}{Lemmas}    %declares the type of the environment for cleverref, use \label[<type>]{<label>}
\Crefname{p}{Proposition}{Propositions}
\Crefname{t}{Theorem}{Theorems}
\Crefname{c}{Corollary}{Corollaries}
\Crefname{r}{Remark}{Remarks}
\Crefname{d}{Definition}{Definitions}
\Crefname{e}{Example}{Examples}
\Crefname{q}{Question}{Questions}

\numberwithin{equation}{section}

\def\al{\alpha}
\def\be{\beta}
\def\ga{\gamma}
\def\de{\delta}
\def\ep{\epsilon}

\def\et{\eta}
\def\th{\theta}

\def\ka{\kappa}
\def\la{\lambda}

\def\rh{\rho}

\def\si{\sigma}

\def\ta{\tau}

\def\vh{\varphi}
\def\ch{\chi}

\def\Ga{\Gamma}
\def\De{\Delta}

\def\La{\Lambda}

\def\Ps{\Psi}

\def\C{\mathbb{C}}

\def\I{\mathbb{I}}

\def\N{\mathbb{N}}

\def\R{\mathbb{R}}

\def\cA{\mathcal{A}}

\def\cE{\mathcal{E}}
\def\cF{\mathcal{F}}

\def\cH{\mathcal{H}}

\def\cL{\mathcal{L}}

\def\cU{\mathcal{U}}
\def\cV{\mathcal{V}}

\def\sE{\mathscr{E}}

\def\p{\partial}
\def\id{\on{id}}

\def\eig{\mathcal{E}}
\def\eigu{\mathcal{E}_u}

\renewcommand{\Re}{\mathrm{Re}}

\def\<{\langle}
\def\>{\rangle}
\renewcommand{\o}{\circ}

\def\ol{\overline}
\def\ul{\underline}

\def\Lip{\on{Lip}}

\def\ol{\overline}
\def\ul{\underline}

\def\Lip{\on{Lip}}

\def\dd{\mathbf{d}}
\def\ddd{\widehat \dd}
\def\bs{\mathbf{s}}

\newcommand{\Iq}{{\mathcal{A}}_d (\C)}
\def\a#1{\left\llbracket{#1}\right\rrbracket}

\newcommand{\abs}[1]{\left|#1\right|}

\newcommand{\ra}{\rightarrow}

\def\Et{H}

\let\on=\operatorname

\newcommand{\sr}[1]%
{\ifmmode{}^\dagger\else${}^\dagger$\fi\ifvmode
\vbox to 0pt{\vss
 \hbox to 0pt{\hskip\hsize\hskip1em
 \vbox{\hsize3cm\raggedright\pretolerance10000
 \noindent #1\hfill}\hss}\vss}\else
 \vadjust{\vbox to0pt{\vss%
 \hbox to 0pt{\hskip\hsize\hskip1em%
 \vbox{\hsize3cm\raggedright\pretolerance10000%
 \noindent \small #1\hfill}\hss}\vss}}\fi%
}

\providecommand{\mapsfrom}{\kern.2em%
\setbox0=\hbox{$\leftarrow$\kern-.10em\rule[0.26mm]{0.1mm}{1.3mm}}\box0%
\kern.3em}

\title[Eigenvalue stability of Hermitian and normal matrices]
{Eigenvalue stability of Hermitian \\ and normal matrices}

\author[Adam Parusi\'nski and  Armin Rainer]
{Adam Parusi\'nski and Armin Rainer}

\address {Adam Parusi\'nski: Universit\'e C\^ote d'Azur,  CNRS,  LJAD, UMR 7351, 06108 Nice, France}
\email{adam.parusinski@univ-cotedazur.fr}

\address{Armin Rainer: Faculty of Mathematics and Geoinformation,
    Institute for Statistics and Mathematical Methods in Economics, E105-04,
TU Wien, Wiedner Hauptstraße 8, 1040 Vienna, Austria}
\email{armin.rainer@tuwien.ac.at}

\begin{document}

\begin{abstract}
    The ordered eigenvalues define a Lipschitz map on the real vector space of Hermitian $d \times d$ 
    matrices. We prove that this map acts continuously, but not uniformly continuously,
    by superposition on the Sobolev spaces $W^{1,q}$, 
    for all $1 \le q < \infty$, on bounded open domains.
    For $q=\infty$, the action is still well-defined and bounded but not continuous.
    We show that this stability result extends to normal matrices, where the eigenvalues 
    are naturally interpreted as multivalued Sobolev functions in the sense of Almgren. 
    Several applications are given, including the stability of singular values, condition numbers of matrices, 
    surface area of eigenvalue graphs, and compact self-adjoint operators in Hilbert space. 
\end{abstract}

\thanks{This research was funded in whole or 
    in part by the Austrian Science Fund (FWF)  DOI 10.55776/PAT1381823.
For open access purposes, the authors have applied a CC BY public copyright license to any author-accepted manuscript version arising from this submission.}
   %Particular nonlinear operators (superposition, Hammerstein, Nemytskiĭ, Uryson, etc.)
%    14P10,      %Semialgebraic sets and related spaces
%	26B05, 	    %Continuity and differentiation questions
%	26B35,  	%Special properties of functions of several variables, HÃ¶lder conditions, etc.
%	26E10,  	%$C^\infty$-functions, quasi-analytic functions
%    26E25,      %Set-valued functions 
%    32B20,  	%Semi-analytic sets and subanalytic sets
%46E15}      %Banach spaces of continuous, differentiable or analytic functions
%  58C20,    %Differentiation theory (Gateaux, FrÃ©chet, etc.) on manifolds
%	58C25}  	%Differentiable maps
	%58B10,  	%Differentiability questions
	%58B25,  	%Group structures and generalizations on infinite-dimensional manifolds
	%58C07,  	%Continuity properties of mappings
	%58D05} 	    %Groups of diffeomorphisms and homeomorphisms as manifolds
%\dedicatory{dedicatory}
\date{\today}

\maketitle

\setcounter{tocdepth}{1}
\tableofcontents

%\clearpage

%-----------------------------------------------------------------------------------------------------------------------
\section{Introduction}
%-----------------------------------------------------------------------------------------------------------------------

The perturbation theory of Hermitian and normal operators has a long history. 
The analytic theory started with Rellich's work \cite{Rellich37}, culminated in Kato's monograph \cite{Kato76},
and extended to smooth perturbation theory in more recent papers, e.g., 
\cite{AKLM98}, \cite{KM03,KMRp,KMR}, 
\cite{RainerAC,RainerN, Rainer:2021vk}, and \cite{Parusinski:2020ac}. 
For a more comprehensive account of the history,
we refer to the recent survey article \cite{Parusinski:2023ab}.

For the present paper, Weyl's perturbation theorem \cite{Weyl12} plays a fundamental role. 
It asserts that the ordered eigenvalues form a Lipschitz map on the space of Hermitian matrices 
with Lipschitz constant $1$, when the space is equipped 
with the operator norm and the eigenvalues with the $\infty$-norm. Together with
L\"owner's theorem \cite{Lowner:1934aa}, which provides the corresponding result for the $2$-norm, 
and its extension to normal matrices by Hoffman and Wielandt \cite{Hoffman:1953aa},
this theorem is foundational for nearly all of the results in this paper. 
Of course, in the case of normal matrices, the eigenvalues must be equipped with a metric that 
accounts for the fact that they are general complex numbers not admitting a coherent ordering.
See also \cite{BhatiaDavisMcIntosh83} for a variant with the operator norm 
(where, however, the Lipschitz constant $C$ satisfies $1<C<3$).
For locally Lipschitz curves of normal matrices, each continuous choice of the eigenvalues is actually 
locally Lipschitz, as shown by \cite{RainerN}. If the (real) parameter space is at least $2$-dimensional, 
continuous choices of the eigenvalues of normal matrices may not exist.

The Lipschitz continuity of the eigenvalues as functions of the matrices 
immediately implies that Lipschitz families of Hermitian or normal matrices are 
assigned to the Lipschitz families of their eigenvalues, in a bounded way, that is, with 
controlled Lipschitz constants. 
This defines a map which we call the \emph{characteristic map}:  
it is given by superposition with the ordered eigenvalue map, in the Hermitian case, 
and with the unordered eigenvalue map, in the normal case.
In this paper, we are interested in the continuity properties of the characteristic map.

The characteristic map is not continuous with respect to the Lipschitz topology
(on the eigenvalues), but it is continuous with respect to the Sobolev $W^{1,q}$ topology, for every $1 \le q<\infty$.
Specifically, 
we prove that for Hermitian matrices the characteristic map is continuous, but not uniformly continuous, 
from $W^{1,q}$ matrices to ordered $W^{1,q}$ eigenvalues,
while for normal matrices it is continuous from 
$W^{1,q}$ matrices to unordered $W^{1,q}$ eigenvalues --- in both cases, for every $1 \le q < \infty$.
In the latter setting, the unordered eigenvalues are naturally treated as multivalued 
Sobolev functions in the sense of Almgren \cite{Almgren00}; see also \cite{De-LellisSpadaro11}.  

These results are considerably stronger than the related continuity properties of the solution map for 
hyperbolic and general polynomials established, respectively, 
in the recent papers \cite{Parusinski:2024aa} and \cite{Parusinski:2024ab}.
The solution map is defined in analogy to the characteristic map.
Recall that Bronshtein's theorem \cite{Bronshtein79} implies that the ordered roots of a $C^{d-1,1}$ family of 
hyperbolic (i.e.\ monic real-rooted) polynomials of degree $d$ are locally Lipschitz; 
explicit bounds for the Lipschitz constants of the roots in terms of the $C^{d-1,1}$ norms 
of the coefficients were obtained by \cite{ParusinskiRainerHyp}.
By \cite{Parusinski:2024aa}, the solution map for hyperbolic polynomials of degree $d$ is continuous 
from $C^d$ coefficients to $C^{0,1}$ roots with respect to the $W^{1,q}_{\on{loc}}$ topology on the roots, 
for each $1 \le q < \infty$ --- but not with respect to the $C^{0,1}$ topology.
For general complex polynomials, the optimal Sobolev regularity $W^{1,q}$, for $1 \le q < \frac{d}{d-1}$,
of the roots was established in \cite{ParusinskiRainerAC, ParusinskiRainer15, Parusinski:2020aa}, 
and the existence and boundedness of the solution map rests on these results. 
As shown in \cite{Parusinski:2024ab}, the solution map is then continuous from $C^d$ coefficients 
to unordered $W^{1,q}$ roots, for each $1 \le q <\frac{d}{d-1}$.

Two further points underline the superiority of the matrix setting.
First, unlike the polynomial case, the size of the matrices plays no role.
Second, the Sobolev spaces $W^{1,q}$ of Hermitian matrices and their eigenvalues 
are defined globally, whereas for hyperbolic polynomials the spaces of coefficients and roots are necessarily local.

These results also have concrete applications. 
We discuss several, including singular values, the condition number for matrices, the surface area of eigenvalue graphs, 
and compact self-adjoint operators in Hilbert space (made possible by independence of dimension).

It should be mentioned that the orthonormal eigenvectors do not share the regularity of the eigenvalues 
and, in general, fail to be even continuous in this setting.

In addition to Weyl's perturbation theorem and its relatives, a key ingredient in the proofs is a 
splitting lemma for complex triangular matrices due to \cite{Parusinski:2020ac}, where it appears in 
a slightly different context. It provides a local uniform unitary block-diagonalization of Hermitian and normal matrices 
whenever not all eigenvalues coincide. This suggests an inductive argument on the size of the matrices
and has the advantage of working directly with the matrices themselves rather than through 
their characteristic polynomials.

In the following, we present our results in more detail.

%-----------------------------------------------------------------------------------------------------------------------
\subsection{Hermitian matrices}
%-----------------------------------------------------------------------------------------------------------------------

Let $M_d(\C)$ be the complex vector space of complex $d\times d$ matrices. We endow $M_d(\C)$ with the 
Frobenius norm $\|A\|_2 = (\sum_{i,j=1}^d |a_{ij}|^2)^{1/2}$.

Let $\on{Herm}(d) = \{A \in M_d(\C) : A^* = A\}$ denote the real vector space of complex Hermitian $d \times d$ matrices.
We associate with $A \in \on{Herm}(d)$ its increasingly ordered eigenvalues 
\[
    \la_1^\uparrow(A) \le \la_2^\uparrow(A) \le \cdots \le \la_d^\uparrow(A)
\]
and thus obtain a continuous map
\begin{equation} \label{eq:Herm}
    \la^\uparrow = (\la_1^\uparrow, \ldots, \la_d^\uparrow) : \on{Herm}(d) \to \R^d.
\end{equation}
In fact, by a result of L\"owner \cite{Lowner:1934aa} (see \Cref{p:Loewner}), this map is even Lipschitz continuous: for $A,B \in \on{Herm}(d)$,
\begin{equation} \label{eq:Frobenius}
    \|\la^\uparrow(A) - \la^\uparrow(B)\|_2 \le \|A-B\|_2.
\end{equation}

As a consequence (see \Cref{p:mcharHerm}),
the map \eqref{eq:Herm} induces a bounded map
\begin{equation} \label{eq:EW}
    \eig:= (\la^\uparrow)_* : W^{1,q}(U, \on{Herm}(d)) \to W^{1,q}(U, \R^d), \quad A \mapsto \la^\uparrow \o A,
\end{equation}
for each $1 \le q < \infty$,
which takes $W^{1,q}$ families of Hermitian matrices to their increasingly ordered $W^{1,q}$ eigenvalues, 
where $U \subseteq \R^m$ is open and bounded. 
(This also true for $q=\infty$ but this case will be discussed separately below.)

We will call $\eig = (\la^\uparrow)_*$ the \emph{characteristic map}, independently of the domain and codomain which may vary in different contexts.

We will show that the characteristic map \eqref{eq:EW} is continuous.

\begin{theorem} \label[t]{thm:eigW}
    Let $1 \le q < \infty$.
    Let $U \subseteq \R^m$ be open and bounded. 
    Then the characteristic map    
    \begin{equation*} 
        \eig : W^{1,q}(U,\on{Herm}(d)) \to W^{1,q}(U,\R^d), \quad A \mapsto \la^\uparrow \o A,
    \end{equation*}
    is continuous.
\end{theorem}

\Cref{thm:eigW} will be proved in \Cref{sec:H1} for $m=1$ and the general case will be concluded from this special 
case in \Cref{sec:proof2}.

\begin{remark}
Note that superposition by any Lipschitz function $f : \R \to \R$ induces a continuous map 
$f_* : W^{1,q}(U) \to W^{1,q}(U)$ for every $1 \le q < \infty$, 
where $U \subseteq \R^m$ is open and bounded, by \cite{Marcus:1979aa}. 
However, there are Lipschitz maps $f : \R^k \to \R^\ell$, with $k\ge 2$, 
such that the induced map $f_* : W^{1,q}(U,\R^k) \to W^{1,q}(U,\R^\ell)$ is not continuous; 
see \cite{Musina:1991aa}, where also sufficient conditions for continuity are given. 
But our proof of \Cref{thm:eigW} will follow a different strategy.
\end{remark}

\begin{remark}
    The image of the characteristic map is 
    \[
        \eig (W^{1,q}(U,\on{Herm}(d))) = W^{1,q}(U,\R^d_{\uparrow}),
    \]
    the space of $f \in W^{1,q}(U,\R^d)$ with $f(U) \subseteq \R^d_\uparrow := \{x \in \R^d : x_1 \le \cdots \le x_d\}$. The 
    continuous surjection $\eig : W^{1,q}(U,\on{Herm}(d)) \to W^{1,q}(U,\R^d_{\uparrow})$ admits a continuous right-inverse
    which takes $\la = (\la_1,\ldots,\la_d) \in W^{1,q}(U,\R^d_{\uparrow})$ to the diagonal matrix with the diagonal entries $\la_1,\ldots,\la_d$.
\end{remark}

We state a simple consequence of \Cref{thm:eigW}.
Here
$\|f\|_{L^q(U,\R^d)} := \big\| \|f\|_2 \big\|_{L^q(U)}$, 
see \Cref{ssec:notation} for notation.

\begin{corollary} \label[c]{cor:se}
    Let $1 \le q < \infty$.
    Let $U \subseteq \R^m$ be open and bounded.
    If $A_n \to A$ in $W^{1,q}(U, \on{Herm}(d))$ as $n \to \infty$, then, 
    for all $1 \le j \le m$,
    \begin{align*}
        \big\| \|\p_j (\eig(A))\|_2 - \|\p_j(\eig(A_n))\|_2 \big\|_{L^q(U)} \to 0 
        \intertext{and}
        \|\p_j (\eig(A_n))\|_{L^q(U,\R^d)} \to \|\p_j(\eig(A))\|_{L^q(U,\R^d)}  
    \end{align*}
    as $n \to \infty$.
\end{corollary}

\begin{proof} 
    Let us set $\la := \eig(A)$ and $\la_{n} := \eig(A_n)$. Then 
    \begin{align*}
        &\big| \|\p_j \la\|_{L^q(U,\R^d)} - \|\p_j \la_n\|_{L^q(U,\R^d)} \big|
        =
        \big| \big\| \|\p_j \la\|_2 \big\|_{L^q(U)} - \big\| \|\p_j \la_n\|_2 \big\|_{L^q(U)} \big|
        \\
        & \quad \le 
        \big\| \|\p_j \la\|_2 - \|\p_j \la_n\|_2 \big\|_{L^q(U)}
        \le \big\| \|\p_j \la - \p_j \la_n\|_2 \big\|_{L^q(U)}
        = \|\p_j \la - \p_j \la_n\|_{L^q(U,\R^d)}
    \end{align*}
    so that the assertions follow from \Cref{thm:eigW}.
\end{proof}

Let us now turn to the Lipschitz case (i.e., $q=\infty$).
Clearly, \eqref{eq:Frobenius} also induces a  
bounded map 
\begin{equation} \label{eq:E}
    \eig : C^{0,1}(\ol U, \on{Herm}(d)) \to C^{0,1}(\ol U, \R^d), \quad A \mapsto \la^\uparrow \o A,
\end{equation}
taking Lipschitz families of Hermitian matrices to their increasingly ordered Lipschitz eigenvalues,
with Lipschitz constants
satisfying
\begin{equation} \label{eq:eigLip}
    |\eig(A)|_{C^{0,1}(\ol U,\R^d)} \le |A|_{C^{0,1}(\ol U,M_d(\C))},
\end{equation}
with respect to the $2$-norms on $\on{Herm}(d)$ and $\R^d$.

As an immediate consequence of \Cref{thm:eigW}, we get
the following theorem which answers Question 1.11 in \cite{Parusinski:2024aa}.

\begin{theorem} \label[t]{thm:eigmap}
    Let $U \subseteq \R^m$ be open and bounded. 
    Then the characteristic map    
    \begin{equation*} 
        \eig : C^{0,1}(\ol U,\on{Herm}(d)) \to C^{0,1}_q(\ol U,\R^d), \quad A \mapsto \la^\uparrow \o A,
    \end{equation*}
    is continuous, for all $1\le q<\infty$, where $C^{0,1}_q(\ol U,\R^d)$ denotes 
    the set $C^{0,1}(\ol U,\R^d)$ 
    equipped with the trace topology of the inclusion 
$C^{0,1}(\ol U,\R^d) \to W^{1,q}(U,\R^d)$.
\end{theorem}

But we will see in \Cref{ex:A} that the map 
$\eig : C^{0,1}(\ol U,\on{Herm}(d)) \to C^{0,1}(\ol U,\R^d)$ is \emph{not} continuous: 
the natural topology on the target $C^{0,1}(\ol U,\R^d)$ is too strong.

As a corollary of \Cref{thm:eigmap}, we find that $\eig$ is continuous 
into the H\"older space $C^{0,\al}(\ol U,\R^d)$, carrying its natural topology,
for all $0<\al<1$.

\begin{corollary} \label[c]{cor:eigmap}
    Let $U \subseteq \R^m$ be a bounded open Lipschitz domain. 
    Then the  characteristic map    
    \begin{equation} 
        \eig : C^{0,1}(\ol U,\on{Herm}(d)) \to C^{0,\al}(\ol U,\R^d), \quad A \mapsto \la^\uparrow \o A,
    \end{equation}
    is continuous, for all $0< \al<1$, but not for $\al=1$.
\end{corollary}

In \Cref{cor:eigmap}, which will be proved in \Cref{sec:proof2}, $U$ is a Lipschitz domain, since we use Morrey's inequality.

In the setup of \Cref{thm:eigmap}, we will actually prove a stronger result, in \Cref{sec:H1} (for $m=1$) and \Cref{sec:proof2}:

\begin{theorem} \label[t]{thm:mptw}
    Let $U \subseteq \R^m$ be an open set. 
    Let $A_n \to A$ in $C^{0,1}(U,\on{Herm}(d))$ as $n \to \infty$. 
    Then, for each $1 \le j \le m$ and almost every $x \in U$,  
    \begin{equation} \label{eq:mptw}
        \p_j (\eig(A_n))(x) \to \p_j (\eig(A))(x) \quad \text{ as } n \to \infty. 
    \end{equation}
\end{theorem}

As a consequence, we obtain a second proof of \Cref{thm:eigmap} based on the dominated convergence theorem; 
see \Cref{thm:eigmapS}. 
Clearly, \eqref{eq:mptw} generally fails, if we only assume   
$A_n \to A$ in $W^{1,q}(U,\on{Herm}(d))$ for some $1 \le q <\infty$ as $n \to \infty$.

By Egorov's theorem \cite{Egoroff:1911aa}, we may conclude that
$\p_j(\eig(A_n)) \to \p_j(\eig(A))$ almost uniformly on $U$ as $n \to \infty$. 
In general, the convergence is not uniform on $U$; see \Cref{ex:A}.

%-----------------------------------------------------------------------------------------------------------------------
\subsection{Normal matrices}
%-----------------------------------------------------------------------------------------------------------------------

The complex normal $d \times d$ matrices form a real algebraic subset 
\[
    \on{Norm}(d) := \{A \in M_d(\C) : A^*A = AA^*\}
\]
of the vectorspace $M_d(\C)$ of all complex $d\times d$ matrices. 
The set $\on{Norm}(d)$ is a real singular cone; see \cite{Huhtanen:2001aa} for background on its geometry.

With $A \in \on{Norm}(d)$ we associate its $d$ eigenvalues $\la_1(A),\ldots, \la_d(A)$ (with multiplicities) 
and the unordered eigenvalue vector
\[
    \La(A) = [\la_1(A),\ldots, \la_d(A)].
\]
In this way, we obtain a map 
\begin{equation} \label{eq:La}
    \La : \on{Norm}(d) \to \cA_d(\C),
\end{equation}
where $\cA_d(\C)$ is the space of unordered complex $d$-tuples (see \Cref{ssec:AdC}), 
a complete metric space with respect to the metric 
\[
    \dd_2([z],[w]) := \min_{\si \in \on{S}_d} \Big(\sum_{i=1}^d |z_i - w_{\si(i)}|^2 \Big)^{1/2},
\]
where $\on{S}_d$ is the symmetric group.
The map \eqref{eq:La} is Lipschitz continuous, by a result of Hoffman and Wielandt \cite{Hoffman:1953aa} 
(see \Cref{p:HW}): for $A,B \in \on{Norm}(d)$,
\begin{equation} \label{eq:HWintro}
    \dd_2(\La(A),\La(B)) \le \|A-B\|_2.
\end{equation}

Due to Almgren \cite{Almgren00},
there exists a bi-Lipschitz embedding $\De : \Iq \to \R^N$, where $N =N(d)$,
which can be used to define Sobolev spaces of $\cA_d(\C)$-valued functions: 
for open $U \subseteq \R^m$ and $1 \le q \le \infty$, we set 
\begin{equation} \label{eq:dSobintro}
    W^{1,q}(U,\cA_d(\C)) := \{f : U \to \cA_d(\C) : \De \o f \in W^{1,q}(U,\R^N)\}.  
\end{equation}
An equivalent intrinsic definition of $W^{1,q}(U,\cA_d(\C))$ 
is due to De Lellis and Spadaro \cite{De-LellisSpadaro11} (see also \Cref{d:W1p}).
The space $W^{1,q}(U,\cA_d(\C))$ carries the metric
\begin{equation} \label{eq:Almgrenintro}
    \rh_\De^{1,q}(f,g) := \|\De \o f - \De \o g\|_{W^{1,q}(U,\R^N)}
\end{equation}
which turns it into a complete metric space; see \cite[Lemma 3.1]{Parusinski:2024ab}.
The topology induced by this metric does not depend on the choice of the Almgren embedding $\De$,
see \Cref{t:topmult}.

The map \eqref{eq:La} induces a bounded map
\begin{equation} \label{eq:eigu}
    \eigu :=\La_* : W^{1,q}(U,\on{Norm}(d)) \to W^{1,q}(U, \cA_d(\C)), \quad A \mapsto \La \o A,
\end{equation}
for each $1 \le q < \infty$,
where $U \subseteq \R^m$ is a bounded open set, see \Cref{p:mcharnorm}. 
(The Lipschitz case $q=\infty$ is discussed below.)
By $W^{1,q}(U,\on{Norm}(d))$ we mean the (nonlinear) space of all $A \in W^{1,q}(U,M_d(\C))$ such that 
$A(U) \subseteq \on{Norm}(d)$, topologized by its inclusion in  $W^{1,q}(U,M_d(\C))$.

In this setup, we call $\eigu = \La_*$ the \emph{characteristic map}.

\begin{theorem} \label[t]{thm:eiguW}
    Let $1 \le q < \infty$.
    Let $U \subseteq \R^m$ be open and bounded. 
    Then the characteristic map    
    \begin{equation*} 
        \eigu : W^{1,q}(U,\on{Norm}(d)) \to W^{1,q}(U,\cA_d(\C)), \quad A \mapsto \La \o A,
    \end{equation*}
    is continuous.
\end{theorem}

It turns out that \Cref{thm:eiguW} is a generalization of \Cref{thm:eigW}; see \Cref{ssec:AdR}. 

In the following corollary of \Cref{thm:eiguW}, 
$|\dot \La|$ denotes the metric speed and $\sE_q(\La)$ 
the $q$-energy of the curve 
$\La  \in W^{1,q}(I, \cA_d(\C))$; see \Cref{ssec:ACq} for definitions.

\begin{corollary} \label[c]{thm:mainse}
    Let $1 \le q < \infty$.
    Let $I \subseteq \R$ be a bounded open interval.
    Let $A_n \to A$ in $W^{1,q}(I,\on{Norm}(d))$ as $n\to \infty$ and set
    $\La := \eigu(A), \La_n:= \eigu(A_n) : I \to \cA_d(\C)$. 
    Then
    \begin{align*}
        \|\dd_2(\La,\La_n) \|_{L^\infty(I)}  &\to 0, 
        \\
        \big\| |\dot\La|- |\dot \La_n| \big\|_{L^q(I)} &\to 0, 
        \\
        \big| \sE_q(\La)- \sE_q(\La_n) \big| &\to 0, 
    \end{align*}
    as $n \to \infty$.
\end{corollary}

In the setup of \Cref{thm:mainse},
there always exist $W^{1,q}$ parameterizations 
$\la,\la_n : I \to \C^d$ of the eigenvalues of $A$, $A_n$, respectively; see \Cref{p:Wselection-1}.
We will see that \Cref{thm:mainse} follows from the following corollary;
for their proofs see \Cref{ssec:mse}.

\begin{corollary} \label[c]{cor:main1}
    Let $1 \le q < \infty$.
    Let $I \subseteq \R$ be a bounded open interval.
    Let $A_n \to A$ in $W^{1,q}(I,\on{Norm}(d))$ as $n\to \infty$.  
    Let $\la,\la_n \in W^{1,q}(I,\C^d)$ be parameterizations of the eigenvalues of $A$, $A_n$, respectively.
    Then
    \begin{align*} 
       &\big\| \|\la'\|_2 - \|\la_n'\|_2 \big\|_{L^q(I)}\to 0, 
        \\
       &\|\la_n'\|_{L^q(I,\C^d)} \to \|\la'\|_{L^q(I,\C^d)}, 
    \end{align*}
    as $n \to \infty$.
\end{corollary}

See also \Cref{thm:main1varWm}, \Cref{thm:main1varm}, and \eqref{eq:conseq} below.

Let us now consider the Lipschitz case ($q=\infty$).
We immediately get from \eqref{eq:HWintro} that, for all $A,B \in \on{Norm}(d)$, 
\begin{equation*}
    \|\dd_2(\eigu(A), \eigu(B)) \|_{L^\infty(U)} \le \|A-B\|_{L^\infty(U,M_d(\C))},
\end{equation*}
But, as evidenced by the Hermitian case, the (by \eqref{eq:HWintro} induced) bounded map  
\[
    \eigu : C^{0,1}(\ol U,\on{Norm}(d)) \to C^{0,1}(\ol U, \cA_d(\C))
\]
is not continuous.
Nevertheless, as an immediate consequence of \Cref{thm:eiguW}, 
we have the following result.

\begin{theorem} \label[t]{thm:multnormal}
    Let $U \subseteq \R^m$ be open and bounded. 
    The characteristic map    
    \begin{equation*} 
        \eigu : C^{0,1}(\ol U,\on{Norm}(d)) \to W^{1,q}(U,\cA_d(\C)), \quad A \mapsto \La \o A,
    \end{equation*}
    is continuous, for all $1\le q<\infty$.
\end{theorem}

Clearly, $f \in C^{0,\al}(\ol U, \cA_d(\C))$ if and only if $\De \o f \in C^{0,\al}(\ol U,\R^N)$
which makes the following corollary meaningful; see \Cref{sec:normmult} for its proof.

\begin{corollary} \label[c]{cor:eigu}
    Let $U \subseteq \R^m$ be a bounded open Lipschitz domain. 
    If $A_n \to A$ in 
    $C^{0,1}(\ol U,\on{Norm}(d))$, then
    $\eigu(A_n) \to \eigu(A)$ in $C^{0,\al}(\ol U, \cA_d(\C))$ as $n \to \infty$, 
    for all $0< \al<1$, but not for $\al=1$, 
    in the sense that 
    \[
        \|\De \o \eigu(A)- \De \o \eigu(A_n)\|_{C^{0,\al}(\ol U,\R^N)} \to 0 \quad \text{ as } n \to \infty,
    \]
    for each Almgren embedding $\De: \cA_d(\C) \to \R^N$.
\end{corollary}

We have the following variants of \Cref{thm:eiguW} and \Cref{thm:multnormal}.

\begin{theorem} \label[t]{thm:main1varWm}
    Let $1 \le q < \infty$.
    Let $U \subseteq \R^m$ be open and bounded.
    Let $A_n \to A$ in $W^{1,q}(U,\on{Norm}(d))$ as $n\to \infty$.  
    Assume that $\la,\la_n \in W^{1,q}(U, \C^d)$ are parameterizations 
    of the eigenvalues of $A,A_n$, respectively.
    Assume that $\la_n \to \la$ in $L^\infty$ on $\cL^{m-1}$-almost all line segments in $U$ parallel to the coordinate axes.
    Then 
    $\p_j \la_n \to \p_j \la$ in $L^q(U,\C^d)$
    as $n \to \infty$, for all $1 \le j \le m$.
\end{theorem}

\begin{theorem} \label[t]{thm:main1varm}
    Let $U \subseteq \R^m$ be open and bounded.
    Let $A_n \to A$ in $C^{0,1}(\ol U,\on{Norm}(d))$ as $n\to \infty$.  
    Assume that $\la,\la_n \in C^{0,1}(\ol U, \C^d)$ are parameterizations 
    of the eigenvalues of $A,A_n$, respectively.
    Assume that $\la_n \to \la$ in $L^\infty$ on $\cL^{m-1}$-almost all line segments in $U$ parallel to the coordinate axes.
    Then 
    $\p_j \la_n \to \p_j \la$ almost everywhere in $U$
    as $n \to \infty$, for all $1 \le j \le m$.
\end{theorem}

\Cref{thm:main1varWm} and \Cref{thm:main1varm} will be proved in \Cref{ssec:variants} (for $m=1$) and \Cref{sec:normmult}.
It is evident (see the proof of \Cref{cor:se}), that in their setup, 
    for all $1 \le j \le m$,
    \begin{align} \label{eq:conseq}
        \begin{split}
        &\big\| \|\p_j \la\|_2 - \|\p_j \la_n\|_2 \big\|_{L^q(U)} \to 0, 
        \\
        &\|\p_j \la_n\|_{L^q(U,\C^d)} \to \|\p_j\la\|_{L^q(U,\C^d)},  
        \end{split}
    \end{align}
    as $n \to \infty$.
    However, if $m\ge 2$ the desired parameterizations $\la,\la_n$ may not always exist. 
    (For $q>m$, the elements of $W^{1,q}(U)$, where $U \subseteq \R^m$, admit continuous representatives,    
    but the eigenvalues of multiparameter families of normal matrices cannot always be represented by continuous functions, e.g.,
    \[
      \begin{pmatrix}
        0 & x
        \\
        |x| & 0
    \end{pmatrix},  
    \]
    for $x \in \C$.)

%-----------------------------------------------------------------------------------------------------------------------
\subsection{Optimality of the results} \label{ssec:optimal}
%-----------------------------------------------------------------------------------------------------------------------

The results on eigenvalue stability obtained in this paper are essentially optimal.
Here we briefly mention several examples that demonstrate this fact; 
these examples will then be discussed in detail in \Cref{sec:examples}.
We also formulate an open problem. 

We will focus on the limitations in the Hermitian case, which clearly entail corresponding negative results in the more 
general normal case.

\begin{enumerate}
    \item {\it The map $\eig : C^{0,1}(\ol U,\on{Herm}(d)) \to C^{0,1}(\ol U,\R^d)$ is not continuous.} 
\end{enumerate}

For instance,  
\begin{equation} \label{eq:ex1}
        A_n(x) = \begin{pmatrix} \frac{1}n & x \\ x & - \frac{1}{n}  \end{pmatrix}
        \to 
        \begin{pmatrix} 0 & x \\ x & 0  \end{pmatrix} = A(x)  \quad \text{ as } n \to \infty
    \end{equation}
    uniformly in  all derivatives on every compact interval in $\R$. 
    Then $\eig(A_n) = (-a_n,a_n)$ and $\eig(A) = (-a,a)$, where 
    $a_n(x) = \sqrt{x^2 + \frac{1}{n^2}}$ and $a(x) = |x|$. We will see in 
    \Cref{ex:A} that $|a - a_n|_{C^{0,1}(\ol I)} \ge 2-\sqrt 2$, for all bounded open intervals $I$ containing $0$ and 
    all sufficiently large $n$. Also note that off $0$ the derivatives $a_n'$ tend pointwise to $a'$ but 
    not uniformly in any neighborhood of $0$.

\begin{enumerate}
    \item[(2)] {\it For no $1 \le q < \infty$, the map $\eig : C^{0,1}(\ol U,\on{Herm}(d)) \to C^{0,1}_q(\ol U,\R^d)$ is uniformly continuous.} 
\end{enumerate}

Indeed, the matrices 
\[
        A_n(x) = \begin{pmatrix} \frac{1}n & \vh_n(x) \\ \vh_n(x) & - \frac{1}{n}  \end{pmatrix}, 
        \quad
        B_n(x) = \begin{pmatrix} \frac{1}{2n} & \vh_n(x) \\ \vh_n(x) & - \frac{1}{2n}  \end{pmatrix}, 
\]
where $\vh_n : [0,1] \to [0,\frac{1}{n}]$ is a sawtooth function with Lipschitz constant $1$, 
satisfy $\|A_n - B_n\|_{C^{0,1}([0,1],M_2(\C))} \to 0$ as $n\to \infty$, but 
for their nonnegative eigenvalues $a_n$ and $b_n$ we find $\|a_n' -b_n'\|_{L^q([0,1])} \ge \frac{1}{12\cdot 2^{1/q}}$,
in \Cref{ex:ucq}. 

\begin{enumerate}
    \item[(3)] {\it For no $\al \in (0,1)$, the map $\eig : C^{0,1}(\ol U,\on{Herm}(d)) \to C^{0,\al}(\ol U,\R^d)$ is uniformly continuous.} 
\end{enumerate}

This is demonstrated in \Cref{ex:Auc}, by the matrices 
\[
    A_n(x) = \begin{pmatrix} \frac{1}{n^r} & nx \\ nx & - \frac{1}{n^r}  \end{pmatrix}, 
        \quad
        B_n(x) = \begin{pmatrix} \frac{1}{2n^r} & nx \\ nx & - \frac{1}{2n^r}  \end{pmatrix}, 
\]
where $r = \frac{\al}{1-\al}$.

Even tough the maps in (2) and (3) are not uniformly continuous, 
it is very desirable to find effective moduli of continuity for their restrictions to interesting (for applications) 
compact subspaces of $C^{0,1}(\ol U, \on{Herm}(d))$.
For instance, for given $C>0$ and $\be \in (0,1]$, the set 
\begin{equation} \label{eq:K}
    K:=\big\{A \in C^{1,\be}(\ol U, \on{Herm}(d)) : \|A\|_{C^{1,\be}(\ol U,M_d(\C))} \le C \big\}, 
\end{equation} 
is a relatively compact subset of $C^{0,1}(\ol U, \on{Herm}(d))$, by the Arzel\`a--Ascoli theorem. 
Thus, $\cE|_{K} : K \to C^{0,1}_q(\ol U,\R^d)$, for $1 \le q< \infty$, and 
$\cE|_{K} : K \to C^{0,\al}(\ol U,\R^d)$, for $0 < \al< 1$, are uniformly continuous, 
by \Cref{thm:eigmap} and \Cref{cor:eigmap}.

\begin{problem}
    Find effective moduli of continuity for the restrictions of the characteristic map $\eig$ to 
    interesting compact subspaces.
\end{problem}

For instance, using the sequence $A_n$ in \eqref{eq:ex1} we will see in \Cref{ex:A2} 
that, for $1\le q <\infty$ and $q^{-1} < \al \le 1$, the map $\eig|_K : K \to C^{0,1}_q([0,1],\R^d)$ is not $\al$-H\"older continuous. 

\begin{remark}
    Due to Rellich \cite{Rellich37}, a real analytic curve of Hermitian matrices admits a real analytic 
    system of eigenvalues and a real analytic orthonormal frame of eigenvectors. There are various extensions 
    of this result, e.g.,
    to normal matrices or to other categories (smooth quasianalytic classes, formal power series), 
    see \cite{AKLM98}, \cite{RainerN}, and \cite{Parusinski:2020ac}.
    But in general there is no continuous orthonormal frame of the eigenvectors, as seen by the following $C^\infty$ example from \cite{Rellich37}:
    \[
      A(x) := e^{-1/x^2} 
       \begin{pmatrix}
           \cos x^{-1} & \sin x^{-1}
           \\
           \sin x^{-1} & -\cos x^{-1}
       \end{pmatrix} \quad \text{ for } x\in \R \setminus \{0\}, \quad A(0) := 0.
    \]

    Let us suppose that $I \ni x \mapsto A(x)$ is a smooth curve of Hermitian matrices with smooth distinct eigenvalues $\la_j$ and corresponding smooth orthonormal 
    eigenvectors $v_j$. Differentiating the equation $A(x)v_j(x) = \la_j(x)v_j(x)$ and taking the result in the Hermitian inner product with $v_k \ne v_j$ 
    gives 
    \[
        \langle v_k(x),v_j'(x)\rangle  = \frac{\langle v_k(x),A'(x)v_j(x) \rangle}{\la_j(x)-\la_k(x)}
    \]
    showing that the projection of the velocity $v_j'$ to $v_k$ is unbounded as $\la_j$ and $\la_k$ approach each other faster
    than the numerator tends to zero.
    This is connected to adiabatic theorems in physics and the quantum Hall effect.  
\end{remark}

%-----------------------------------------------------------------------------------------------------------------------
\subsection{Singular values}
%-----------------------------------------------------------------------------------------------------------------------

Let us consider the vector space $M_{D,d}(\C)$ of complex $D \times d$ matrices, where $d \le D$ (without loss of generality). 
The singular values of $A \in M_{D,d}(\C)$ are the nonnegative square roots of the eigenvalues 
of the Hermitian matrix $A^*A$, usually ordered decreasingly
\[
    \si_1(A) \ge \si_2(A) \ge \cdots \ge \si_d(A) \ge 0.
\]
This defines a map $\si=(\si_1,\ldots,\si_d) : M_{D,d}(\C) \to \R^d$.

Observing that the Hermitian matrix 
\[
   \mathbf A:= \begin{pmatrix}
        0 & \tilde A
        \\
        \tilde A^* & 0
    \end{pmatrix},
\]
where $\tilde A$ is the $D \times D$ matrix resulting from $A$ by adding $D-d$ columns consisting of zeros, has the eigenvalues 
\[
    \si_1(A), \ldots,\si_d(A), 0,\ldots,0, -\si_d(A),\ldots,-\si_1(A),
\]
we conclude from \eqref{eq:HW} that, for $A,B \in M_{D,d}(\C)$ and $1 \le i \le d$, 
\begin{align*}
    \sqrt 2 \, \|\si(A) - \si(B)\|_2 &\le \|\mathbf A - \mathbf B\|_2 
    = |\on{Tr}\big( (\mathbf A - \mathbf B)^*(\mathbf A - \mathbf B)\big)|^{1/2}
    \\
                          &= |2\, \on{Tr}\big( (\tilde A - \tilde B)^*(\tilde  A - \tilde B)\big)|^{1/2} 
                         = \sqrt 2 \, \|A-B\|_2,
\end{align*}
that is,
\begin{equation} \label{eq:estsv}
    \|\si(A) - \si(B)\|_2 \le \|A-B\|_2.
\end{equation}
Consequently, for each $1 \le q < \infty$, the map
\[
    \si_* : W^{1,q}(U,M_{D,d}(\C)) \to W^{1,q}(U,\R^d), \quad A \mapsto \si \o A,
\]
is well-defined and bounded, as well as 
\[
    \si_* : C^{0,1}(\ol U,M_{D,d}(\C)) \to C^{0,1}(\ol U,\R^d), \quad A \mapsto \si \o A.
\]

Furthermore, \Cref{thm:eigW}, \Cref{thm:eigmap}, \Cref{cor:eigmap}, and \Cref{thm:mptw} immediately give the following result.

\begin{theorem} \label[t]{thm:sv}
    Let $U \subseteq \R^m$ be open and bounded. Then:
    \begin{enumerate}
        \item For each $1 \le q < \infty$, the map $\si_* : W^{1,q}(U,M_{D,d}(\C)) \to W^{1,q}(U,\R^d)$ is continuous.
        \item The map $\si_* : C^{0,1}(\ol U,M_{D,d}(\C)) \to C^{0,1}_q(\ol U,\R^d)$ is continuous, for all $1\le q< \infty$.
        \item Assume additionally that $U$ is a Lipschitz domain. Then 
            the map $\si_* : C^{0,1}(\ol U,M_{D,d}(\C)) \to C^{0,\al}(\ol U,\R^d)$ is continuous, for all $0 < \al < 1$, but not for $\al=1$.
        \item If $A_n \to A$ in  $C^{0,1}(\ol U,M_{D,d}(\C))$, then 
            $\p_j(\si_*(A_n))(x) \to \p_j(\si_*(A))(x)$ as $n \to \infty$, for each $1 \le j \le m$ and almost every $x \in U$.
    \end{enumerate}
\end{theorem}

This theorem answers a question raised in \cite[Section 7.5]{Parusinski:2024aa}. 

%-----------------------------------------------------------------------------------------------------------------------
\subsection{Applications}
%-----------------------------------------------------------------------------------------------------------------------

Let us present a few applications of our results.

%-----------------------------------------------------------------------------------------------------------------------
\subsubsection{Functions defined on the spectrum or the singular values}
%-----------------------------------------------------------------------------------------------------------------------

Clearly, the stability results for the eigenvalues and singular values have immediate applications, e.g., 
for the spectral gaps $\la_{i+1}(A) - \la_i(A)$ or the Ky Fan norms $\sum_{i=1}^k \si_i(A)$, for $1 \le k \le d$.
This can be extended using a result of \cite{Marcus:1979aa}. 

\begin{corollary}
    Let $f : \R \to \R$ be a Lipschitz function. 
    Let $1 \le q < \infty$.
    Let $U \subseteq \R^m$ be open and bounded.
    Then the map 
    \begin{equation} \label{eq:superposition}
         W^{1,q}(U,M_{D,d}(\C)) \to W^{1,q}(U,\R), \quad A \mapsto f \o \si_i \o A,
    \end{equation}
    is well-defined and continuous, where $\si_1(A) \ge \cdots \ge \si_d(A)$ are the singular values of $A$.
    The statement remains true if $M_{D,d}(\C)$ is replaced by $\on{Herm}(d)$ and the ordered eigenvalues 
    are used instead of the singular values.  
\end{corollary}

\begin{proof}
    The map \eqref{eq:superposition} is the composite 
    \[
        W^{1,q}(U,M_{D,d}(\C)) \xrightarrow[\qquad]{(\si_i)_*} W^{1,q}(U,\R) \xrightarrow[\qquad]{f_*} W^{1,q}(U,\R).
    \]
    The first map is continuous by \Cref{thm:sv}, the second by \cite{Marcus:1979aa}.
    For Hermitian matrices use \Cref{thm:eigW} instead of \Cref{thm:sv}.
\end{proof}

For instance: 
\begin{enumerate}
    \item The $p$-th power of the Schatten p-norm $\sum_{i=1}^d \si_i(A)^p$, for $1 \le p < \infty$,
          induces a continuous map $W^{1,q}(U,M_{D,d}(\C)) \to W^{1,q}(U,\R)$.
      \item If $A \in W^{1,q}(U,\on{Herm}(d))$ is a family of density matrices (i.e. $\on{Tr}(A)=1$) such that
          $\inf_{x \in U} \min_{1 \le i \le d} \la_i(A(x))>0$, then the von Neumann entropy 
          $- \sum_{i=1}^d \la_i(A) \log \la_i(A)$ belongs to $W^{1,q}(U)$ and varies continuously in $A$.
      \item In the setup of (2), the R\'enyi entropy $\frac{1}{1-\al} \log \big(\sum_{i=1}^d \la_i(A)^\al\big)$, for $\al>1$, 
            belongs to $W^{1,q}(U)$ and varies continuously in $A$.
\end{enumerate}

%-----------------------------------------------------------------------------------------------------------------------
\subsubsection{Condition numbers}
%-----------------------------------------------------------------------------------------------------------------------

Let $A \in M_d(\C)$ and let $\si_1(A) \ge \cdots \ge \si_d(A)$ be the singular values of $A$. 
In numerical analysis, the \emph{condition number} of $A$ is defined by 
\begin{equation}
    \ka(A) := \frac{\si_1(A)}{\si_d(A)}.
\end{equation}
Considering the linear equation $Ax=b$ and assuming that $A$ is nonsingular (i.e., $\si_d(A)>0$), 
$\ka(A)$ describes the maximum ratio of the relative error in $x$ to the relative error in $b$. 
If $\ka(A)$ is much larger than $1$, then the problem is considered to be ill-conditioned.

\begin{theorem} \label[t]{t:cn}
    Let $1 \le q < \infty$. Let $I \subseteq \R$ be a bounded open interval.
    Suppose that $A_0 \in W^{1,q}(I, M_d(\C))$ and $\inf_{x \in I}\si_d(A_0(x)) > 0$.
    Then $\ka(A) \in W^{1,q}(I,\R)$ is well-defined in a neighborhood $\cU$ of $A_0$ in $W^{1,q}(I, M_d(\C))$ 
    and $\ka : \cU \to W^{1,q}(I,\R)$ is continuous. 
\end{theorem}

\Cref{t:cn} will be proved in \Cref{ssec:cn}.

%-----------------------------------------------------------------------------------------------------------------------
\subsubsection{Surface area}
%-----------------------------------------------------------------------------------------------------------------------

Let $U \subseteq \R^m$ be a bounded open set and $f \in W^{1,1}(U, \R)$. Then, by the area formula, the surface area 
of the graph of $f$ is given by 
\begin{equation}
    \on{Area}(f) := \cH^m(\ol f(U)), 
\end{equation}
where $\ol f : U \to U \times \R$, $x \mapsto (x,f(x))$ is the graph mapping and $\cH^m$ the 
$m$-dimensional Hausdorff measure; see \cite{MalySwansonZiemer03}.

\begin{theorem} \label[t]{t:surface}
   Let $U \subseteq \R^m$ be open and bounded.
   Suppose that $A_n \to A$ in $C^{0,1}(\ol U, \on{Herm}(d))$ as $n \to \infty$. 
Set $\la_j := \eig(A)_j$ and $\la_{n,j} := \eig(A_n)_j$, for $1 \le j \le d$.
   Then: 
   \begin{enumerate}
       \item For $1 \le j \le d$,
           \begin{equation}
               \on{Area}(\ol \la_{n,j}) \to \on{Area}(\ol \la_{j}) \quad \text{ as } n \to \infty.
           \end{equation}
       \item Let $Z := \bigcup_{j=1}^d \ol \la_j(U) \subseteq U \times \R$, resp.\ $Z_n := \bigcup_{j=1}^d \ol \la_{n,j}(U)$, be the total zero set
           of the characteristic polynomial of $A$, resp.\ $A_n$. Then 
           \begin{equation}
                    \liminf_{n \to \infty} \cH^m(Z_n) \ge \cH^m(Z).
           \end{equation}
   \end{enumerate}
\end{theorem}

\begin{proof}
    This follows from \Cref{thm:eigmap} and \Cref{thm:mptw} combined with the proof of 
    Corollary 7.7 and Corollary 7.8 in \cite{Parusinski:2024aa}. 
    (Use \eqref{eq:Frobenius} instead of Bronshtein's theorem.)
\end{proof}

%-----------------------------------------------------------------------------------------------------------------------
\subsubsection{Compact self-adjoint operators}
%-----------------------------------------------------------------------------------------------------------------------

The spectrum of a compact self-adjoint nonnegative operator $A$ in Hilbert space 
is a countable set consisting of eigenvalues of finite multiplicities which we order 
decreasingly
\[
    \la_1(A) \ge \la_2(A) \ge \cdots
\]
and possibly zero. The eigenvalues $\la_i(A)$ accumulate (only) at zero.

\begin{theorem} \label[t]{t:copintro}
    Let $U \subseteq \R^m$ be a bounded open set and $H$ a Hilbert space. 
    Suppose that $A, A_n \in C^{0,1}(\ol U,K(H))$, 
    for $n\ge 1$, are Lipschitz families of compact self-adjoint nonnegative operators (w.r.t.\ the operator norm on the space $K(H)$ of compact operators on $H$). 
    Assume that 
    $A(x)$ is positive definite for all $x \in U$ and 
    $A_n \to A$ in $C^{0,1}(\ol U,K(H))$ as $n\to \infty$.
    Then:
    \begin{enumerate}
        \item The decreasingly ordered eigenvalues $\la_i(A)$ and $\la_i(A_n)$, for all $i$ and large enough $n$, belong to $C^{0,1}(\ol U)$. 
        \item  For all $1 \le j \le m$ and almost every $x \in U$,
            \[
                \lim_{n \to \infty} \p_j (\la_i(A_n))(x) = \p_j (\la_i(A))(x), \quad i =1,2,\ldots 
            \]
        \item For every $1 \le q < \infty$, 
            \[
                \lim_{n \to \infty} \|\la_i(A) - \la_i(A_n)\|_{W^{1,q}(U)} = 0, \quad i= 1,2,\ldots
            \]
    \end{enumerate}
\end{theorem}

This theorem will be proved in \Cref{sec:cop}.

%---------------------------------------------------------------------------------------------
\subsection{Structure of the paper}
%---------------------------------------------------------------------------------------------

We separate the Hermitian and the normal case in a large part of the proofs, 
because the normal case requires more machinery for unordered eigenvalue maps and 
multivalued Sobolev functions. The reader only interested in the Hermitian case can safely 
skip Sections \ref{sec:dSob}--\ref{sec:normmult}.

In \Cref{sec:spaces}, we introduce the relevant function spaces and discuss some 
tools that are used frequently in the paper. 
In \Cref{sec:Hnm}, we recall matrix norms and results on the spectral variation of Hermitian and 
normal matrices. As a consequence, we introduce the characteristic map for Hermitian matrices.
\Cref{sec:BD} is devoted to the local uniform unitary block-diagonalization of Hermitian and normal matrices 
and to related pointwise bounds for the derivatives of absolutely continuous curves of such matrices.
In this section, we treat the Hermitian and normal case simultaneously.

The proofs of the main results in the Hermitian case are completed in \Cref{sec:H1} and \Cref{sec:proof2}. 
We treat the one-parameter cases in \Cref{sec:H1} and then deduce the multiparameter cases by sectioning arguments
in \Cref{sec:proof2}. In \Cref{sec:examples}, 
we discuss four examples that show optimality of our results.

For the treatment of the normal case, we provide in \Cref{sec:dSob} the necessary background on the metric space 
of unordered $d$-tuples of complex numbers and on Sobolev maps with values in this space.
This allows us to introduce the characteristic map for normal matrices in \Cref{sec:charn}. 
Then we complete the proofs of the main results for normal matrices in \Cref{sec:n1} (one-parameter case) 
and \Cref{sec:normmult} (multiparameter case). 

The final \Cref{sec:cop} is dedicated to an application of some of our results 
to compact self-adjoint operators in Hilbert space. 

%---------------------------------------------------------------------------------------------
\subsection{Notation} \label{ssec:notation}
%---------------------------------------------------------------------------------------------

The $m$-dimensional Lebesgue measure in $\R^m$ is denoted by $\cL^m$.  
If not stated otherwise, `measurable' means `Lebesgue measurable' and `almost everywhere' means 
`almost everywhere with respect to Lebesgue measure'. For measurable $E \subseteq \R^m$, we usually write 
$|E|=\cL^m(E)$.
We will also use the $k$-dimensional Hausdorff measure $\cH^k$.

For $1 \le p \le \infty$, $\|x\|_p$ denotes the $p$-norm of $x \in \R^d$. 
If $f : E \to \R^d$, for measurable $E \subseteq \R^m$, is a measurable map, then we set 
\[
    \|f\|_{L^p(E,\R^d)} := \big\|\|f\|_2 \big\|_{L^p(E)}. 
\]
In the following, a set is called \emph{countable} if it is either finite or has the cardinality of $\N$.

We consider the standard action of the symmetric group $\on{S}_d$ on $\C^d$ by permuting the coordinates:
for $\si \in \on{S}_d$ and $z \in \C^d$,
\[
    \si z = \si(z_1,\ldots,z_d) = (z_{\si(1)},\ldots,z_{\si(d)}).
\]
We write $X^{<2>} = \{(x,y) \in X\times X : x\ne y\}$ for the cartesian product with the diagonal removed.

We use the notation $C(d,\ldots)$ to denote a constant that depends only on $d,\ldots$; its value may change 
from line to line.

%---------------------------------------------------------------------------------------------
\section{Function spaces} \label{sec:spaces}
%---------------------------------------------------------------------------------------------

Let us fix notation and recall background on the function spaces used in this paper.

%---------------------------------------------------------------------------------------------
\subsection{Lebesgue spaces}
%---------------------------------------------------------------------------------------------

Let $U \subseteq \R^m$ be open and  $1 \le q \le \infty$.
We denote by $L^q(U)$ the Lebesgue space with respect to the $m$-dimensional Lebesgue measure $\cL^m$, 
and $\| \cdot \|_{L^q(U)}$ is the corresponding $L^q$-norm.
We will also use the space $L^q_{\on{loc}}(U)$ of measurable functions $f : U \to \R$ satisfying
$\|f\|_{L^q(K)} < \infty$ for all compact subsets $K \subseteq U$.
For Lebesgue measurable sets $E \subseteq \R^m$ we also write $|E| = \cL^m(E)$. 
We remark that for continuous functions $f : U \to \R$ we have (and use interchangeably)  
$\|f\|_{L^\infty(U)} = \|f\|_{C^0(\ol U)}$.

A map $f=(f_1,\ldots,f_n) : U \to \R^n$ is measurable if and only if each $f_i$ is measurable.
Now $f$ belongs to $L^q(U,\R^n)$ if $f_i \in L^q(U)$ for all $1 \le i \le n$, or equivalently
\[
    \|f\|_{L^q(U,\R^n)} := \big\| \|f\|_2 \big\|_{L^q(U)} < \infty.
\]

%---------------------------------------------------------------------------------------------
\subsection{Sobolev spaces}
%---------------------------------------------------------------------------------------------

For $1 \le q \le \infty$, 
we consider the Sobolev space 
\[
    W^{1,q}(U) := \{f \in L^q(U) : \p_i f \in L^q(U) \text{ for } 1 \le i \le m\},
\]
where $\p_i f$ are distributional derivatives, endowed with the norm 
\[
    \|f\|_{W^{1,q}(U)} := \|f\|_{L^q(U)} + \sum_{i=1}^m \|\p_i f\|_{L^q(U)}.
\]

A map $f=(f_1,\ldots,f_n) : U \to \R^n$ belongs to $W^{1,q}(U,\R^n)$ if $f_j \in W^{1,q}(U)$ for all $1 \le j \le m$, 
or equivalently,
\[
    \|f\|_{W^{1,q}(U,\R^n)} := \|f\|_{L^q(U,\R^n)} + \sum_{i=1}^m \|\p_i f\|_{L^q(U,\R^n)} < \infty.
\]

Let $I \subseteq \R$ be a bounded open interval. Let $1 \le q<\infty$.
A function $f: I \to \R$ belongs to $W^{1,q}(I)$ if and only if $f$ has an absolutely continuous 
representative whose classical derivative belongs to $L^q(I)$.
Moreover, $f \in W^{1,\infty}(I)$ if and only if $f$ has a Lipschitz continuous representative.

Let $U \subseteq \R^m$ be bounded and open. Let $1 \le q<\infty$.
A function $u \in L^q(U)$ belongs to $W^{1,q}(U)$ if and only if it has a representative 
that is absolutely continuous on $\cL^{m-1}$-almost all line segments in $U$ that are parallel 
to the coordinate axes and whose classical first-order partial derivatives belong to $L^q(U)$.
They coincide $\cL^m$-almost everywhere with the weak derivatives of $f$.
If $U$ is a Lipschitz domain, then $f \in W^{1,\infty}(U)$ if and only if $f$ has a Lipschitz continuous 
representative.

%-----------------------------------------------------------------------------------------------------------------------
\subsection{Lipschitz and H\"older spaces}
%-----------------------------------------------------------------------------------------------------------------------

Let $(M_i,\mathsf d_i)$, $i=1,2$, be metric spaces. A map $f : M_1 \to M_2$ is 
\emph{$\al$-H\"older} continuous, for $\al \in (0,1]$, if 
\[
    |f|_{C^{0,\al}(M_1,M_2)} := \sup_{x\ne y \in M_1} \frac{\mathsf d_2(f(x),f(y))}{\mathsf d_1(x,y)^\al} < \infty.
\]
In the case $\al=1$, $f$ is called \emph{Lipschitz} continuous and we sometimes also write $\on{Lip}(f)$ for the Lipschitz constant
$|f|_{C^{0,1}(M_1,M_2)}$.

Most often, $M_1$ will be a bounded open set $U \subseteq \R^m$ and $M_2 = \R^n$, both equipped with the metrics induced by the $2$-norm. 
Then, for $\al \in (0,1]$, we consider the Banach space $C^{0,\al}(\ol U,\R^n)$ consisting of continuous maps $f : U \to \R^n$ 
that extend continuously to $\ol U$ with the norm
\[
    \|f\|_{C^{0,\al}(\ol U,\R^n)} := \sup_{x \in U} \|f(x)\|_2 + |f|_{C^{0,\al}(\ol U,\R^n)} < \infty.
\]
We write $C^{0,\al}(U,\R^n)$ for the space of continuous functions on $U$ that 
belong to $C^{0,\al}(\ol V,\R^n)$ for each relatively compact open $V \Subset U$,
and endow $C^{0,\al}(U,\R^n)$ with its natural Fr\'echet topology.

%-----------------------------------------------------------------------------------------------------------------------
\subsection{Absolutely continuous curves in a metric space}  \label{ssec:ACq}
%-----------------------------------------------------------------------------------------------------------------------

Let $I \subseteq \R$ be a bounded open interval. 
Let $1 \le q \le \infty$. 
A curve $\ga : I \to X$ in a complete metric space $(X,\mathsf d)$ belongs to $AC^q(I,X)$ if there exists $m \in L^q(I)$ such that 
\begin{equation} \label{eq:metricspeed1}
    \mathsf d(\ga(x),\ga(y)) \le \int_x^y m(t)\, dt, \quad \text{ for all } x, y \in I,~ x \le y.
\end{equation}
In that case, the limit
\[
    \lim_{h \to 0} \frac{\mathsf d(\ga(x+h),\ga(x))}{|h|} =: |\dot \ga|(x)
\]
exists for almost every $x \in I$ and is called the \emph{metric speed} of $\ga$ at $x$.
Furthermore, $|\dot \ga| \in L^q(I)$ and \eqref{eq:metricspeed1} holds with $m$ replaced by $|\dot \ga|$; 
one has $|\dot \ga| \le m$ almost everywhere in $I$ for every $m$ that satisfies \eqref{eq:metricspeed1}.
See \cite[Definition 1.1.1]{Ambrosio:2008aa}.

For $1 \le q < \infty$,
the \emph{$q$-energy} $\sE_q(\ga)$ of $\ga \in AC^q(I,X)$ is defined by 
\begin{equation}
    \sE_q(\ga) := \int_I (|\dot \ga|(t))^q\, dt.
\end{equation}

We also write $AC$ for $AC^1$.

%-----------------------------------------------------------------------------------------------------------------------
\subsection{Vitali's convergence theorem}
%-----------------------------------------------------------------------------------------------------------------------

Let $(X,\cA,\mu)$ be a measure space with nonnegative measure $\mu$ (finite or with values in $[0,\infty]$).
A set of functions $\cF \subseteq L^1(\mu)$ is called \emph{uniformly integrable} if 
\[
    \lim_{C\to +\infty} \sup_{f \in \cF} \int_{|f|>C} |f|\, d\mu = 0.
\]

\begin{theorem}[De la Vall\'ee Poussin's criterion {\cite[Theorem 4.5.9]{Bogachev:2007aa}}] \label[t]{thm:VP}
    Let $\mu$ be a finite nonnegative measure. 
    A family $\cF$ of $\mu$-integrable functions is uniformly integrable if and only if 
    there exists a nonnegative increasing function $G$ on $[0,\infty)$ such that
    \[
        \lim_{t \to +\infty} \frac{G(t)}{t} = \infty 
        \quad \text{ and }
        \quad
        \sup_{f \in \cF} \int G(|f(x)|)\, \mu(dx) < \infty.
    \]
    In such a case, one can choose a convex increasing function $G$.
\end{theorem}

Recall that a sequence of complex valued measurable functions $f_n$ on $X$ is said to 
\emph{converge in measure} to $f$ 
if, for all $\ep >0$, 
\[
    \mu(\{x \in X : |f(x) - f_n(x)| \ge \ep\}) \to 0 \quad \text{ as } n \to \infty.
\]

\begin{theorem}[Vitali's convergence theorem {\cite[Theorem 4.5.4]{Bogachev:2007aa}}] \label[t]{thm:Vitali}
    Let $\mu$ be a finite nonnegative measure.
    Suppose that $f$ is a $\mu$-measurable function and $\{f_n\}$ is a sequence of $\mu$-integrable functions.
    Then the following assertions are equivalent:
    \begin{enumerate}
        \item $f_n \to f$ in measure and $\{f_n\}$ is uniformly integrable.
        \item $f$ is integrable and $f_n \to f$ in $L^1(\mu)$.
    \end{enumerate}
\end{theorem}

The next lemma is a simple application of Vitali's convergence theorem which we state here for later reference.

\begin{lemma} \label[l]{l:Vapp}
    Let $(X, \cA,\mu)$ be a finite measure space.
    Suppose $\{f_n\}$ is a sequence of nonnegative $\mu$-integrable functions
    and $\{X_i\}$ is a countable cover of $X$ by $\mu$-measurable sets.
    If $\{f_n\}$ is uniformly integrable and 
    \[
        \int_{X_i} f_n \, d\mu \to 0 \quad \text{ as } n \to \infty,
    \] 
    for each $i$,
    then 
    \[
        \int_{X} f_n \, d\mu \to 0 \quad \text{ as } n \to \infty.
    \]
\end{lemma}

\begin{proof}
    For each $i$, there is a subsequence $(n^i_k)$ of $(n)$ such that $f_{n^i_k} \to 0$ $\mu$-almost everywhere in $X_i$ as $k \to \infty$.
    By choosing the subsequences successively, we may assume that $(n^{i+1}_k)$ is a subsequence of $(n^i_k)$ for all $i$.
    Then, for $n_k := n_k^k$, we find that $f_{n_k} \to 0$ $\mu$-almost everywhere in $X$ as $k \to \infty$.

    Because $\{f_n\}$ is uniformly integrable, Vitali's convergence theorem \ref{thm:Vitali} implies 
    that  
    \[
        \int_{X} f_{n_k}\, d\mu \to 0 \quad \text{ as } k \to \infty;
    \]
    noting that on a finite measure space almost everywhere convergence implies convergence in measure, by Egorov's theorem. 
    This implies the assertion.
\end{proof}

%-----------------------------------------------------------------------------------------------------------------------
\subsection{Slope-convergence}
%-----------------------------------------------------------------------------------------------------------------------

The results presented here are used in the proof of \Cref{thm:mptw} in \Cref{ssec:mptw}.

\begin{definition}[Slope convergence]
Let $U \subseteq \R^m$ be a bounded open set and $U^{<2>} := \{(x,y) \in U\times U : x \ne y\}$.
For $f \in C^{0,1}(\ol U,\R^\ell)$ consider the \emph{slope function} $s_f : U^{<2>} \to \R$ defined by 
\[
    s_f(x,y) := \frac{\|f(x)-f(y)\|_2}{\|x - y\|_2}.
\]
Let $f,f_n \in C^{0,1}(\ol U,\R^\ell)$, for $n\ge 1$.
We say that $f_n$ 
\emph{slope-converges} to $f$ if 
\begin{align}
    \|s_f - s_{f_n}\|_{L^\infty(U^{<2>})} \to 0 \quad \text{ as } n \to \infty, \label{eq:uniform}
\end{align}
i.e., the slope functions converge uniformly.
\end{definition}

We have the following characterization of slope convergence.

\begin{lemma} \label[l]{l:slope1}
Let $f,f_n \in C^{0,1}(\ol U,\R^\ell)$, for $n\ge 1$.
Then $f_n$ slope-converges to $f$ if and only if,  
    for each subset $F \subseteq U^{<2>}$,  
    \begin{equation} \label{eq:LipE}
        \sup_F s_{f_n} \to \sup_F s_f \quad \text{ as } n \to \infty.
    \end{equation}
    In particular,
    for each subset $E \subseteq U$, we have convergence of the Lipschitz constants,
    \begin{equation} \label{eq:LipE2}
        |f_n|_{C^{0,1}(E,\R^\ell)} \to |f|_{C^{0,1}(E,\R^\ell)} \quad \text{ as } n \to \infty.
    \end{equation}
\end{lemma}

\begin{proof}
    For arbitrary $(x,y) \in U^{<2>}$, we have  
    \[
        s_{f}(x,y) \le s_{f_n}(x,y)  + \|s_f - s_{f_n}\|_{L^\infty(U^{<2>})}
    \]
    so that, for each subset $F \subseteq U^{<2>}$, 
     \[
         \sup_{(x,y) \in F}  s_{f}(x,y) \le \sup_{(x,y) \in F} s_{f_n}(x,y)  + \|s_f - s_{f_n}\|_{L^\infty(U^{<2>})}.
    \]
    Hence, by symmetry, 
     \[
         \Big|\sup_{(x,y) \in F}  s_{f}(x,y) - \sup_{(x,y) \in F} s_{f_n}(x,y)  \Big| \le \|s_f - s_{f_n}\|_{L^\infty(U^{<2>})}
    \]
    which shows that \eqref{eq:uniform} implies \eqref{eq:LipE}.

    Conversely, assume that \eqref{eq:uniform} fails. Then, there exist $\ep>0$ and sequences $n_k\to \infty$ and $(x_k,y_k) \in U^{<2>}$ such that 
    \[
        |s_f(x_k,y_k) - s_{f_{n_k}}(x_k,y_k)| \ge \ep.
    \]
    After passing to a subsequence of $k$, we may assume that 
    \[
        s_f(x_k,y_k) - s_{f_{n_k}}(x_k,y_k) \ge \ep \quad \text{ or } \quad s_{f_{n_k}}(x_k,y_k) - s_{f}(x_k,y_k) \ge \ep.
    \]
    In the former case,
    setting $F := \{(x_k,y_k) : k \ge 1\}$,
    \[
      \sup_{(x,y) \in F} s_f(x,y)      \ge \sup_{(x,y) \in F} s_{f_{n_k}}(x,y)+ \ep
    \]
    which contradicts \eqref{eq:LipE}. The latter case is analogous.
\end{proof}

The next lemma shows that $C^{0,1}$ convergence implies slope convergence. 

\begin{lemma}\label[l]{l:slope2}
Let $f,f_n \in C^{0,1}(\ol U,\R^\ell)$, for $n\ge 1$.
If 
\begin{equation}
    |f-f_n|_{C^{0,1}(\ol U,\R^\ell)} \to 0 \quad \text{ as } n \to \infty,
\end{equation}
then $f_n$ slope-converges to $f$. 
\end{lemma}

\begin{proof}
    We have
\begin{align*}
    \big|s_f(x,y) - s_{f_n}(x,y) \big| &= \Big| \frac{\|f(x)-f(y)\|_2}{\|x - y\|_2} - \frac{\|f_n(x)-f_n(y)\|_2}{\|x - y\|_2} \Big|
    \\
                                       &\le \frac{\| f(x)-f(y) - (f_n(x)-f_n(y))\|_2}{\|x-y\|_2} = s_{f-f_n}(x,y)
\end{align*}
and hence 
\begin{align*}
    \sup_{(x,y) \in U^{<2>}} \big|s_f(x,y) - s_{f_n}(x,y) \big| \le \sup_{(x,y) \in U^{<2>}} s_{f-f_n}(x,y) = |f-f_n|_{C^{0,1}(\ol U,\R^\ell)}
\end{align*}
which shows the assertion.
\end{proof}

%-----------------------------------------------------------------------------------------------------------------------
\section{Hermitian and normal matrices} \label{sec:Hnm}
%-----------------------------------------------------------------------------------------------------------------------

In this section, we recall some preliminaries on Hermitian and normal matrices.

%-----------------------------------------------------------------------------------------------------------------------
\subsection{Matrix norms} 
%-----------------------------------------------------------------------------------------------------------------------

Let $M_{d}(\C)$ denote the vector space of complex $d \times d$ matrices. The \emph{operator norm} of $A \in M_d(\C)$ is given by
\[
    \|A\|_{\on{op}} := \sup \{\|Ax\|_2 : x \in \C^d,\, \|x\|_2=1\}.
\]
We will also use the \emph{Frobenius norm} or \emph{$2$-norm} of $A = (a_{ij}) \in M_d(\C)$,
\[
\|A\|_2 := \big(\on{Tr}(A^*A)\big)^{1/2} = \Big(\sum_{i,j=1}^d |a_{ij}|^2\Big)^{1/2}. 
\]
If $\si_1(A) \ge \si_2(A) \ge \cdots \ge \si_d(A) \ge 0$ are the decreasingly ordered singular values of $A$, then 
\[
    \|A\|_{\on{op}} = \si_1(A) \quad \text{ and } \quad  \|A\|_2 = \Big(\sum_{j=1}^d \si_j(A)^2\Big)^{1/2}.
\]
Both operator norm and Frobenius norm are unitarily invariant (i.e., $\|A\|_{\on{op}} = \|UAV\|_{\on{op}}$ and $\|A\|_{2} = \|UAV\|_{2}$ 
for unitary $U,V$). It is easy to see that 
\[
    \|A\|_{\on{op}} \le \|A\|_2 \le \sqrt d \, \|A\|_{\on{op}}.
\]

If $A$ is normal and $\la_1,\ldots,\la_d$ are the eigenvalues of $A$, then
\[
    \|A\|_{\on{op}} = \max_{1\le j \le d} |\la_j| \quad \text{ and } \quad  \|A\|_2 = \Big(\sum_{j=1}^d |\la_j|^2\Big)^{1/2}.
\]

If not stated otherwise, we use the Frobenius norm on $M_d(\C)$ (and on its subsets).

%-----------------------------------------------------------------------------------------------------------------------
\subsection{Spectral variation of Hermitian and normal matrices}
%-----------------------------------------------------------------------------------------------------------------------

Recall that $\la^\uparrow : \on{Herm}(d) \to \R^d$ is the map assigning to a Hermitian $d \times d$ its
$d$-tuple of increasingly ordered eigenvalues; see \eqref{eq:Herm}.

\begin{proposition}[Weyl's perturbation theorem \cite{Weyl12}; see e.g.\ {\cite[III.2.6]{Bhatia97}}] \label[p]{p:Weyl}
    Let $A, B \in \on{Herm}(d)$. Then 
    \begin{equation} \label{eq:Weyl}
        \|\la^\uparrow(A)-\la^\uparrow(B)\|_\infty \le \|A-B\|_{\on{op}}.
   \end{equation}
\end{proposition}

In \cite{Bhatia97}, the result is stated for eigenvalue vectors with decreasing eigenvalues, but reversing the order evidently 
leaves the left-hand side of \eqref{eq:Weyl} unchanged.

Most of the time we will work with $2$-norms for which the following variant due to L\"owner {\cite{Lowner:1934aa} holds.

\begin{proposition}[{\cite{Lowner:1934aa}}, see e.g.\ {\cite[III.6.15]{Bhatia97}}] \label[p]{p:Loewner}
    Let $A, B \in \on{Herm}(d)$. Then 
    \begin{equation} \label{eq:Loewner}
       \|\la^\uparrow(A)-\la^\uparrow(B)\|_2 \le \|A-B\|_2.
   \end{equation}
\end{proposition}

L\"owner's result was generalized to normal matrices by Hoffmann and Wielandt \cite{Hoffman:1953aa}.
There is no canonical order on the spectrum of a normal matrix. Instead we consider the map 
$\La : \on{Norm}(d) \to \cA_d(\C)$ which takes a normal $d\times d$ matrix to its unordered $d$-tuple of eigenvalues; 
see \eqref{eq:La}. Recall that $\cA_d(\C)$ carries the metric $\dd_2([z],[w]) = \min_{\si \in \on{S}_d} \|z-\si w\|_2$, where 
$\si w = \si(w_1,\ldots,w_d) = (w_{\si(1)},\ldots,w_{\si(d)})$.

\begin{proposition}[{\cite{Hoffman:1953aa}, see e.g.\ {\cite[VI.4.1]{Bhatia97}} }] \label[p]{p:HW}
    Let $A,B \in \on{Norm}(d)$.
    Then 
    \begin{equation} \label{eq:HW}
        \dd_2(\La(A),\La(B)) 
        \le \|A-B\|_2.
    \end{equation}
\end{proposition}

If $A$ and $B$ are Hermitian, then \eqref{eq:HW} reduces to \eqref{eq:Loewner}; see \Cref{lem:unordered}. 

With respect to the operator norm, there is the following result due to Bhatia, Davis, and McIntosh \cite{BhatiaDavisMcIntosh83}, 
where $\dd_\infty([z],[w]) = \min_{\si \in \on{S}_d} \|z-\si w\|_\infty$.

\begin{proposition}[{\cite{BhatiaDavisMcIntosh83}, see e.g.\ {\cite[VII.4.1]{Bhatia97}} }]  \label[p]{p:BDM} 
    Let $A,B \in \on{Norm}(d)$.
    Then
    \begin{equation} \label{eq:BDM}
        \dd_\infty(\La(A),\La(B)) 
        \le C\, \|A-B\|_{\on{op}},
    \end{equation}
    where $C$ is a universal constant satisfying $1 <C<3$.
\end{proposition}

%-----------------------------------------------------------------------------------------------------------------------
\subsection{The characteristic map for Hermitian matrices}
%-----------------------------------------------------------------------------------------------------------------------

\begin{proposition} \label[p]{p:mcharHerm}
    Let $1 \le q \le \infty$. 
    Let $U \subseteq \R^m$ be open and bounded.
    Then the characteristic map
    \[
        \eig : W^{1,q}(U,\on{Herm}(d)) \to W^{1,q}(U,\R^d), \quad A \mapsto \la^\uparrow \o A
    \]
    is well-defined and bounded, satisfying 
    \begin{align}\label{eq:charHerm3}
        \|\eig(A)(x)\|_2 \le \|A(x)\|_2 \quad \text{ and } \quad 
        \|\nabla (\eig(A))(x)\|_2 \le  \|\nabla A(x)\|_2 
    \end{align}
    for almost every $x \in U$.
\end{proposition}

\begin{proof} 
    It is well-known (see \cite{AmbrosioDalMaso90}) that superposition with the $1$-Lipschitz map $\la^\uparrow$ 
    (see \Cref{p:Loewner})
    defines a bounded map from $W^{1,q}(U,\on{Herm}(d))$ to $W^{1,q}(U,\R^d)$, where 
    \[
        \|\la^\uparrow \o A\|_2 \le \|A\|_2, 
    \]
    because $\la^\uparrow(0) = 0$, and 
    \[
        \|\nabla (\la^\uparrow  \o A)\|_2 \le  \|\nabla A\|_2 
    \] 
    almost everywhere in $U$.
    This implies the assertion. 
\end{proof}

In \Cref{p:mcharnorm}, we will encounter a variant of \Cref{p:mcharHerm} for normal matrices; 
its formulation requires some background on multivalued Sobolev functions, which will be recalled in \Cref{sec:dSob}.

%-----------------------------------------------------------------------------------------------------------------------
\section{Uniform block-diagonalization of Hermitian and normal matrices} \label{sec:BD}
%-----------------------------------------------------------------------------------------------------------------------

%-----------------------------------------------------------------------------------------------------------------------
\subsection{Splitting lemma for complex triangular matrices}
%-----------------------------------------------------------------------------------------------------------------------

The following splitting lemma for complex triangular matrices follows from the proof of Hensel's lemma 
of {\cite[Lemma 2.1]{Parusinski:2020ac}} which is stated there in a different setup,  
namely for normal matrices over the ring of formal power series $\C\llbracket X \rrbracket$, 
where $X = (X_1,\ldots,X_m)$. %see \cite[Lemma 2.1]{Parusinski:2020ac}.  

We shall state this lemma in the Nash real algebraic setup because the unitary group $U_d(\C)$ is not 
a complex manifold but only a nonsingular real algebraic variety.  Recall that a map is a real Nash map 
if it is real analytic and its graph is semialgebraic.  Nash maps between open subsets of nonsingular real 
algebraic varieties form the smallest family of maps that contains the polynomial morphisms and admits the 
Implicit Function Theorem, see e.g.\  \cite{BCR} for more on Nash maps.

\begin{lemma}[Local unitary block-triangularization of matrices]\label[l]{lem:SplitMat}
    Let $A_0 \in M_d(\C)$ be a complex block-triangular  matrix,      
    \begin{align*}
        A_0 =   \begin{pmatrix}
            B_0 & 0 \\
            D_0 & C_0 
        \end{pmatrix}, 
    \end{align*}
    with $B_0 \in M_{d_1} (\C)$, $C_0 \in M_{d_2}(\C)$, $D_0 \in M_{d_2,d_1}(\C)$, for $d=d_1+d_2$, and such that  the characteristic polynomials of $B_0$ and 
    $C_0$ are coprime.  
    Then there is a neighborhood $\mathcal V$ of $A_0$ in $M_d(\C)$ and a Nash map to the unitary matrices,
    $$
    U : \mathcal V \to U_d (\C ), \quad  U(A_0) = \I_d,
    $$
    such that, for all $A\in \mathcal V$,
    \begin{align}\label{eq:equationforA}
        U^*(A)  A \, U(A) =   \begin{pmatrix}
            B(A) & 0 \\
            D(A) &  C(A) 
        \end{pmatrix}. 
    \end{align} 
\end{lemma}

\begin{proof}
    Consider  
    \begin{align*} 
        \Psi =(\Psi_1, \Psi_2, \Psi _3,\Psi_4)   : 
        U_d(\C ) \times  M_{d_1} (\C ) \times M_{d_2} (\C ) \times 
        M_{d_2,d_1} (\C ) 
        \to 
        M_{d} (\C )  , 
    \end{align*} 
    defined by 
    \begin{align*}
        (U ,  Y_1 , Y_2, Y_3)  \to    U^{*}  \begin{pmatrix}
            B_0 + Y_1 & 0 \\
            D_0 + Y_3 &  C_0 +Y_2 
        \end{pmatrix}  U 
        =  \begin{pmatrix}
            \Ps_1 & \Ps_4 \\
            \Ps_3 &   \Ps_2 
        \end{pmatrix} .
    \end{align*} 
    Recall that a tangent vector at $\I_d$ to $U_d(\C) $  is a matrix $\mathbf u$ that is skew-Hermitian, 
    $\mathbf u = - \mathbf u^*$.  We shall write it as 
    \begin{align*}
        \mathbf u  =   \begin{pmatrix}
            \mathbf z_1  & \mathbf x \\
            - \mathbf x^* &  \mathbf z_2
        \end{pmatrix},  
    \end{align*}
    where $\mathbf z_1$ and $\mathbf z_2$ are skew-Hermitian.
    The differential of $\Psi$ at $(\I_d,0,0,0)$ on the vector  $(\mathbf u , \mathbf  y_1 , \mathbf y_2, \mathbf y_3)$ 
    is given by 
    \begin{align*}
  & d\Psi_1  (\mathbf u , \mathbf  y_1 , \mathbf y_2, \mathbf y_3) = \mathbf y_1  - \mathbf z_1 B_0 + B_0 \mathbf z_1 -\mathbf x D_0, \\
  & d\Psi_2  (\mathbf u , \mathbf  y_1 , \mathbf y_2, \mathbf y_3) = \mathbf y_2  - \mathbf z_2 C_0 + C_0 \mathbf z_2 +D_0\mathbf x,  \\
  & d\Psi_3  (\mathbf u , \mathbf  y_1 , \mathbf y_2, \mathbf y_3)=  \mathbf y_3 + D_0 \mathbf z_1 -\mathbf z_2 D_0 + \mathbf x^* B_0 - C_0  \mathbf x^* , 
  \\
  & d\Psi_4  (\mathbf u , \mathbf  y_1 , \mathbf y_2, \mathbf y_3) =  B_0  \mathbf x - \mathbf x C_0. 
    \end{align*}
    This differential is a linear epimorphism if and only if so is $d\Psi_4$ and the latter holds true thanks 
    to the following lemma of \cite[Lemma 2.3]{cohn},  see also \cite{zurro} or \cite[Lemma 2.4]{Parusinski:2020ac}.

    \begin{lemma}\label{lem:Cohn}
        Let $R$ be an unitary commutative ring, $A \in  M_{p} (R)$, $B \in M_q (R)$, $C \in M_{p,q} (R)$, such that  
        the characteristic polynomials $P_A$ and $P_B$ are coprime, i.e., there exist polynomials $U$ and $V$ such that $UP_A+VP_B=1$. Then there is a  
        matrix $M\in M_{p,q}(R)$ such that $AM - MB = C$.
    \end{lemma} 

    By the Nash IFT, see e.g.  \cite[Corollary 2.9.8]{BCR}, 
    there exist a neighborhood $\mathcal V$ of $A_0$ in $M_d(\C)$ 
    and a Nash map 
    $$
    \Phi= (U, Y_1, Y_2, Y_3)  : \mathcal V \to U_d(\C ) \times  M_{d_1} (\C ) \times M_{d_2} (\C ) \times M_{d_2,d_1} (\C ), 
    $$
    $\Phi(A_0) = (\I_d,0,0,0)$, that satisfies $\Psi\circ \Phi (A)=A $ for $A\in \mathcal V$.  
    Thus it suffices to take $U(A) := \Phi_1^* (A)$. 
\end{proof}

\begin{corollary}[Local unitary block-diagonalization of normal matrices]\label[c]{cor:SplitNorMat}
    With the assumptions of \Cref{lem:SplitMat} if, moreover, 
    $A_0$ is normal (resp.\ Hermitian) then $A_0$ is block-diagonal, i.e., $D_0=0$, and 
    thus for all  normal (resp.\ Hermitian) $A\in \mathcal V$
    the matrix     
    \begin{align}\label{eq:equationforAnormal}
        U^*(A)  A \, U(A) =   \begin{pmatrix}
            B(A) & 0 \\
            0 &  C (A) 
        \end{pmatrix} 
    \end{align} 
    is normal (resp. Hermitian).
\end{corollary}

\begin{proof}
    It immediately follows from the facts that conjugation by unitary matrices preserves normal (resp.\ Hermitian) matrices, 
    and that every block-triangular normal matrix is block-diagonal.
\end{proof}

%-----------------------------------------------------------------------------------------------------------------------
\subsection{Uniform unitary block-diagonalization} \label{ssec:UBD}
%-----------------------------------------------------------------------------------------------------------------------

\begin{definition}[Spaces of Hermitian and normal matrices]
    Let $\on{Herm}_T(d)$ (resp.\ $\on{Norm}_T(d)$) denote the space of Hermitian (resp.\ normal) 
    $d\times d$ matrices with trace zero and $\on{Herm}_T^0(d)$ (resp.\ $\on{Norm}_T^0(d)$)
    its compact subspace of matrices of norm $1$, i.e.,
    \begin{align*}
        \on{Herm}_T(d) &= \{A \in \on{Herm}(d) : \on{Tr} (A)=0\},
        \\
        \on{Herm}_T^0(d) &= \{A \in \on{Herm}_T(d) : \|A\|_2=1\},
        \\
        \on{Norm}_T(d) &= \{A \in \on{Norm}(d) : \on{Tr} (A)=0\},
        \\
        \on{Norm}_T^0(d) &= \{A \in \on{Norm}_T(d) : \|A\|_2=1\}.
    \end{align*}
\end{definition}

Let $A_0\in \on{Norm}_T(d)\setminus \{0\}$. Then $A_0$ has at least two distinct eigenvalues. 
Therefore, 
after a unitary change of coordinates, we may assume that $A_0$ is block-diagonal,   
\begin{align}
    A_0 =   \begin{pmatrix}
        B_0 & 0 \\
        0 &  C_0 
    \end{pmatrix}, 
\end{align} 
where $B_0, C_0$ are normal and the resultant of the characteristic polynomials of $B_0$ and $C_0$ is nonzero.

If $A_0 \in \on{Herm}_T(d)\setminus \{0\}$, then $B_0, C_0$ are Hermitian.
In this case,
we may assume furthermore that the largest eigenvalue of $B_0$ is strictly smaller than the smallest eigenvalue of $C_0$.

By \Cref{cor:SplitNorMat}, there is an open neighborhood $\cV = \cV_{A_0}$ of $A_0$ in $M_d(\C)$ and a Nash map 
$U : \cV \to U_d(\C)$, with $U(A_0) = \mathbb I_d$, such that, for all $A \in \cV \cap \on{Norm}(d)$, 
\begin{align}\label{eq:split_matrix}
U^*(A) A \, U(A) =   \begin{pmatrix}
   B(A) & 0 \\
  0 &  C(A) 
  \end{pmatrix},
\end{align}
where $ B(A)$ and $ C(A)$ are normal and satisfy $B(A_0)= B_0$ and $C(A_0)= C_0$. 
Shrinking $\cV$ slightly, we may assume that the derivatives of all orders of the map $U : \cV \to U_d(\C)$ are bounded. 
In addition, we may assume that $\cV$ is convex.

Moreover, by shrinking $\cV$ further if necessary, we may assume that there exists $\ch =\ch_{A_0} >0$ such that, 
for all $A_1,A_2 \in \cV \cap \on{Norm}(d)$ and $B_j := B(A_j)$, $C_j := C(A_j)$, for $j=1,2$, given by \eqref{eq:split_matrix},
we have
\begin{equation} \label{eq:sep}
    |\mu_1 - \nu_2| > \ch  
\end{equation}
for all eigenvalues $\mu_1$ of $B_1$ and all eigenvalues $\nu_2$ of $C_2$.

If $A_0$ is Hermitian, then, for $A \in \cV \cap \on{Herm}(d)$, we have \eqref{eq:split_matrix} with Hermitian $B(A)$ and $C(A)$. 
Shrinking $\cV$ if necessary, 
we may assume that, for all $A \in \cV \cap \on{Herm}(d)$, the largest eigenvalue of $B(A)$ is strictly smaller than the smallest eigenvalue of $C(A)$.

The open neighborhoods $\cV_{A_0}$, for $A_0 \in \on{Norm}_T^0(d)$, form an open cover of the  
compact space $\on{Norm}_T^0(d)$. We fix a finite subcover $\{\cV_1,\ldots,\cV_s\}$. 
Then, for each $i=1,\ldots, s$, 
we have a Nash map $U_i : \cV_i  \to U_d(\C)$ with bounded derivatives on $\cV_i$ 
such that 
\begin{align}\label{eq:split_matrix2}
U_i^* (A) A \, U_i(A) =   \begin{pmatrix}
   B_i(A) & 0 \\
  0 &  C_i(A) 
\end{pmatrix}, \quad A \in \cV_i \cap \on{Norm}(d),
\end{align}
where $ B(A)$ and $ C(A)$ are normal.
By the Lebesgue covering lemma, there exists $\de>0$ such that each subset of 
$\on{Norm}_T^0(d)$ of diameter less than $\de$ is contained in some $\cV_i$.
Choose $r \in (0,\min\{\de/2,1\})$. 
Then for each $A_0 \in \on{Norm}_T^0(d)$ there exists $1 \le i \le s$ such that
\begin{equation} \label{eq:r}
    B(A_0,r) \cap \on{Norm}_T^0(d)  \subseteq \cV_i \cap \on{Norm}_T^0(d).
\end{equation}

Thanks to the property \eqref{eq:sep}, there exists 
$\ch >0$ such that, for all $i \in  \{1,\ldots,s\}$ and 
$A_1,A_2 \in \cV_i \cap \on{Norm}(d)$, 
we have
\begin{equation} \label{eq:sep2}
    |\mu_1 - \nu_2| > \ch  
\end{equation}
for all eigenvalues $\mu_1$ of $B_i(A_1)$ and all eigenvalues $\nu_2$ of $C_i(A_2)$, 
where $B_i(A_1)$ and $C_i(A_2)$ are defined by \eqref{eq:split_matrix2}.

Since $\on{Herm}_T^0(d) \subseteq \on{Norm}_T^0(d)$, also $\on{Herm}_T^0(d)$ is covered by $\{\cV_1,\ldots,\cV_s\}$  
and \eqref{eq:split_matrix2} 
holds with Hermitian $B_i(A), C_i(A)$ and such that 
the largest eigenvalue of $B_i(A)$ is strictly smaller than the smallest eigenvalue of $C_i(A)$, for all $A \in \cV_i \cap \on{Herm}(d)$.

\begin{definition}[Uniform unitary block-diagonalization] \label[d]{d:UBD}
    We say that the data 
    \[
        (\{U_i : \cV_i \to U_d(\C)\}_{i=1}^s,r,\ch)
    \]
    that we just fixed is 
    a \emph{uniform unitary block-diagonalization} for $\on{Norm}_T^0(d)$, respectively, for $\on{Herm}_T^0(d)$.
\end{definition}

%-----------------------------------------------------------------------------------------------------------------------
\subsection{Pointwise bounds for absolutely continuous curves of Hermitian and normal matrices}
%-----------------------------------------------------------------------------------------------------------------------

Let us fix a uniform unitary block-diagonalization $(\{U_i : \cV_i \to U_d(\C)\}_{i=1}^s,r,\ch)$ for $\on{Norm}_T^0(d)$
for the rest of the section.
Fix $i \in \{1,\ldots,s\}$ and let us simply write $U : \cV \to U_d(\C)$ for $U_i : \cV_i \to U_d(\C)$. 
Let $I\subseteq \R$ be a bounded open interval.
Then $U$ induces a map
\begin{align*}
    U_*    :     AC(I,\cV ) \to AC(I, U_d(\C)), \quad A \mapsto U\o A 
\end{align*}
and the chain rule holds;
see e.g.\ \cite{MarcusMizel72}.

In the following, we use the norm  
\begin{equation*} 
    \|f\|_{C^k} := \max_{0\le j \le k} \sup_{x \in \cV} \|d^j f(x)\|_{L_j(\R^m,\R^\ell)} 
\end{equation*}
on the space $C^{k}(\ol \cV,\R^\ell)$, where $\cV \subseteq \R^m$ is open and bounded 
and $L_j(\R^m,\R^\ell)$ is the space of $j$-linear maps with $j$ arguments in $\R^m$ and values in $\R^\ell$.

\begin{lemma}
    Let  $A_1,A_2 \in AC(I, \cV)$. Then, for every $x \in I$,
    \begin{align} \label{eq:pU1}
       \|U_*(A_1)(x) - U_*(A_2)(x)\|_2 \le \|U\|_{C^1} \|A_1(x)-A_2(x)\|_2
    \end{align}
    and, for almost every $x \in I$ and $j=1,2$,
    \begin{align} \label{eq:pU2}
       \MoveEqLeft \|(U_*(A_1))'(x) - (U_*(A_2))'(x)\|_2 \nonumber
       \\
       &\le \|U\|_{C^2} \big( \|A_j'(x)\|_2 \|A_1(x) - A_2(x)\|_2  + \|A_1'(x) -A_2'(x)\|_2\big).
    \end{align}
\end{lemma}

\begin{proof}
    The mean value theorem implies \eqref{eq:pU1}.
    We have 
    \begin{align*}
        &\|(U_*(A_1))'(x) - (U_*(A_2))'(x)\|_2  
        =\|dU(A_1(x))A_1'(x) - dU(A_2(x))A_2'(x)\|_2 
        \\
        &\quad \le \|(dU(A_1(x)) - dU(A_2(x)))A_1'(x)\|_2+\|dU(A_2(x))(A_1'(x) -A_2'(x))\|_2
        \\
        &\quad \le \|U\|_{C^2} \|A_1(x) - A_2(x)\|_2 \|A_1'(x)\|_2+\|U\|_{C^1} \|A_1'(x) -A_2'(x)\|_2
    \end{align*}
    which implies \eqref{eq:pU2} for $j=1$. The assertion for $j=2$ follows by symmetry.
\end{proof}

Let $A \in AC(I,\on{Norm}_T(d))$ and assume that $0 \not\in A(I)$. Then 
\[
    \ul A(x) := \frac{A(x)}{\|A(x)\|_2} \in \on{Norm}_T^0(d) 
\]
    is well-defined for $x \in I$.  
If $\ul A(I) \subseteq \cV_i$ for some $i \in \{1,\ldots,s\}$, then 
\begin{align*}
   U_i^* (\ul A)  \ul A \, U_i(\ul A) =   \begin{pmatrix}
   \ul B_i(\ul A) & 0 \\
  0 &  \ul C_i(\ul A) 
  \end{pmatrix}
\end{align*}
and consequently
\begin{align} \label{eq:BD}
     U_i^* (\ul A) A \, U_i(\ul A) =   \begin{pmatrix}
   B_i(A) & 0 \\
  0 &  C_i(A) 
  \end{pmatrix},
\end{align}
where $B_i(A) = \|A\|_2 \cdot \ul B_i(\ul A)$ and $C_i(A) = \|A\|_2 \cdot \ul C_i(\ul A)$ are normal.

By the property \eqref{eq:sep2}, there exists 
$\ch >0$ such that the following holds. 
Assume that $A_1,A_2 \in AC(I,\on{Norm}_T(d)\setminus \{0\})$ and $\ul A_1(I), \ul A_2(I) \subseteq \cV_i$ for some $i \in \{1,\ldots,s\}$.
Then, for all $x \in I$,
\begin{equation} \label{eq:sep3}
    \big|\|A_2(x)\|_2\, \mu_1(x) - \|A_1(x)\|_2\, \nu_2(x)\big| > \ch\, \|A_1(x)\|_2 \|A_2(x)\|_2
\end{equation}
for all eigenvalues $\mu_1$ of $B_i(A_1)$ and all eigenvalues $\nu_2$ of $C_i(A_2)$, 
where $B_i(A_1)$ and $C_i(A_2)$ are defined by \eqref{eq:BD}.

If $A \in AC(I,\on{Herm}_T(d))$, then $\ul A$, $B_i(A)$, and $C_i(A)$ are Hermitian
and the largest eigenvalue of $B_i(A)$ is strictly smaller than the smallest eigenvalue of $C_i(A)$, on $I$, i.e.,
\begin{equation} \label{eq:eigABC}
    \eig(A) = (\eig(B_i(A)),\eig(C_i(A))). 
\end{equation}

In \Cref{l:bdiag}, we derive pointwise bounds for the difference of the expression \eqref{eq:BD} 
for two absolutely continuous curves $A_1,A_2$ of normal matrices. 
In particular, they apply to Hermitian matrices.
We need the following preliminary lemma.

\begin{lemma} \label[l]{l:A'}
Let $A \in AC(I,M_d(\C))$ and assume that $0 \not\in A(I)$. 
Then, almost everywhere in $I$,  
\begin{align}
    \|A\|_2' &= \Re \on{Tr}(\ul A^*A'), \label{eq:TrH}
       \\
    \big|\|A\|_2' \big| &\le  \|A'\|_2, \label{eq:pU4}
       \\
    \|\ul A'\|_2 &\le 2\, \frac{\|A'\|_2}{\|A\|_2}. \label{eq:ulA'}
\end{align}
\end{lemma}

\begin{proof}
    We have
    \begin{align*}
        \|A\|_2' &= (\on{Tr}(A^*A)^{1/2})' = \frac{\on{Tr}(A^*A')+\on{Tr}((A^*)'A)}{2\|A\|_2} 
    \\
    &=\frac{\on{Tr}(A^*A')+\on{Tr}((A')^*A)}{2\|A\|_2}
    =\frac{\on{Tr}(A^*A')+\ol {\on{Tr}(A^*A')}}{2\|A\|_2}
    = \Re \on{Tr}(\ul A^*A').
    \end{align*}
    Consequently,
    \begin{align*}
        \big|\|A\|_2' \big| \le \|\ul A^*\|_2 \|A'\|_2 = \|A'\|_2.
    \end{align*}
    Finally, 
    \begin{align*}
        \|\ul A'\|_2 = \Big\|\frac{A' \|A\|_2- A \|A\|_2'}{\|A\|_2^2} \Big\|_2 \le \frac{\|A'\|_2}{\|A\|_2} +  \frac{\big|\|A\|_2'\big|}{\|A\|_2} \le 2\,\frac{\|A'\|_2}{\|A\|_2} 
    \end{align*}
    and the lemma is proved.
\end{proof}

\begin{lemma} \label[l]{l:bdiag}
For $j=1,2$, let $A_j \in AC(I,\on{Norm}_T(d))$ be such that $0 \not\in A_j(I)$. 
Suppose that $\ul A_1(I), \ul A_2(I) \subseteq \cV_i$ for some $i \in \{1,\ldots,s\}$. 
For simplicity, we write $U : \cV \to U_d(\C)$ for $U_i : \cV_i \to U_d(\C)$.
Then, pointwise on $I$, 
\begin{align}\label{eq:bdiagbounds2}
\| U^*(\ul A_1)  A_1 U(\ul A_1) - U^*(\ul A_2) A_2 U(\ul A_2) \|_2  
 \le   C\,  \|A_1-A_2\|_2 
\end{align}
and, almost everywhere on $I$ and for $j=1,2$,
\begin{align} \label{eq:bdiagbounds3}
    \| &(U^*(\ul A_1)  A_1 U(\ul A_1))'  - (U^*(\ul A_2) A_2 U(\ul A_2))'  \|_2 \nonumber \\ 
                  &\quad \le  C\,\Bigl ( { \|A'_1 - A'_2\|_2} +    \big(\| A'_1\|_2 +\| A'_2\|_2 \big) \|A_j\|_2^{-1} \|A_1 -A_2\|_2 \Bigr ),
\end{align}
where $C$ is a universal constant depending only on $d$.
\end{lemma}

\begin{proof}
We have, for $j=1,2$,  
\begin{equation} \label{eq:pU3}
    \|\ul A_1 - \ul A_2\|_2 \le 2\, \|A_j\|_2^{-1} \|A_1 - A_2\|_2,  
\end{equation}
since
\begin{align*}
    \|\ul A_1 - \ul A_2\|_2  &=  \Big\|\frac {A_1}{\|A_1\|_2} - \frac {A_2}{\|A_2\|_2}  \Big\|_2 
\\
                             &\le   \frac {\| A_1 - A_2\|_2 \|A_2\|_2  +\|A_2\|_2 \, \big|\|A_1\|_2- \|A_2\|_2\big|}{\|A_1\|_2 \|A_2\|_2}  
                             \le 2\, \frac {\| A_1 - A_2\|_2 }{\|A_1\|_2}     
\end{align*}
and the case $j=2$ follows by symmetry.
Thus, by \eqref{eq:pU1} and  \eqref{eq:pU3},
\begin{align*}
\MoveEqLeft \| U^*(\ul A_1)   A_1 U(\ul A_1) - U^*(\ul A_2) A_2 U(\ul A_2) \|_2 
\\
                             &\le 
\| ( U^*(\ul A_1)- U^*(\ul A_2)) A_1\, U(\ul A_1) \|_2   
  + \| U^* (\ul A_2) A_1 (U(\ul A_1) - U(\ul A_2) \|_2
\\
                             &\quad + \| U^*(\ul A_2) (A_1 -A_2) U(\ul A_2) \|_2  \\
  &\lesssim \|A_1\|_2  \|\ul A_1 - \ul A_2\|_2  + \|A_1-A_2\|_2     \lesssim    \|A_1-A_2\|_2 
\end{align*}
showing \eqref{eq:bdiagbounds2}. 

Let us prove \eqref{eq:bdiagbounds3}.
By \eqref{eq:TrH},
\begin{align*}
        \big|\|A_1\|_2' - \|A_2\|_2' \big| &= \big|\Re \on{Tr}(\ul A_1^*A_1') - \Re \on{Tr}(\ul A_2^*A_2') \big|
        \\
                                           &= \big|\Re \on{Tr}((\ul A_1^* - \ul A_2^*)A_1') + \Re \on{Tr} (\ul A_2^* (A_1'-A_2'))\Big|
                                           \\
                                           &\le \|\ul A_1^* - \ul A_2^*\|_2 \|A_1'\|_2  +  \|\ul A_2^*\|_2 \|A_1'-A_2'\|_2
                                           \\
                                           &=\|\ul A_1 - \ul A_2\|_2 \|A_1'\|_2  +  \|A_1'-A_2'\|_2,
    \end{align*}
    so that by symmetry, for $j=1,2$,
    \[
        \big|\|A_1\|_2' - \|A_2\|_2' \big| \le \|\ul A_1 - \ul A_2\|_2 \|A_j'\|_2  +  \|A_1'-A_2'\|_2. 
    \]
Using \eqref{eq:pU3}, we conclude, for $j=1,2$,
\begin{equation}\label{eq:pU5}
    \big|\|A_1\|_2'- \|A_2\|_2'\big| \lesssim \frac{\|A_j'\|_2}{\|A_j\|_2} \|A_1 - A_2\|_2 + \|A_1'-A_2'\|_2.
\end{equation}

We claim that  
\begin{equation} \label{eq:pU6}
    \|\ul A'_1- \ul A'_2\|_2   \lesssim  \frac {\| A_2'\|_2   \|A_1- A_2\|_2 }{\|A_1\|_2 \|A_2\|_2 }   + \frac { \|A'_1 - A'_2\|_2}{\|A_1\|_2}.
\end{equation}
Indeed
\begin{align*}
    \|\ul A'_1- \ul A'_2\|_2 &= \Big \| \frac {A'_1 \|A_1\|_2 - A_1 \|A_1\|_2'}{\|A_1\|_2^2} - \frac {A'_2 \|A_2\|_2 - A_2 \|A_2\|_2'}{\|A_2\|_2^2} \Big \|_2 
    \\
                             &\le \Big \| \frac {A'_1 \|A_1\|_2 \|A_2\|_2^2 - A'_2 \|A_1\|_2^2 \|A_2\|_2}{\|A_1\|_2^2 \|A_2\|_2^2}\Big \|_2 
                             \\
                             &\quad +\Big \| \frac {A_2 \|A_1\|_2^2 \|A_2\|_2' - A_1 \|A_2\|_2^2 \|A_1\|_2'}{ \|A_1\|_2^2\|A_2\|_2^2} \Big \|_2. 
\end{align*}
The first summand equals
\begin{align*}
    \Big \| \frac {A'_1 \|A_2\|_2 - A'_2 \|A_1\|_2}{\|A_1\|_2 \|A_2\|_2}\Big \|_2  
    &= \Big \| \frac {(A'_1-A'_2) \|A_2\|_2  - A'_2 (\|A_1\|_2- \|A_2\|_2) }{\|A_1\|_2 \|A_2\|_2}\Big \|_2 \\ 
    & \lesssim    \frac { \|A'_1 - A'_2\|_2}{\|A_1\|_2} +   \frac {\| A'_2\|_2 \|A_1- A_2\|_2 }{\|A_1\| \|A_2\|_2 }  
\end{align*}
and, by \eqref{eq:pU4} and \eqref{eq:pU5},
the second summand is bounded by
\begin{align*}
    \Big \|& \frac { \|A_2\|_2' (A_2 \|A_1\|_2^2 - A_1 \|A_2\|_2^2 ) }{\|A_1\|_2^2 \|A_2\|_2^2} \Big \|_2 
    + \Big \| \frac {A_1 \|A_2\|_2^2 (\|A_2\|_2' - \|A_1\|_2' ) }{\|A_1\|_2^2 \|A_2\|_2^2}\Big \|_2 \\ 
    \\
           &= \frac{\big|\|A_2\|_2' \big| \, \big\| A_2 \|A_1\|_2^2 - A_2 \|A_1\|_2 \|A_2\|_2 + A_2 \|A_1\|_2 \|A_2\|_2 - A_1 \|A_2\|_2^2\big\|_2}{\|A_1\|_2^2 \|A_2\|_2^2} 
           \\
           &\quad +  \frac{\big| \|A_1\|_2' - \|A_2\|_2'\big|}{\|A_1\|_2}
           \\
           &\lesssim \frac{\|A_2'\|_2 \, \big| \|A_1\|_2 - \|A_2\|_2\big|}{\|A_1\|_2\|A_2\|_2} 
           + \frac{\|A_2'\|_2 \, \big\| A_2 \|A_1\|_2 - A_1 \|A_2\|_2\big\|_2}{\|A_1\|_2^2\|A_2\|_2}
           \\
           &\quad + \frac {\| A_2'\|_2   \|A_1- A_2\|_2 }{\|A_1\|_2 \|A_2\|_2 }   + \frac { \|A'_1 - A'_2\|_2}{\|A_1\|_2}
           \\
           %& \lesssim     \frac {\| A_2'\|_2  \|A_1\|_2 \|A_2\|_2 \big|\|A_1\|_2- \|A_2\|_2 \big| }{\|A_2\|_2^2 \|A_1\|_2^2 }  
           %+   \frac { \|A_2'\|_2 \|A_2\|_2 \|A_2 \|A_1\|_2 - A_1 \|A_2\|_2 \|_2 }{\|A_1\|_2^2 \|A_2\|_2^2} 
           %\\
           %&\quad + \frac { \|A'_1 - A'_2\|_2}{\|A_1\|_2} + \frac{\|A_2'\|_2\|A_1-A_2\|_2}{\|A_1\|_2 \|A_2\|_2} 
           %\\ 
            & \lesssim  \frac {\| A_2'\|_2   \|A_1- A_2\|_2 }{\|A_1\|_2 \|A_2\|_2 }   + \frac { \|A'_1 - A'_2\|_2}{\|A_1\|_2} 
       \end{align*}
       and \eqref{eq:pU6} follows.

       By \eqref{eq:pU2}, \eqref{eq:ulA'}, \eqref{eq:pU3}, and \eqref{eq:pU6},
       \begin{align}
           \| (  U (\ul A_1) - U (\ul A_2) )'\|_2 &\lesssim \|\ul A_2'\|_2 \|\ul A_1 - \ul A_2\|_2 + \|\ul A_1' - \ul A_2'\|_2
           \\ \label{eq:pU7}
                                                  &\lesssim \frac{\|A_2'\|_2}{\|A_2\|_2} \frac{\|A_1-A_2\|_2}{\|A_1\|_2}+ \frac { \|A'_1 - A'_2\|_2}{\|A_1\|_2},
       \end{align}
       similarly for $U^*$.

Finally we want to bound  
\begin{align*}
\MoveEqLeft \| (U^*(\ul A_1)  A_1 U(\ul A_1))' - (U^*(\ul A_2) A_2U(\ul A_2))'  \|_2  
\\
                             &\le \| ( ( U^* (\ul A_1) - U^* (\ul A_2) ) A_1 U(\ul A_1) )' \|_2 +  \| ( U^* (\ul A_2) A_1 (U(\ul A_1) - U(\ul A_2)))'  \|_2
 \\
                             &\quad + \| ( U^* (\ul A_2) (A_1 -A_2) U(\ul A_2))' \|_2.
\end{align*}  
The first summand can be bounded by
\begin{align*}
\MoveEqLeft \| (  U^* (\ul A_1) - U^* (\ul A_2) )'\|_2 \|A_1\|_2 
 + \| U^* (\ul A_1) - U^* (\ul A_2) \|_2 \| A'_1\|_2 
 \\ 
 &\quad + \|  U^* (\ul A_1) - U^* (\ul A_2) \|_2  \|A_1\|_2 \| (U(\ul A_1))'  \|_2 
 \\ 
 & \lesssim  \frac{(\|A_1'\|_2+ \|A_2'\|_2)}{\|A_2\|_2} \|A_1-A_2\|_2+  \|A'_1 - A'_2\|_2,
\end{align*}
using \eqref{eq:pU1}, \eqref{eq:ulA'}, \eqref{eq:pU3}, and \eqref{eq:pU7}.  
The bound for the second summand is similar and the third summand can be bounded by 
\begin{align*}
\MoveEqLeft \| ( U^* (\ul A_2))'\|_2 \|A_1 -A_2\|_2  + \| A_1' -A_2' \|_2
 + \|A_1 -A_2\|_2 \|U(\ul A_2))' \|_2  
\\ 
  & \lesssim  \frac{\|A_2'\|_2}{\|A_2\|_2} \|A_1-A_2\|_2+  \|A'_1 - A'_2\|_2.
\end{align*}
Thus \eqref{eq:bdiagbounds3} follows.
\end{proof}

\begin{remark}
   In the proof we did not use the fact that the matrices are normal. 
\end{remark}

\begin{remark} \label[r]{r:bdiag}
    \Cref{l:bdiag} remains valid if 
    $\ul A_1(I) \subseteq \cV_{i_1}$,
    $\ul A_2(I) \subseteq \cV_{i_2}$, for $i_1\ne i_2 \in \{1,\ldots,s\}$, 
    and $\|\ul A_1(x) - \ul A_2(x)\|_2 \ge \ep$, for some universal constant $\ep>0$, in the following sense:
    pointwise on $I$, 
    \begin{align*}
        \| U_{i_1}^*(\ul A_1)  A_1 U_{i_1}(\ul A_1) - U_{i_2}^*(\ul A_2) A_2U_{i_2}(\ul A_2) \|_2  
        \le   C\,  \|A_1-A_2\|_2 
    \end{align*}
    and, almost everywhere on $I$ and for $j=1,2$,
    \begin{align*}
        \| (U_{i_1}^*(\ul A_1) & A_1 U_{i_1}(\ul A_1))'  - (U_{i_2}^*(\ul A_2) A_2U_{i_2}(\ul A_2))'  \|_2 \nonumber \\ 
                               & \qquad \le  C\Bigl ( { \|A'_2 - A'_1\|_2} +  \big(\| A'_1\|_2 + \| A'_2\|_2 \big)\|A_j\|_2^{-1} \|A_1 -A_2\|_2 \Bigr ),
    \end{align*}
    where $C$ is a universal constant depending only on $d$.

    Indeed, 
    \begin{align*}
        \|U_{i_1}(\ul A_1)- U_{i_2}(\ul A_2)\|_2 
    &\le \|U_{i_1}(\ul A_1)\|_2+ \| U_{i_2}(\ul A_2)\|_2
    \lesssim \ep \le \|\ul A_1 - \ul A_2\|_2
    \end{align*}
    and 
    \begin{align*}
        \|(U_{i_1}(\ul A_1)- U_{i_2}(\ul A_2))'\|_2 
    &\le \|(U_{i_1}(\ul A_1))'\|_2+ \| (U_{i_2}(\ul A_2))'\|_2 
    \\
    &\lesssim (\|\ul A_1'\|_2+\|\ul A_2'\|_2) \ep \le (\|\ul A_1'\|_2+\|\ul A_2'\|_2)\|\ul A_1 - \ul A_2\|_2.
    \end{align*}
    Note that the sizes of the blocks in \eqref{eq:BD} corresponding to $i_1$ and $i_2$ may be different. 
\end{remark}

%-----------------------------------------------------------------------------------------------------------------------
\section{Eigenvalue stability for Hermitian matrices: one-parameter case} \label{sec:H1}
%-----------------------------------------------------------------------------------------------------------------------

This section is devoted to the proofs of \Cref{thm:eigW} and \Cref{thm:mptw} in 
the one-parameter case. In addition, we prove \Cref{t:cn} in \Cref{ssec:cn}. 

The following theorem is a version of \Cref{thm:eigW} for $m=1$.

\begin{theorem} \label[t]{thm:mainhermW}
    Let $1 \le q < \infty$.
    Let $I \subseteq \R$ be a bounded open interval.
    Let $A_n \to A$ in $W^{1,q} (I, \on{Herm}(d))$,
    i.e.,  
    \begin{equation} \label{eq:HermassW} 
        \| A- A_n\|_{W^{1,q}(I,M_d(\C))} \to 0 \quad \text{ as } n \to \infty.
    \end{equation}
    Then 
    \begin{equation} \label{eq:HermconclW}
        \|\eig(A) -  \eig(A_n)\|_{W^{1,q}(I,\R^d)}  \to 0 \quad \text{ as } n\to \infty. 
    \end{equation}
\end{theorem}

The proof of \Cref{thm:mainhermW} comprises \Cref{ssec:redH} and \Cref{ssec:pfH}.

%-----------------------------------------------------------------------------------------------------------------------
\subsection{Preliminary observations and reductions} \label{ssec:redH}
%-----------------------------------------------------------------------------------------------------------------------

\begin{lemma} \label[l]{l:mat}
    Let $A \in W^{1,q}(I,M_d(\C))$, where $1 \le q < \infty$ and
    $I \subseteq \R$ is a bounded open interval. Then, for all $x \in I$,
    \begin{align} \label{eq:supW}
     \|A(x)\|_2 &\le    |I|^{-1/q}\, \|A\|_{L^q(I,M_d(\C))} + |I|^{1-1/q} \, \|A'\|_{L^q(I,M_d(\C))}.
    \end{align}
\end{lemma}

\begin{proof}
For $x,y \in I$, we have
\[
    A(x) - A(y) = \int_y^x A'(t)\, dt
\]
and integrating over $y$, we get 
\begin{align*}
    A(x) - A_I = \frac{1}{|I|}\int_I \int_y^x A'(t)\, dt \,dy,
\end{align*}
where $A_I := |I|^{-1} \int_I A(y)\, dy$. Therefore,
\[
    \|A(x) - A_I\|_2 \le \frac{1}{|I|}\int_I \int_I \|A'(t)\|_2\, dt \,dy = \|A'\|_{L^1(I,M_d(\C))}.
\]
Consequently,
\begin{align*} \nonumber
    \|A(x)\|_2 &\le \|A_I\|_2 + \|A(x) - A_I\|_2   
             \\ \nonumber
               &\le |I|^{-1}\, \|A\|_{L^1(I,M_d(\C))} + \|A'\|_{L^1(I,M_d(\C))} 
            \\ 
              &\le |I|^{-1/q}\, \|A\|_{L^q(I,M_d(\C))} + |I|^{1-1/q} \, \|A'\|_{L^q(I,M_d(\C))}, 
\end{align*}
by H\"older's inequality.
\end{proof}

\begin{corollary} \label[c]{c:supW}
    Let $1 \le q < \infty$.
    Let $I \subseteq \R$ be a bounded open interval.
    Let $A_n \to A$ in $W^{1,q} (I, M_d(\C))$ as $n \to \infty$.
    Then 
    \begin{align} \label{eq:supW1}
        \|A - A_n\|_{L^\infty(I,M_d(\C))} &\to 0 \quad \text{ as } n \to \infty.
    \end{align}
\end{corollary}

\begin{proof}
    Apply \eqref{eq:supW} to $A-A_n$ instead of $A$.
\end{proof}

\begin{corollary} \label[c]{c:supeigW}
    Let $1 \le q < \infty$.
    Let $I \subseteq \R$ be a bounded open interval.
    Let $A_n \to A$ in $W^{1,q} (I, \on{Herm}(d))$ as $n \to \infty$.
    Then 
    \begin{align} \label{eq:supW2}
        \|\eig(A) - \eig(A_n)\|_{L^\infty(I,\R^d)} &\to 0 \quad \text{ as } n \to \infty.
    \end{align}
\end{corollary}

\begin{proof}
    This follows from
    \eqref{eq:Loewner} and \Cref{c:supW}. 
\end{proof}

To complete the proof of \Cref{thm:mainhermW} it suffices, by \Cref{c:supeigW}, to show that 
\begin{equation*} 
    \|\eig(A)' - \eig(A_n)'\|_{L^q(I,\R^d)} \to 0 \quad \text{ as } n \to \infty.
\end{equation*}

With $A \in W^{1,q} (I, \on{Herm}(d))$ we associate $\tilde A:=A-\frac{1}{d}\on{Tr}(A)\I_d \in W^{1,q} (I, \on{Herm}_T(d))$. 
Note that 
\begin{equation} \label{eq:Tr}
    \eig(A) - \eig(\tilde A) = \frac{1}{d}\on{Tr}(A)(1,1,\ldots,1).
\end{equation}

\begin{lemma} \label[l]{l:TrW}
    Let $1 \le q < \infty$.
    Let $I \subseteq \R$ be a bounded open interval.  
    Assume that $A_n \to A$ in $W^{1,q} (I, \on{Herm}(d))$ as $n \to \infty$.
    Then:
    \begin{enumerate}
        \item $\tilde A_n \to \tilde A$ in $W^{1,q} (I, \on{Herm}_T(d))$ as $n \to \infty$.
        \item Moreover, 
            \[
                \|\eig(A) -  \eig(A_n)\|_{W^{1,q}(I,\R^d)}  \to 0 
            \]
            if and only if 
            \[
                \|\eig(\tilde A) -  \eig(\tilde A_n)\|_{W^{1,q}(I,\R^d)}  \to 0 
            \]
            as  $n\to \infty$. 
    \end{enumerate}
\end{lemma}

\begin{proof}
    Clearly, $\on{Tr}(A_n) \to \on{Tr}(A)$ in $W^{1,q}(I)$ as $n \to \infty$. 
    The second assertion follows easily from \eqref{eq:Tr}.
\end{proof}

By \Cref{l:TrW}, we may assume that $\on{Tr}(A) = \on{Tr}(A_n) =0$ for all $n \ge 1$.

%-----------------------------------------------------------------------------------------------------------------------
\subsection{Proof of \Cref{thm:mainhermW}} \label{ssec:pfH}
%-----------------------------------------------------------------------------------------------------------------------

By the preliminary reductions in \Cref{ssec:redH}, it suffices to show the following proposition:

\begin{proposition} \label[p]{p:indhermW}
    Let $1 \le q < \infty$.
    Let $I \subseteq \R$ be a bounded open interval.
    Let $A_n \to A$ in $W^{1,q} (I, \on{Herm}_T(d))$ as $n \to \infty$.
    Then 
    \begin{equation} \label{eq:toshow}
        \|\eig(A)' - \eig(A_n)'\|_{L^q(I,\R^d)} \to 0 \quad \text{ as } n \to \infty.
    \end{equation}
\end{proposition}

We will proceed by induction on $d$. If $d=1$ the assertion is trivially true. 
So assume that $d\ge 2$.

Let us first consider the zero set 
\[
    Z_A := \{x \in I : A(x) =0\} = \{x \in I : \eig(A)(x) = 0\} =: Z_{\eig(A)}
\]
of $A$. 
Since $\eig(A) \in W^{1,q}(I,\R^d)$, the derivative $\eig(A)'$ exists almost everywhere in $I$. 
If $x_0$ belongs to the set of accumulation points of $Z_A$, denoted by $\on{acc}(Z_A)$,
and $\eig(A)'(x_0)$ exists, 
then necessarily 
$\eig(A)'(x_0)=0$.

\begin{lemma} \label[l]{l:Z}
    Let $1 \le q < \infty$.
    Let $I \subseteq \R$ be a bounded open interval.     
    Assume that $A_n \to A$ in $W^{1,q} (I, \on{Herm}(d))$ as $n \to \infty$.
    Then
    \begin{equation*}
        \|\eig(A_n)' - \eig(A)'\|_{L^q(Z_A,\R^d)} \to 0 \quad \text{ as } n\to \infty.
    \end{equation*}
\end{lemma}

\begin{proof}
    Let $E$ be the set of $x \in \on{acc}(Z_A)$, where the derivatives $A'(x)$, $\eig(A)'(x)$ and $A_n'(x)$, $\eig(A_n)'(x)$ for all $n\ge 1$ exist.
    For each $x \in E$, $A'(x) = 0$, $\eig(A)'(x) = 0$, and 
    \[ 
        \big\| \eig(A_n)'(x) \big\|_2  \le \|A_n'(x)\|_2,
    \]
    by \eqref{eq:Loewner} or \eqref{eq:charHerm3}.
    Since $E$ has full measure in $Z_A$, we thus get 
    \begin{align*}
        \|\eig(A_n)' - \eig(A)'\|_{L^q(Z_A,\R^d)} &= \|\eig(A_n)' \|_{L^q(Z_A,\R^d)} 
        \\
                                                  &\le \|A_n' \|_{L^q(Z_A,M_d(\C))}
                                                  = \|A_n'- A' \|_{L^q(Z_A,M_d(\C))},
    \end{align*}
    which implies the assertion.
\end{proof}

The next lemma takes care of the complement of $Z_A$, where around each point we
can find an interval on which we have a simultaneous unitary block-diagonalization for sufficiently large $n$.

Let us fix a uniform unitary block-diagonalization 
$(\{U_i : \cV_i \to U_d(\C)\}_{i=1}^s, r, \ch)$
for $\on{Herm}_T^0(d)$; see \Cref{d:UBD}.

\begin{lemma} \label[l]{l:red}
    Let $1 \le q < \infty$.
    Let $I \subseteq \R$ be a bounded open interval.     
    Assume that $A_n \to A$ in $W^{1,q} (I, \on{Herm}_T(d))$ as $n \to \infty$.
    Let $x_0 \in I\setminus Z_A$. 
    Then there exist an open interval $J$ with $x_0 \in J \subseteq I \setminus Z_A$, 
    $n_0 \ge 1$, and $i \in \{1,\ldots,s\}$ such that 
    \begin{enumerate}
        \item for all $n \ge n_0$, the curves
            $\ul A := A/\|A\|_2 $ and $\ul A_n := A_n/\|A_n\|_2$ belong to $W^{1,q}(J, \on{Herm}_T^0(d))$ 
            and satisfy
            $\ul A(J), \ul A_n(J) \subseteq \cV_i$; 
        \item on $J$ and for all $n \ge n_0$, we have  
            \begin{align} \label{eq:simBD}
                U_i^* (\ul A) A \, U_i(\ul A) 
                =   
                \begin{pmatrix}
                    B & 0 \\
                    0 &  C
                \end{pmatrix}, \quad
                U_i^* (\ul A_n) A_n \, U_i(\ul A_n) 
                =   
                \begin{pmatrix}
                    B_n & 0 \\
                    0 &  C_n
                \end{pmatrix},
            \end{align}
            where $B,B_n \in W^{1,q}(J,\on{Herm}(d_1))$ and $C,C_n \in W^{1,q}(J,\on{Herm}(d_2))$ with $d_1+d_2 = d$;
        \item on $J$ and for all $n \ge n_0$, we have 
            \begin{equation} \label{eq:eigBD}
    \eig(A) = (\eig(B),\eig(C)), \quad  \eig(A_n) = (\eig(B_n),\eig(C_n));
\end{equation}
        \item we have 
            \begin{equation} \label{eq:convBD}
                \|B - B_n\|_{W^{1,q}(J,M_{d_1}(\C))} \to 0, \quad 
                \|C - C_n\|_{W^{1,q}(J,M_{d_2}(\C))} \to 0 
            \end{equation}
            as $n \to \infty$.  
    \end{enumerate}
\end{lemma}

\begin{proof}
    Let $x_0 \in I\setminus Z_A$. Since $A_n \to A$ in $W^{1,q} (I, \on{Herm}(d))$ as $n \to \infty$, 
    there exists $n_0 \ge 1$ such that 
    \begin{equation} \label{eq:L1clW}
        \|A'-A_n'\|_{L^1(I,M_d(\C))} \le \frac{r}{32}\|A(x_0)\|_2 \quad \text{ if } n \ge n_0.
    \end{equation}
    By \eqref{eq:supW1}, 
    there exists $n_1\ge n_0$ such that 
    \begin{equation} \label{eq:clW}
        \|A(x_0)-A_n(x_0)\|_2 \le \frac{r}{8}\|A(x_0)\|_2 \quad \text{ if } n \ge n_1,
    \end{equation}
    and hence, in particular, 
    \begin{equation} \label{eq:cl1W}
        \frac{1}{2} \|A(x_0)\|_2 \le \|A_n(x_0)\|_2 \le \frac{3}2 \|A(x_0)\|_2.
    \end{equation}

    Choose a open subinterval $J \subseteq I$ containing $x_0$ such that 
    \begin{equation} \label{eq:JW}
        \|A'\|_{L^1(J,M_d(\C))} \le \frac{r}{32} \, \|A(x_0)\|_2.
    \end{equation}
    For $x \in J$ we have 
    \begin{align*}
        |\|A(x)\|_2 - \|A(x_0)\|_2| \le \|A(x)-A(x_0)\|_2 \le \|A'\|_{L^1(J,M_d(\C))} \le \frac{r}{32} \, \|A(x_0)\|_2,
    \end{align*}
    and thus 
    \begin{equation} \label{eq:cl2W}
        \frac{1}{2} \|A(x_0)\|_2 \le \|A(x)\|_2 \le \frac{3}2 \|A(x_0)\|_2.
    \end{equation}
    For $x \in J$ we also have, by \eqref{eq:L1clW}, \eqref{eq:cl1W}, and \eqref{eq:JW},
    \begin{align*}
        |\|A_n(x)\|_2 - \|A_n(x_0)\|_2| 
    &\le \|A_n(x)-A_n(x_0)\|_2 
    \\
    &\le \|A_n'\|_{L^1(J,M_d(d))} 
    \le \frac{r}{16}\, \|A(x_0)\|_2
    \le \frac{r}{8}\, \|A_n(x_0)\|_2
    \end{align*}
    and thus 
    \begin{equation} \label{eq:clnW}
        \frac{1}{2} \|A_n(x_0)\|_2 \le \|A_n(x)\|_2 \le \frac{3}2 \|A_n(x_0)\|_2.
    \end{equation}

    Therefore, $\ul A := A/\|A\|_2 $ and $\ul A_n := A_n/\|A_n\|_2$ are well-defined on $J$ and belong to $W^{1,q}(J, \on{Herm}_T^0(d))$.
    By \eqref{eq:pU3} and \eqref{eq:clW},
    \begin{align*}
        \|\ul A(x_0) - \ul A_n(x_0)\|_2\le \frac{2}{\|A(x_0)\|_2} \|A(x_0) - A_n(x_0)\|_2 \le \frac{r}{4}. 
    \end{align*}
    Moreover, by \eqref{eq:ulA'} and \eqref{eq:cl2W}, for $x \in J$, 
    \begin{align*}
        \|\ul A'(x)\|_2 \le 2\, \frac{\|A'(x)\|_2}{\|A(x)\|_2} \le 4\,\frac{\|A'(x)\|_2}{\|A(x_0)\|_2} 
    \end{align*}
    and, by \eqref{eq:cl2W} and \eqref{eq:cl1W}, 
    \begin{align*}
        \|\ul A_n'(x)\|_2 \le 2\, \frac{\|A_n'(x)\|_2}{\|A_n(x)\|_2} \le 8\, \frac{\|A_n'(x)\|_2}{\|A(x_0)\|_2}.
    \end{align*}
    Therefore, by \eqref{eq:JW},
    \begin{align*}
        \|\ul A'\|_{L^1(J,M_d(\C))} \le \frac{r}{8} 
    \end{align*}
    and, by \eqref{eq:L1clW} and \eqref{eq:JW},
    \begin{align*}
        \|\ul A_n'\|_{L^1(J,M_d(\C))} \le \frac{r}{2}. 
    \end{align*}
    Thus 
    $\ul A(J) \subseteq  B(\ul A(x_0),r/4)$ and $\ul A_n(J) \subseteq B(\ul A_n(x_0),r/2) \subseteq B(\ul A(x_0),3r/4)$.
    By \eqref{eq:r}, there exists $i \in \{1,\ldots,s\}$
    such that
    \[
        \ul A(J) \subseteq \cV_i \quad \text{ and } \quad  \ul A_n(J) \subseteq \cV_i 
    \]
    for all $n \ge n_1$. 
    Thus (1) is proved.

    By \eqref{eq:BD} and \eqref{eq:eigABC}, we have the simultaneous unitary block-diagonalization \eqref{eq:simBD} as well as \eqref{eq:eigBD} on $J$ for $n \ge n_1$. 
    Let us check that $B,B_n, C,C_n$ are of class $W^{1,q}$ on $J$ and satisfy \eqref{eq:convBD}.
    To this end, we suppress the subscript and just write $U : \cV \to U_d(\C)$ for $U_i : \cV_i \to U_d(\C)$.
    By \Cref{l:bdiag}, 
    pointwise on $J$, 
    \begin{align*}
        \| U^*(\ul A)  A\, U(\ul A) - U^*(\ul A_n) A_n\, U(\ul A_n) \|_2  
        \le   C \, \|A-A_n\|_2 
    \end{align*}
    and, using also \eqref{eq:cl2W}, almost everywhere on $J$,
    \begin{align*} 
        \| (U^*(\ul A) & A\, U(\ul A))'  - (U^*(\ul A_n) A_n\, U(\ul A_n))'  \|_2 \nonumber \\ 
                       & \quad \le  C\, \Bigl (  \|A' - A_n'\|_2 +   ( \| A'\|_2+ \| A_n'\|_2) \|A\|_2^{-1}  \|A -A_n\|_2 \Bigr ),
                       \\
                       & \quad \le  C\,\Bigl (  \|A' - A_n'\|_2 +   2\,\|A(x_0)\|_2^{-1}\, ( \| A'\|_2+ \| A_n'\|_2)   \|A -A_n\|_2 \Bigr ),
    \end{align*}
    where $C$ is a universal constant depending only on $d$. This implies 
    \begin{align*}
        \| U^*(\ul A)  A\, U(\ul A) - U^*(\ul A_n) A_n\, U(\ul A_n) \|_{L^q(J,M_d(\C))}  
        \le   C\,  \|A-A_n\|_{L^q(J,M_d(\C))} 
    \end{align*}
    and
    \begin{align*} 
    &\| (U^*(\ul A)  A\, U(\ul A))'  - (U^*(\ul A_n) A_n\, U(\ul A_n))'  \|_{L^q(J,M_d(\C))}^q
    \\
    &\quad\le (2C)^q \,\|A' - A_n'\|_{L^q(J,M_d(\C))}^q 
    \\
    &\qquad + \Big(\frac{8C}{\|A(x_0)\|_2}\Big)^q \big(\| A'\|_{L^q(J,M_d(\C))}^q+ \| A_n'\|_{L^q(J,M_d(\C))}^q\big) \|A -A_n\|_{L^\infty(J,M_d(\C))}^q.
    \end{align*}
    By \Cref{c:supW}, we conclude \eqref{eq:convBD}. This completes the proof.
\end{proof}

By \Cref{l:TrW}, \Cref{l:red}, and  
the induction hypothesis, 
for each $x_0 \in I \setminus Z_A$ there is an open interval $J$ with $x_0 \in J \subseteq I \setminus Z_A$ such that 
\[
\|\eig(A)' - \eig(A_n)'\|_{L^q(J,\R^{d})} \to 0  
     \quad \text{ as } n \to \infty.
\]
Thus, in view of \Cref{l:Z}, there is a countable cover $\{F_i\}_{i \ge 1}$ of $I$ by measurable 
sets such that, for each $F_i$, 
\[
    \|\eig(A)' - \eig(A_n)'\|_{L^q(F_i,\R^{d})} \to 0  
    \quad \text{ as } n \to \infty.
\]
In view of \Cref{l:Vapp} (applied to $f_n := \|\eig(A)' - \eig(A_n)'\|_2^q$), the desired assertion \eqref{eq:toshow} will 
follow from the following lemma.

\begin{lemma} \label[l]{l:ui}
    Let $1 \le q < \infty$.
    Let $I \subseteq \R$ be a bounded open interval.     
    Assume that $A_n \to A$ in $W^{1,q} (I, \on{Herm}(d))$ as $n \to \infty$. 
    Then $\{\|\eig(A)' - \eig(A_n)'\|_2^q : n \ge 1\} \subseteq L^1(I)$ is uniformly integrable.
\end{lemma}

\begin{proof}
    Since $A_n' \to A'$ in $L^q(I,M_d(\C))$ as $n \to \infty$,
    the set $\{\|A'\|_2^q+\|A_n'\|_2^q : n \ge 1\} \subseteq L^1(I)$ is uniformly integrable, by Vitali's convergence theorem \ref{thm:Vitali}.
Thus,
by de la Vall\'ee Poussin's criterion \ref{thm:VP}, there exists a nonnegative increasing function $G$ on $[0,\infty)$ such that 
$G(t)/t \to \infty$ as $t \to \infty$ and 
\begin{align*}
    \sup_{n \ge 1} \int_I G(2^q(\|A'\|_2^q+\|A'_n\|_2^q)) \, dx < \infty.
\end{align*}
By \eqref{eq:Loewner} or \eqref{eq:charHerm3}, almost everywhere in $I$,
\[
    \|\eig(A)' - \eig(A_n)'\|_2^q \le  2^q(\|A'\|_2^q+\|A'_n\|_2^q)
\]
so that the assertion follows, by de la Vall\'ee Poussin's criterion \ref{thm:VP} again.
\end{proof}

This completes the proof of \Cref{p:indhermW} and thus of \Cref{thm:mainhermW}.

%-----------------------------------------------------------------------------------------------------------------------
\subsection{Almost everywhere convergence of the derivatives of the eigenvalues} \label{ssec:mptw} 
%-----------------------------------------------------------------------------------------------------------------------

The following theorem is a version of \Cref{thm:mptw} for $m=1$.

\begin{theorem} \label[t]{thm:ptw}
    Let $I \subseteq \R$ be a bounded open interval. 
    Let $A_n \to A$ in $C^{0,1}(\ol I,\on{Herm}(d))$ as $n \to \infty$. 
    Then, for almost every $x \in I$,  
    \begin{equation} \label{eq:1ptw}
        \eig(A_n)'(x) \to \eig(A)'(x) \quad n \to \infty. 
    \end{equation}
\end{theorem}

Let us prove \Cref{thm:ptw} following the proof of \Cref{thm:mainhermW} and indicating the required modifications.

In analogy to \Cref{l:TrW}, it is easy to conclude the following lemma
which allows us to assume $\on{Tr}(A) = \on{Tr}(A_n) =0$ for all $n \ge 1$.

\begin{lemma} \label[l]{l:Tr}
    Let $I \subseteq \R$ be a bounded open interval.  
    Assume that $A_n \to A$ in $C^{0,1} (\ol I, \on{Herm}(d))$ as $n \to \infty$.
    Then:
    \begin{enumerate}
        \item $\tilde A_n \to \tilde A$ in $C^{0,1} (\ol I, \on{Herm}_T(d))$ as $n \to \infty$.
        \item Moreover, 
            \[
                \eig(A_n)' \to  \eig(A)'  \quad \text{ almost everywhere in }  I 
            \]
            if and only if 
            \[
                \eig(\tilde A_n)' \to  \eig(\tilde A)'  \quad \text{ almost everywhere in }  I 
            \]
            as  $n\to \infty$. 
    \end{enumerate}
\end{lemma}

We prove \eqref{eq:1ptw} by induction on $d$. If $d=1$ the assertion is trivially true. 
So assume that $d\ge 2$.

On the zero set $Z_A$ 
of $A$, we have the following lemma. 

\begin{lemma} \label[l]{l:Zp}
    Let $I \subseteq \R$ be a bounded open interval
    and assume that $A_n \to A$ in $C^{0,1} (\ol I, \on{Herm}(d))$ as $n \to \infty$.
    Then, for almost every $x \in Z_A$, 
    \begin{equation} \label{eq:ptw3}
        \eig(A_n)'(x)  \to  0 \quad \text{ as } n\to \infty. 
    \end{equation}
\end{lemma}

\begin{proof}
    By \Cref{p:Loewner}, for every $n \ge 1$,
    \[
        \sup_{(x,y) \in Z_A^{<2>}} \frac{\|\mathcal E(A_n)(x) - \mathcal E(A_n)(y)\|_2}{|x-y|} \le \sup_{(x,y) \in Z_A^{<2>}} \frac{\|A_n(x)-A_n(y)\|_2}{|x-y|}.
    \]
    By \Cref{l:slope1} and \Cref{l:slope2},
    \[
        \sup_{(x,y) \in Z_A^{<2>}} \frac{\|A_n(x)-A_n(y)\|_2}{|x-y|} \to \sup_{(x,y) \in Z_A^{<2>}} \frac{\|A(x)-A(y)\|_2}{|x-y|} = 0 \quad \text{ as } n \to \infty. 
    \]

    Let $E$ be the set of $x \in \on{acc}(Z_A)$, where the derivatives $\mathcal E(A_n)'(x)$ for all $n\ge 1$ exist.
    For each $x_0 \in E$ there is a sequence $x_k \to x_0$ with $x_k \in Z_A$ and $x_k \ne x_0$ for all $k \ge 1$. Thus, for fixed $n$, 
    \begin{align*}
        \big\| \mathcal E(A_n)'(x_0) \big\|_2 
       & = \lim_{k \to \infty} \frac{\|\mathcal E(A_n)(x_0) - \mathcal E(A_n)(x_k)\|_2}{|x_0-x_k|}
        \\
       &\le \sup_{(x,y) \in Z_A^{<2>}} \frac{\|\mathcal E(A_n)(x) - \mathcal E(A_n)(y)\|_2}{|x-y|}.
    \end{align*}
    Hence we may conclude that, for each $x_0\in E$, 
    \[
        \mathcal E(A_n)'(x_0) \to  0  \quad \text{ as } n \to \infty. 
    \]
    This completes the proof, since $E$ has full measure in $Z_A$.
\end{proof}

We need the following variant of \Cref{l:red}.

\begin{lemma} \label[l]{l:redC}
    Let $I \subseteq \R$ be a bounded open interval.     
    Assume that $A_n \to A$ in $C^{0,1} (\ol I, \on{Herm}_T(d))$ as $n \to \infty$.
    Let $x_0 \in I\setminus Z_A$. 
    Then there exist an open interval $J$ with $x_0 \in J \subseteq I \setminus Z_A$, 
    $n_0 \ge 1$, and $i \in \{1,\ldots,s\}$ such that 
    \begin{enumerate}
        \item for all $n \ge n_0$, the curves
            $\ul A := A/\|A\|_2 $ and $\ul A_n := A_n/\|A_n\|_2$ belong to $C^{0,1}(\ol J, \on{Herm}_T^0(d))$ 
            and satisfy
            $\ul A(J), \ul A_n(J) \subseteq \cV_i$; 
        \item on $J$ and for all $n \ge n_0$, we have  
            \begin{align*} \label{eq:simBDC}
                U_i^* (\ul A) A \, U_i(\ul A) 
                =   
                \begin{pmatrix}
                    B & 0 \\
                    0 &  C
                \end{pmatrix}, \quad
                U_i^* (\ul A_n) A_n \, U_i(\ul A_n) 
                =   
                \begin{pmatrix}
                    B_n & 0 \\
                    0 &  C_n
                \end{pmatrix},
            \end{align*}
            where $B,B_n \in C^{0,1}(\ol J,\on{Herm}(d_1))$ and $C,C_n \in C^{0,1}(\ol J,\on{Herm}(d_2))$ with $d_1+d_2 = d$;
        \item on $J$ and for all $n \ge n_0$, we have 
            \begin{equation*} \label{eq:eigBDC}
                \eig(A) = (\eig(B),\eig(C)), \quad  \eig(A_n) = (\eig(B_n),\eig(C_n));
            \end{equation*}
        \item we have 
            \begin{equation*} \label{eq:convBDC}
                \|B - B_n\|_{C^{0,1}(\ol J,M_{d_1}(\C))} \to 0, \quad 
                \|C - C_n\|_{C^{0,1}(\ol J,M_{d_2}(\C))} \to 0 
            \end{equation*}
            as $n \to \infty$.  
    \end{enumerate}
\end{lemma}

\begin{proof}
    The existence of $J$, $n_0$, and $i$ such that 
    $\ul A(J), \ul A_n(J) \subseteq \cV_i$, for $n\ge n_0$, 
    follows from \Cref{l:red} (since $C^{0,1}$-convergence entails $W^{1,q}$-convergence).
    The other assertions follow easily from \eqref{eq:BD}, \eqref{eq:eigABC}, and \Cref{l:bdiag}.
\end{proof}

By \Cref{l:Tr}, \Cref{l:redC}, and  
the induction hypothesis, 
for each $x_0 \in I \setminus Z_A$ there is an open interval $J$ with $x_0 \in J \subseteq I \setminus Z_A$ such that 
\[
    \eig(A_n)' \to \eig(A)' \quad \text{ almost everywhere in } J \text{ as } n \to \infty.
\]
Together with \Cref{l:Zp}, this implies 
\[
    \eig(A_n)' \to \eig(A)' \quad \text{ almost everywhere in } I \text{ as } n \to \infty
\]
and hence completes the proof of \Cref{thm:ptw}.

%-----------------------------------------------------------------------------------------------------------------------
\subsection{Stability of the condition number of matrices} \label{ssec:cn}
%-----------------------------------------------------------------------------------------------------------------------

We finish this section with the proof of \Cref{t:cn}.

Recall that the condition number of a nonsingular matrix $A \in M_d(\C)$ is the ratio of the largest by the smallest 
singular value: 
\begin{equation}
    \ka(A) = \frac{\si_1(A)}{\si_d(A)}.
\end{equation}
Let $1 \le q < \infty$. Let $I \subseteq \R$ be a bounded open interval.
Suppose that $A_0 \in W^{1,q}(I, M_d(\C))$ and $\inf_{x \in I}\si_d(A_0(x)) > 0$.
We must show that
$\ka(A) \in W^{1,q}(I,\R)$ is well-defined in a neighborhood $\cU$ of $A_0$ in $W^{1,q}(I, M_d(\C))$ 
and that the induced map $\ka : \cU \to W^{1,q}(I,\R)$ is continuous.

    By the continuity of $\si_* : W^{1,q}(I,M_{d}(\C)) \to W^{1,q}(I,\R^d)$, due to \Cref{thm:sv}, 
    and by \Cref{l:mat} and \eqref{eq:estsv}, 
    there is a neighborhood $\cU$ of $A_0$ in $W^{1,q}(I, M_d(\C))$ such that
    \begin{equation}
        \de:=\inf_{A \in \cU}\inf_{x \in I}\si_d(A(x)) >0.
    \end{equation}
    Thus the map $\ka : \cU \to W^{1,q}(I,\R)$ is well-defined. 
    
    To see continuity, let $A,B \in \cU$. Then, in $I$, 
    \begin{align}
        (\ka(A)- \ka(B)) \si_d(A)\si_d(B)= (\si_1(A)-\si_1(B))\si_d(B) - \si_1(B)(\si_d(A)-\si_d(B)) 
    \end{align}
    and, almost everywhere in $I$,
    \begin{align}
        &(\ka(A)'- \ka(B)') \si_d(A)\si_d(B) +(\ka(A)- \ka(B)) \si_d(A)'\si_d(B) 
        \\
        &\hspace{2cm} +(\ka(A)- \ka(B)) \si_d(A)\si_d(B)'
                                            \\
        &\quad =(\si_1(A)'-\si_1(B)')\si_d(B) - \si_1(B)'(\si_d(A)-\si_d(B)) 
        \\
        &\hspace{2cm} +(\si_1(A)-\si_1(B))\si_d(B)' - \si_1(B)(\si_d(A)'-\si_d(B)').
    \end{align}
    We see that, in $I$, 
    \begin{align*}
        |\ka(A)- \ka(B)| \le  \frac{\|B\|_2}{\de^2} \big(|\si_1(A)-\si_1(B)| + |\si_d(A)-\si_d(B)|\big)
    \end{align*}
    and, almost everywhere in $I$,
    \begin{align*}
        |\ka(A)'- \ka(B)'| &\le \frac{1}{\de^2} \Big(\|B\|_2|\si_1(A)'-\si_1(B)'|  + |\si_1(B)'| |\si_d(A)-\si_d(B)|  
        \\
        &\qquad +|\si_d(B)'||\si_1(A)-\si_1(B)| + \|B\|_2|\si_d(A)'-\si_d(B)'|
        \\
        &\qquad + \big(\|B\|_2|\si_d(A)'| + \|A\|_2 |\si_d(B)'|\big) |\ka(A)- \ka(B)|  
        \Big).
    \end{align*}
    Now continuity of $\ka: \cU \to W^{1,q}(I,\R)$ can be easily deduced from \Cref{thm:sv}, in view of \Cref{l:mat} and \eqref{eq:estsv}. 
    The proof of \Cref{t:cn} is complete.

%---------------------------------------------------------------------------------------------
\section{Eigenvalue stability for Hermitian matrices: multiparameter case} \label{sec:proof2}
%---------------------------------------------------------------------------------------------

In this section, we will prove \Cref{thm:eigW}, \Cref{thm:eigmap}, \Cref{cor:eigmap} and \Cref{thm:mptw}.

%-----------------------------------------------------------------------------------------------------------------------
\subsection{Proof of \Cref{thm:eigW}}
%-----------------------------------------------------------------------------------------------------------------------

The next \Cref{thm:multhermW} implies \Cref{thm:eigW} 
because $W^{1,q}(U,\on{Herm}(d))$ is first-countable. %it is a Banach space hence metrizable 

\begin{theorem} \label[t]{thm:multhermW}
    Let $1 \le q < \infty$.
    Let $U \subseteq \R^m$ be open and bounded. 
    Let $A_n \to A$ in $W^{1,q}(U,\on{Herm}(d))$ as $n \to \infty$.
    Then $\eig(A_n) \to \eig(A)$ in $W^{1,q}(U,\R^d)$ as $n \to \infty$. 
\end{theorem}

\begin{proof}
    Set $\la := \eig(A)$ and $\la_n := \eig(A_n)$. 
    By \eqref{eq:Loewner}, we have 
    \begin{equation*}
        \|\la - \la_n\|_{L^q(U,\R^d)} \le \|A - A_n\|_{L^q(U,M_d(\C))} \to 0 \quad \text{ as } n \to \infty.
    \end{equation*}
    Hence we have to show that 
    \begin{equation} \label{eq:showLq}
        \|\p_j \la -  \p_j \la_n\|_{L^{q}(U,\R^d)}  \to 0 \quad \text{ as } n\to \infty,
    \end{equation}
    for all $1 \le j \le m$. 
    It is enough to show that there is a subsequence $(n_k)$ with this property.

    Let us first assume that $U = I_1 \times \cdots \times I_m$ is an open box with sides parallel to the coordinate axes.
    Let $j=1$.
   We may assume that $A$ and $A_n$ are absolutely continuous on almost all line segments in $U$ parallel to the $x_1$-axis 
   and $\p_1 A \in L^q(U,\on{Herm}(d))$. 
    Let $x = (x_1,x')$ for 
    $x' \in U' = I_2\times \cdots \times I_m$.
    By Tonelli's theorem, 
    \[
        \int_{U'} \int_{I_1} \|A(x_1,x') - A_n(x_1,x')\|_2^q\, dx_1 \, dx'= \| A -  A_n\|_{L^q(U,M_d(\C))}^q
    \]
    and
    \[
        \int_{U'} \int_{I_1} \|\p_1 A(x_1,x') - \p_1 A_n(x_1,x')\|_2^q\, dx_1 \, dx'= \|\p_1 A - \p_1 A_n\|_{L^q(U,M_d(\C))}^q.
    \]
    Thus there is a subsequence $(n_k)$ such that 
    \[
         \int_{I_1} \|A(x_1,x') - A_{n_k}(x_1,x')\|_2^q\, dx_1  \to 0 
    \]
    and
    \[
         \int_{I_1} \|\p_1 A(x_1,x') - \p_1 A_{n_k}(x_1,x')\|_2^q\, dx_1  \to 0 
    \]
    for almost every $x' \in U'$ as $k \to \infty$.  
    By \Cref{thm:mainhermW}, 
    \[
        F_{1,n_k}(x'):= \int_{I_1} \|\p_1 \la(x_1,x') - \p_1 \la_{n_k}(x_1,x')\|_2^q\, dx_1\to 0 \quad \text{ as } k \to \infty
    \]
    for almost every $x' \in U'$. 

    By Vitali's convergence theorem \ref{thm:Vitali},  
    $\|\p_1 A- \p_1 A_n\|_{L^q(U,M_d(\C))} \to 0$ as $n \to \infty$ implies that the set 
    $\{\|\p_1 A\|_2^q + \|\p_1 A_n\|_2^q : n \ge 1\} \subseteq L^1(U)$ 
    is uniformly integrable and thus, by de la Vall\'ee Poussin's criterion \ref{thm:VP}, 
    there is a nonnegative increasing convex function $G$ on $[0,\infty)$ 
    such that $G(t)/t \to \infty$ as $t \to \infty$ and 
    \[
        \sup_{n \ge 1} \int_U G(2^q(\|\p_1 A(x)\|_2^q + \|\p_1 A_n(x)\|_2^q)) \, dx < \infty.
    \]
    We claim that $\{|I_1|^{-1} F_{1,n_k} : k \ge 1\}$ and 
    thus $\{F_{1,n_k} : k \ge 1\} \subseteq L^1(U')$ is uniformly integrable.
    By Tonelli's theorem, Jensen's inequality, and \eqref{eq:charHerm3},
    \begin{align*}
        \int_{U'} G(|I_1|^{-1} F_{1,n_k}(x'))\, dx' &= 
        \int_{U'} G \Big( \int_{I_1} \|\p_1 \la(x_1,x') - \p_1 \la_{n_k}(x_1,x')\|_2^q\, \frac{dx_1}{|I_1|} \Big)\, dx'
        \\
                                                    &\le \int_{U'}  \int_{I_1} G( \|\p_1 \la(x_1,x') - \p_1 \la_{n_k}(x_1,x')\|_2^q )\, \frac{dx_1}{|I_1|} \, dx'
                                                    \\
                                                    &\le \frac{1}{|I_1|} \int_{U}  G(2^q (\|\p_1 \la(x)\|_2^q + \| \p_1 \la_{n_k}(x)\|_2^q) ) \, dx
                                                    \\
                                                    &\le \frac{1}{|I_1|} \int_{U}  G(2^q (\|\p_1 A(x)\|_2^q + \| \p_1 A_{n_k}(x)\|_2^q) ) \, dx
    \end{align*}
    is bounded by a constant independent of $k$, which implies the claim, again by de la Vall\'ee Poussin's criterion \ref{thm:VP}.

    By Vitali's convergence theorem \ref{thm:Vitali} and Tonelli's theorem, 
    \[
        \|\p_1 \la - \p_1 \la_{n_k}\|_{L^q(U,M_d(\C))}\to 0 \quad \text{ as } k \to \infty.
    \]
    The reasoning for $2 \le j \le m$ is analogous.
    Thus \eqref{eq:showLq} is proved, in the case $U = I_1 \times \cdots \times I_m$.

    Let now $U$ be a general open bounded subset of $\R^m$. 
    Since $U$ is a countable union of open bounded boxes $V_i$ with sides parallel to the coordinate axes, 
    \eqref{eq:showLq} follows from \Cref{l:Vapp} applied to $f_n := \|\p_j \la - \p_j \la_n \|_2^q$, 
    noting that $\{f_n\}$ is uniformly integrable which can be seen as in the proof of \Cref{l:ui}.
    This ends the proof.
\end{proof}

%---------------------------------------------------------------------------------------------
\subsection{Proofs of \Cref{thm:eigmap} and \Cref{thm:mptw}}
%---------------------------------------------------------------------------------------------

First observe that \Cref{thm:mptw} follows easily from \Cref{thm:ptw} (applying the latter coordinate-wise).

As already pointed out, \Cref{thm:eigmap} is an immediate consequence of \Cref{thm:eigW}.
Using \Cref{thm:mptw}, we deduce the following slightly stronger version.

\begin{theorem} \label[t]{thm:eigmapS}
    Let $U \subseteq \R^m$ be open and bounded.
    Assume that $A_n \to A$ in $C^{0,1}(\ol U,\on{Herm}(d))$ as $n \to \infty$. 
    Then 
    \begin{equation} \label{eq:S1}
        \|\eig(A) - \eig(A_n) \|_{L^\infty(U,\R^d)} \le \|A-A_n\|_{L^\infty(U,M_d(\C))} \to 0
    \end{equation}
    and, for each $1 \le j \le m$ and all $1 \le q <\infty$,
    \begin{equation} \label{eq:S2}
        \|\p_j(\eig(A)) - \p_j(\eig(A_n)) \|_{L^q(U,\R^d)} \to 0
    \end{equation}
    as $n \to \infty$.
\end{theorem}

\begin{proof}
    First, \eqref{eq:S1} follows immediately from \eqref{eq:Loewner}.

    By \eqref{eq:eigLip},
    \begin{equation*}
        |\eig(A_n)|_{C^{0,1}(\ol U,\R^d)} \le |A_n|_{C^{0,1}(\ol U,M_d(\C))} \le L,
    \end{equation*}
    for a constant $L>0$ independent of $n$. In particular, 
    for each $1 \le j \le m$ and almost every $x \in U$,
    \[
        \|\p_j(\eig(A_n))(x)\|_2 \le L.
    \]
    We conclude \eqref{eq:S2},
    by \Cref{thm:mptw} and the dominated convergence theorem.
\end{proof}

%---------------------------------------------------------------------------------------------
\subsection{Proof of \Cref{cor:eigmap}}
%---------------------------------------------------------------------------------------------

\Cref{cor:eigmap} is an immediate consequence of the following corollary of \Cref{thm:multhermW} 
and \Cref{ex:A}.

\begin{corollary}
    Let $U \subseteq \R^m$ be a bounded open Lipschitz domain.
    Let $A_n \to A$ in $C^{0,1}(\ol U,\on{Herm}(d))$ as $n\to  \infty$.
    Then, for each $0 < \al  < 1$,
    \begin{equation*} \label{eq:mainconcl3}
        \|\eig(A) -  \eig(A_n)\|_{C^{0,\al}(\ol U,\R^d)}  \to 0 \quad \text{ as } n\to \infty. 
    \end{equation*}
\end{corollary}

\begin{proof}
    The assertion follows from \Cref{thm:multhermW} and Morrey's inequality,
    \[
    \|\eig(A)-  \eig(A_n)\|_{C^{0,\al}(\ol U,\R^d)} \le C\, \|\eig(A) -  \eig(A_n)\|_{W^{1,q}(U,\R^d)},
    \]
    where $\al = 1-m/q$, $q> m$, and $C=C(d,m,q,U)$.
\end{proof}

%-----------------------------------------------------------------------------------------------------------------------
\section{Examples} \label{sec:examples}
%-----------------------------------------------------------------------------------------------------------------------

\Cref{ex:A} shows that the characteristic map $\eig : C^{0,1}(\ol U,\on{Herm}(d)) \to C^{0,1}(\ol U,\R^d)$ 
is not continuous. This example appeared in \cite[Examples 1.12 and 7.13]{Parusinski:2024aa}; we repeat it for 
the reader's convenience.

\begin{example} \label[e]{ex:A}
    The sequence $(A_n)_n$ of curves of symmetric $2 \times 2$ matrices 
    \begin{equation} \label{eq:exA}
        A_n(x) = \begin{pmatrix} \frac{1}n & x \\ x & - \frac{1}{n}  \end{pmatrix}, \quad x \in \R,
    \end{equation}
    converges to 
    \[
        A(x) = \begin{pmatrix} 0 & x \\ x & 0  \end{pmatrix}, \quad x \in \R,
    \]
    uniformly in all derivatives on every compact interval. 
    We have $\eig(A_n) = (-a_n,a_n)$ and $\eig(A) = (-a,a)$, where
    \[
        a_n(x) :=  \sqrt{x^2 + \tfrac{1}{n^2}} \quad \text{ and } \quad a(x) := |x|.
    \]
    
    {\it Then $\eig(A_n) \not\to \eig(A)$ as $n \to \infty$ in the $C^{0,1}$ topology.}
   
    Indeed, for each bounded open interval $I \subseteq \R$ containing $0$, 
    \begin{align*}
        |a- a_n|_{C^{0,1}(\ol I)} &\ge 
        \sup_{0< x \in I} \Big|\frac{(a(x) - a_n(x)) - (a(0)-a_n(0))}{x} \Big| 
        \\
                                  &=  \sup_{0< x \in I} \Big|\frac{x - \sqrt{x^2 + \frac{1}{n^2}} + \frac{1}n}{x}\Big|
                                  \ge \Big|\frac{\frac{1}n- \sqrt{\frac{1}{n^2}+ \frac{1}{n^2}} + \frac{1}n}{\frac{1}n}\Big| 
                                  = 2-\sqrt 2,
    \end{align*}
    for large enough $n$. 
    
    Furthermore, observe that
    \[
        a_n'(x) = \frac{x}{\sqrt{x^2 + \frac{1}{n^2}}}
    \]
    tends pointwise to $a'(x) = \on{sgn}(x)$ for all $x \ne 0$ but not uniformly on any neighborhood of $0$:
    \[
        a_n'(\pm \tfrac{1}{n}) = \pm \frac{1}{\sqrt 2}.
    \]
    This also violates the first conclusion of \Cref{cor:se} for $q=\infty$.
\end{example}

The following \Cref{ex:ucq} proves that, for no $1 \le q < \infty$, 
the characteristic map $\eig : C^{0,1}(\ol U, \on{Herm}(d)) \to C^{0,1}_q(\ol U, \R^d)$ is uniformly continuous.

\begin{example} \label[e]{ex:ucq}
    Let $\vh_n(x) : I := (0,1) \to \R$ be the sawtooth function defined by
    \[
        \vh_n(x) = (-1)^k (x - \tfrac{k}{n}) + \vh_n(\tfrac{k}n) \quad \text{ if } k=0,\ldots, n-1 
        \text{ and } \tfrac{k}{n} \le x \le \tfrac{k+1}{n}.
    \]
    Then $|\vh_n'| =1$ almost everywhere in $I$.
    Moreover, $0 \le \vh_n \le \tfrac{1}{n}$ in $I$ and $\vh_n \ge \tfrac{1}{2n}$ on a measurable subset $E_n$ of $I$ of measure $|E_n|=\frac{1}{2}$.

   Consider 
   \[
       A_n(x) = \begin{pmatrix} \frac{1}{n} & \vh_n(x) \\ \vh_n(x) & - \frac{1}{n}  \end{pmatrix}, \quad 
       B_n(x) = \begin{pmatrix} \frac{1}{2n} & \vh_n(x) \\ \vh_n(x) & - \frac{1}{2n}  \end{pmatrix}, \quad x \in I.
   \]
   Then 
   \[
       \|A_n\|_{C^{0,1}(\ol I,M_2(\C))} = \frac{2}{n} + \sqrt 2, 
       \\
       \quad \|B_n\|_{C^{0,1}(\ol I,M_2(\C))} = \frac{\sqrt{5}}{\sqrt{2}n} + \sqrt 2,
   \]
   are bounded 
   and 
   \[
       \|A_n-B_n\|_{C^{0,1}(\ol I,M_2(\C))} = \frac{\sqrt{2}}{2n} \to 0 \quad \text{ as } n \to \infty.
   \]

   The nonnegative eigenvalues $a_n$ and $b_n$ of $A_n$ and $B_n$,
   \[
       a_n = \sqrt{\vh_n^2 + \tfrac{1}{n^2}},
       \quad b_n = \sqrt{\vh_n^2 + \tfrac{1}{4n^2}},
   \]
   are Lipschitz and satisfy, almost everywhere in $I$,
   \begin{align}
       |a_n'-b_n'| &= \vh_n|\vh_n'| \Big|\frac{n}{\sqrt{n^2 \vh_n^2 + 1}} -  \frac{2n}{\sqrt{4n^2 \vh_n^2 + 1}}\Big| \label{eq:compute} 
       \\
                   &= \vh_n \frac{3n}{\sqrt{n^2 \vh_n^2 + 1}\sqrt{4n^2 \vh_n^2 + 1}\big(2\sqrt{n^2 \vh_n^2 + 1}+\sqrt{4n^2 \vh_n^2 + 1}\big)}.
                 \\
                   &\ge  \frac{n\,\vh_n}{6},
   \end{align} 
   since $\vh_n \le \frac{1}{n}$ so that the denominator is bounded by 
   \[
      \sqrt 2 \sqrt 5 (2 \sqrt 2 + \sqrt 5) = 4 \sqrt{5} + 5 \sqrt{2} \le 18.  
   \]
   Thus, almost everywhere in $E_n$,
   \begin{align*}
       |a_n'-b_n'| &\ge \frac{1}{2n} \cdot \frac{n}{6} = \frac{1}{12}
   \end{align*} 
   and, consequently,
   \[
       \|a_n'-b_n'\|_{L^q(I)}^q \ge \|a_n'-b_n'\|_{L^q(E_n)}^q \ge \frac{1}{12^q} |E_n| =  \frac{1}{2\cdot 12^q},
   \]
   for every $1 \le q < \infty$,
   showing that 
   {\it the map $\eig : C^{0,1}(\ol I, \on{Herm}(2)) \to C^{0,1}_q(\ol I, \R^2)$ is not uniformly continuous.}
\end{example}

The next \Cref{ex:Auc} shows that, for no $0<\al<1$, 
the characteristic map $\eig : C^{0,1}(\ol U,\on{Herm}(d)) \to C^{0,\al}(\ol U,\R^d)$ is uniformly continuous. 
In contrast to \Cref{ex:ucq}, the curves of symmetric matrices are smooth but their Lipschitz constants are unbounded.

\begin{example} \label[e]{ex:Auc}
    Let $\al \in (0,1)$ and set $r= \frac{\al}{1-\al}$.
   Consider 
   \[
       A_n(x) = \begin{pmatrix} \frac{1}{n^r} & nx \\ nx & - \frac{1}{n^r}  \end{pmatrix}, 
       \quad
       B_n(x) = \begin{pmatrix} \frac{1}{2n^r} & nx \\ nx & - \frac{1}{2n^r}  \end{pmatrix}, \quad x \in \R.
   \]
   Then 
   \[
       A_n - B_n =\begin{pmatrix} \frac{1}{2n^r} & 0 \\ 0 & - \frac{1}{2n^r}  \end{pmatrix} 
   \]
   tends to zero uniformly in all derivatives as $n \to \infty$. 
   The nonnegative eigenvalues of $A_n$ and $B_n$ are 
   \[
       a_n(x) = \sqrt{n^2 x^2 + \tfrac{1}{n^{2r}}} \quad \text{ and } \quad b_n(x) = \sqrt{n^2 x^2 + \tfrac{1}{4n^{2r}}}.
   \]

   For each bounded open interval $I \subseteq \R$ containing $0$, 
    \begin{align*}
        |a_n- b_n|_{C^{0,\al}(\ol I)} &\ge 
        \sup_{0< x \in I} \frac{|(a_n(x) - b_n(x)) - (a_n(0)-b_n(0))|}{x^\al}  
        \\
                                      &=  \sup_{0< x \in I} \frac{\big| \sqrt{n^2 x^2 + \frac{1}{n^{2r}}} - \sqrt{n^2 x^2 + \frac{1}{4n^{2r}}} -\frac{1}{n^r} +\frac{1}{2n^r} \big|}{x^\al}
                                      \intertext{and setting $x= n^{-1/(1-\al)}$,}
                                      &\ge\frac{\big| \sqrt{\frac{1}{n^{2r}} + \frac{1}{n^{2r}}} - \sqrt{\frac{1}{n^{2r}}  + \frac{1}{4n^{2r}}} -\frac{1}{n^r} +\frac{1}{2n^r} \big|}{\frac{1}{n^r}}
                                      \\
                                      &= \frac{\sqrt{5}}{2} + \frac{1}2-\sqrt 2  
    \end{align*}
    for large enough $n$, 
    showing that {\it  the map $\eig : C^{0,1}(\ol I, \on{Herm}(2)) \to C^{0,\al}(\ol I, \R^2)$ is not uniformly continuous.}
    
    It also follows that $\eig : C^{0,1}(\ol I, \on{Herm}(2)) \to C^{0,1}_q(\ol I, \R^2)$ is not uniformly continuous, 
    for every $1 < q<\infty$, 
    by Morrey's inequality.
\end{example}

As discussed in \Cref{ssec:optimal}, it would be desirable to find effective moduli of 
continuity for restrictions of the characteristic map $\eig$ to (relatively) compact subspaces of $C^{0,1}(\ol U,\on{Herm}(d))$, 
e.g., the subset $K$ defined in \eqref{eq:K}.
The following \Cref{ex:A2} excludes $\al$-H\"older continuity of $\eig|_K : K \to C^{0,1}_q(\ol U,\R^d)$ for a
certain range of $\al$ depending on $q$.

\begin{example} \label[e]{ex:A2}
    Let $A_n$, for $n \ge 1$, be the smooth curves of symmetric matrices defined in \eqref{eq:exA} and $a_n$ the nonnegative eigenvalue of $A_n$.
    For $x>0$, we find (by a computation similar to \eqref{eq:compute})
\begin{align*}
    |a_n'(x) - a_{2n}'(x)| &= \frac{3nx}{\sqrt{n^2x^2 + 1}\sqrt{4n^2x^2 + 1} (\sqrt{4n^2x^2 + 1}+2\sqrt{n^2x^2 + 1})}.
\end{align*}
If $x \in (0, \frac{1}{n})$, then the denominator is bounded by $18$.
Thus, for fixed $1 \le q < \infty$,
\begin{align*}
    \|a_n' - a_{2n}'\|_{L^q((0,1))}^q \ge  \|a_n' - a_{2n}'\|_{L^q((0,\frac{1}{n}))}^q 
    \ge \frac{n^q}{6^q}\int_0^{1/n} x^q \, dx =  \frac{1}{6^q(q+1)n}, 
\end{align*}
so that 
\[
    \|a_n' - a_{2n}'\|_{L^q((0,1))} \ge \frac{1}{6(q+1)^{1/q} \, n^{1/q}},
\]
while
\[
    \|A_n-A_{2n}\|_{L^\infty(\R,M_2(\C))} = \frac{\sqrt 2}{2n} \quad \text{ and } \quad  A_n'-A_{2n}' \equiv 0.
\]
This implies that, 
{\it for no $\al \in (q^{-1},1]$, 
    the map $\eig : C^{0,1}(\ol I,\on{Herm}(2)) \to C^{0,1}_q(\ol I,\R^2)$ is $\al$-H\"older continuous,
where $I \subseteq \R$ is any bounded open interval containing $0$.}
\end{example}

%-----------------------------------------------------------------------------------------------------------------------
\section{\texorpdfstring{$d$}{d}-valued Sobolev functions} \label{sec:dSob}
%-----------------------------------------------------------------------------------------------------------------------

The aim of this section is to set up the background and tools for \Cref{thm:eiguW} and its proof. 
We follow \cite[Section 3]{Parusinski:2024ab}.

%-----------------------------------------------------------------------------------------------------------------------
\subsection{Unordered $d$-tuples of complex numbers} \label{ssec:AdC}
%-----------------------------------------------------------------------------------------------------------------------

The symmetric group $\on{S}_d$ acts on $\C^d$ by permuting the coordinates,
\[
    \si z = \si (z_1,\ldots,z_d) := (z_{\si(1)},\ldots, z_{\si(d)}), \quad \si \in \on{S}_d,~ z \in \C^d,
\]
and thus induces an equivalence relation.
The equivalence class of $z=(z_1,\ldots,z_d)$ is the \emph{unordered tuple} $[z]=[z_1,\ldots,z_d]$.  
Let us consider the set 
\[
    \cA_d(\C) := \{[z] : z \in \C^d\} 
\]
of unordered complex $d$-tuples 
and equip it with the metric
\begin{equation*}
    \dd_2([z],[w]) := \min_{\si \in \on{S}_d} \|z - \si w\|_2. 
\end{equation*}
Then $(\Iq,\dd_2)$ is a complete metric space and
the induced map $[~] : \C^d \to \cA_d(\C)$ is Lipschitz with $\Lip([~])=1$.

The element $[z_1,\ldots,z_d]$ of $\cA_d(\C)$ can also be represented by the sum $\sum_{i=1}^d \a{z_i}$, where 
$\a{z_i}$ denotes the Dirac mass at $z_i \in \C$.
If normalized, i.e., $\frac{1}{d}\sum_{i=1}^d \a{z_i}$, then, in this picture, $\frac{1}{\sqrt d}\dd_2$ is induced by 
the $L^2$-based Wasserstein metric on the space of probability measures on $\C$. 
(For this reason, in \cite[Section 3]{Parusinski:2024ab}, the factor $\frac{1}{\sqrt d}$ 
was integrated in the definition of $\dd_2$.)

%-----------------------------------------------------------------------------------------------------------------------
\subsection{$d$-valued Sobolev functions} \label{ssec:dvS}
%-----------------------------------------------------------------------------------------------------------------------

Due to Almgren \cite{Almgren00}, see also \cite{De-LellisSpadaro11}, 
there exist an integer $N= N(d)$
and an injective Lipschitz mapping
\[
    \De : \cA_d(\C) \to \R^N
\]
with $\Lip(\De)\le 1$ and   
Lipschitz inverse $\De|^{-1}_{\De(\cA_d(\C))}$ with $\Lip(\De|^{-1}_{\De(\cA_d(\C))}) \le C(d)$. 
Moreover, there is a Lipschitz retraction of $\R^N$ onto $\De(\cA_d(\C))$.

Following Almgren, the bi-Lipschitz embedding $\De$ can be used to define Sobolev spaces of $\cA_d(\C)$-valued functions: 
for open $U \subseteq \R^m$ and $1 \le q \le \infty$ set 
\[
    W^{1,q}(U,\cA_d(\C)) := \{f : U \to \cA_d(\C) : \De \o f \in W^{1,q}(U,\R^N)\}.  
\]
For an equivalent intrinsic definition of $W^{1,q}(U,\cA_d(\C))$, see \cite[Definition~0.5 and Theorem~2.4]{De-LellisSpadaro11}.
Then $W^{1,q}(U,\cA_d(\C))$ carries the metric
\begin{equation} \label{eq:Almgren}
    \rh^{1,q}_\De(f,g) :=  \|\De \o f - \De \o g\|_{W^{1,q}(U,\R^N)}
\end{equation}
which makes it a complete metric space (where functions that coincide almost everywhere are identified); see \cite[Lemma 3.1]{Parusinski:2024ab}.

We will see in \Cref{t:topmult} that, for $1 \le q <\infty$, the topology on $W^{1,q}(U,\cA_d(\C))$ induced by $\rh^{1,q}_\De$
is independent of the choice of the Almgren embedding $\De$.

\begin{remark}
    In \cite{Almgren00} and \cite{De-LellisSpadaro11}, $W^{1,q}(U,\cA_d(\R^\ell))$ is studied for arbitrary $\ell$.
\end{remark}

%-----------------------------------------------------------------------------------------------------------------------
\subsection{Almgren's embedding}
%-----------------------------------------------------------------------------------------------------------------------

Let us recall Almgren's construction of $\De$ for $\cA_d(\C)$. 

\begin{definition}[Almgren map]\label[d]{def:Amap}
    We say that 
    $$
    \Et : \cA_d(\C) \to \R^d
    $$
    is an \emph{Almgren map} if there is a unit complex number 
    $\th \in \C$ such that $\Et ([z])$ is an array of $d$ real numbers $\et (z_i):=\Re (\th z_i) $ arranged in increasing order, i.e.,
    $$
    \Et ([z]) = \Et ([z_1,\ldots,z_d]) = (\et (z_{\sigma(1)}) , \ldots , \et (z_{\sigma(d)}), 
    $$
    where $\si \in \on{S}_d$ is chosen such that $\et (z_{\sigma(1)})\le \et (z_{\sigma(2)}) \le 
    \cdots \le \et (z_{\sigma(d)})$.
    We also say that $\Et$ is the Almgren map associated to the real linear form $\et$. 
\end{definition}

By Almgren's combinatorial lemma (see e.g.\ \cite[Lemma 2.3]{De-LellisSpadaro11}) there exists $\al=\al(d)>0$ and a finite set of linear forms 
$\Lambda=\{\et_1, \ldots \et_h\}$, where $\et_l (z):=\Re (\th_l z) $ 
for unit complex numbers $\th_l$, with the following property: given any set of $d^2$ complex numbers, $\left\{z_1,\ldots,z_{d^2}\right\}\subseteq\C$, 
there exists $\et_l\in\Lambda$ such that
\begin{equation}\label{eq:combineq}
    |\et_l (z_k)| \geq \alpha |z_k| \quad\textrm{ for all }k\in\big\{1,\ldots,d^2\big\}.
\end{equation}
For instance, we may take $h=2d^2+1$ and $\{\et_1,\ldots,\et_l\}$ induced by  the set $\{\th_1, \ldots \th_h\}$ of all $h$-th roots of unity. 

Let $\Et_l$ denote the Almgren map associated to $\et_l \in \La$.
The \emph{Almgren embedding} $\De : \cA_d(\C) \to \R^N$, $N=dh$, is then defined by
\begin{align}\label{eq:AlmgrenDE}
    \De ([z]) = h^{-1/2} (\Et_1 ([z]), \ldots, \Et_h ([z])).
\end{align}  
The property \eqref{eq:combineq} guarantees that $\De$ is as required in \Cref{ssec:dvS}; 
see e.g.\ \cite[Section 2.1.2]{De-LellisSpadaro11}.

%-----------------------------------------------------------------------------------------------------------------------
\subsection{The space \texorpdfstring{$\cA_d(\R)$}{AdR}} \label{ssec:AdR}
%-----------------------------------------------------------------------------------------------------------------------

Consider the subspace $\cA_d(\R) := \{[x] : x \in \R^d \}$ of $\cA_d(\C)$.
For $x \in \R^d$, let $x^\uparrow \in \R^d$ be the representative of the equivalence class $[x]$ with increasingly ordered coordinates. 
Clearly, $x^\uparrow$ only depends on $[x]$ and thus we have an injective map $(~)^\uparrow : \cA_d(\R) \to \R^d$.
It is a right-inverse of $[~] : \R^d \to \cA_d(\R)$. 

\begin{lemma}[{\cite[Lemma 7.1]{Parusinski:2024aa}}] \label[l]{lem:unordered}
   We have 
   \[
       \dd_2([x],[y]) =  \|x^\uparrow - y^\uparrow \|_2, \quad x,y \in \R^d.
   \]
   In particular, $(~)^\uparrow : \cA_d(\R) \to \R^d$ and $[~] : \R^d \to \cA_d(\R)$ are Lipschitz maps.
\end{lemma}

It is easy to check that the map $H = (~)^\uparrow : \cA_d(\R) \to \R^d$ is an Almgren map (for $\R$ instead of $\C$,  
associated to the real linear form $\et=\id$) and \eqref{eq:combineq} is trivially true with $\La=\{\id\}$. 
Thus $\De=H = (~)^\uparrow : \cA_d(\R) \to \R^d$ is an Almgren embedding. 

Consequently, we can equip the space
\[
    W^{1,q}(U,\cA_d(\R)) = \{f : U \to \cA_d(\R) : f^\uparrow \in W^{1,q}(U,\R^d)\}
\]
with the metric 
\[
    \rh^{1,q}_\uparrow(f,g) :=  \|f^\uparrow - g^\uparrow\|_{W^{1,q}(U,\R^d)}.
\]

Thus \Cref{thm:eigW} can be seen as a special case of \Cref{thm:eiguW}.

%-----------------------------------------------------------------------------------------------------------------------
\subsection{Curves of class $W^{1,q}$ in $\cA_d(\C)$}  
%-----------------------------------------------------------------------------------------------------------------------

Let $I \subseteq \R$ be an open bounded interval.
We recall an equivalent definition of $W^{1,q} (I,\cA_d(\C))$ (see \cite[Definition 0.5]{De-LellisSpadaro11}) which is independent of Almgren's embedding.  

\begin{definition}[Intrinsic definition]\label[d]{d:W1p}
    A measurable function $f:I \ra\Iq$ belongs to $W^{1,q}(I,\cA_d(\C))$ ($1\leq q\leq\infty$) if there exists a function
    $\varphi \in L^q(I,\R_{\ge 0})$
    such that 
    \begin{itemize}
        \item[(i)] $x\mapsto \dd_2 (f(x),T)\in W^{1,q}(I)$ for all $T\in \Iq$;
        \item[(ii)] $\abs{(\dd_2 (f, T))'}\leq\varphi$ almost everywhere in $I$
            for all $T\in \Iq$.
    \end{itemize}
The minimal function $\tilde\varphi$ fulfilling (ii), that is, 
\begin{equation*}
    \tilde\varphi\leq\varphi \quad \text{almost everywhere for every other $\varphi$ satisfying (ii),}
\end{equation*}
is measurable and we denote it by $|D f|$.
\end{definition}

\begin{proposition}[{\cite[Proposition 1.2]{De-LellisSpadaro11}}]\label[p]{p:Wselection-1} 
    Let $f\in W^{1,q} (I, \Iq)$.
    Then,
    \begin{itemize}
        \item[$(a)$] $f\in AC(I,\Iq)$ and, moreover,
            $f\in C^{0,1-\frac{1}{q}}(I,\Iq)$ for $q>1$;
        \item[$(b)$] 
            there exists a parameterization 
            $f_1,\ldots,f_d\in W^{1,q}(I,\C)$ of $f$, i.e.,
            \[
                f = \a{f_1} + \cdots + \a{f_d},
            \]
            such that $\abs{Df_i} = \abs{f_i'}\leq\abs{Df}$ almost everywhere.
    \end{itemize}
\end{proposition}

Actually, the proof of \Cref{p:Wselection-1} in \cite{De-LellisSpadaro11} 
implies that $f \in W^{1,q}(I,\Iq)$ belongs to $AC^q(I,\Iq)$ in 
the sense of \Cref{ssec:ACq}. 
In the situation of \Cref{p:Wselection-1}, 
we will always mean without further mention that $f$ and $f_i$ are the continuous representatives.

Following \cite[Definition 1.9]{De-LellisSpadaro11}, we can also define the differential $Df$. 

\begin{definition}[Differential]\label[d]{d:diff}
    Let $f = \sum_i \a{f_i} : I\ra\Iq$ and $x_0\in I$. We say that $f$ is \emph{differentiable}
    at $x_0$ if there exist $d$ complex numbers $L_i$ satisfying:
    \begin{itemize}
        \item[(i)]$\dd_2(f(x),T_{x_0} f(x))=o(\abs{x-x_0})$, where
            \begin{equation}\label{e:taylor1st}
                T_{x_0} f(x):=\sum_i\a{f_i(x_0) + L_i\cdot(x-x_0)};
            \end{equation}
        \item[(ii)] $L_i=L_j$ if $f_i(x_0)=f_j(x_0)$.
    \end{itemize}
    The $d$-valued map $T_{x_0} f$
    is called the {\em first-order approximation} of $f$ at $x_0$.
    We denote $L_i$ by $Df_i(x_0)$ and the point $Df(x_0) := \sum_i \a{Df_i(x_0)} \in \Iq$
    is called the \emph{differential} of $f$ at $x_0$. 
\end{definition}

By \Cref{d:diff}(ii), the notation is consistent (see \cite[Remark 1.11]{De-LellisSpadaro11}):
if $g_1,\ldots, g_d$ is another parameterization of $f$, $f$ is differentiable at $x_0$, and $\si \in \on{S}_d$ is such that 
$g_i(x_0) = f_{\si(i)}(x_0)$ for all $1 \le i \le d$, then $Dg_i(x_0) = Df_{\si(i)}(x_0)$.

By \Cref{p:Wselection-1}, every $f\in W^{1,q} (I, \Iq)$ is differentiable almost everywhere and,
if $f_1,\ldots,f_d \in W^{1,q}(I,\C)$ is a parameterization of $f$, then $Df = \sum_i \a{f_i'}$ almost everywhere; 
see \cite[Section 3.4]{Parusinski:2024ab}.

%-----------------------------------------------------------------------------------------------------------------------
\subsection{A distance notion on $W^{1,q}(I,\Iq)$}
%-----------------------------------------------------------------------------------------------------------------------

The following definition is taken from \cite[Definition 3.6]{Parusinski:2024ab}.

\begin{definition}[The distance $\dd_E^{1,q}$] \label[d]{def:dd}
    Let $f,g \in W^{1,q}(I,\Iq)$ and let 
    \[
        f=\a{f_1} + \cdots + \a{f_d}, \quad g=\a{g_1} + \cdots + \a{g_d}
    \]
    be parameterizations of $f$, $g$ with $f_i,g_i \in W^{1,q}(I,\C)$ as in \Cref{p:Wselection-1}.
    Set
    \begin{equation*}
        \bs_0(f,g)(x) :=  \dd_2(f(x),g(x)).
    \end{equation*}
    Fix an arbitrary ordering of the elements of $\on{S}_d$.
    For $x \in I$, let
    \begin{equation*}
        \ta(x) := \min \Big\{\ta\in \on{S}_d : \Big(\sum_i | f_i(x) - g_{\ta(i)}(x)|^2\Big)^{1/2} =  \dd_2(f(x),g(x)) \Big\}.
    \end{equation*}
    For $x \in I$ such that $Df(x) = \sum_i \a{Df_i(x)}$ and $Dg(x)  = \sum_i \a{Dg_i(x)}$ exist in the sense of \Cref{d:diff},
    set 
    \begin{equation}\label{eq:maxorderings}
        \bs_1(f, g)(x) := \max  \Big(\sum_i | Df_i(x) - Dg_{\ta(x)(i)}(x)|^2\Big)^{1/2},
    \end{equation}
    where the maximum is taken over all orderings of $\on{S}_d$.
    By the remarks above, $\bs_1(f,g)(x)$ is defined for almost every $x \in I$ and 
    independent of the choices of parameterizations $f_1,\ldots,f_d$ and $g_1,\ldots,g_d$ of $f$ and $g$.

    For each measurable subset $E \subseteq I$, we define
    \begin{equation*}
        \mathbf{d}^{1,q}_{E}( f,g )
        := \|\bs_0(f,g)\|_{L^\infty(E)}
        + \| \bs_1(f,g) \|_{L^q(E)}
    \end{equation*}
    which is justified, since 
    the functions $\mathbf{s}_i(f,g) : I \to \R$, for $i=0,1$, are Borel measurable, by \cite[Lemma 3.7]{Parusinski:2024ab}.
\end{definition}

\begin{lemma}[{\cite[Lemma 3.8]{Parusinski:2024ab}}] \label[l]{lem:dd1qsemidist}
    Let $I \subseteq \R$ be a bounded open interval and $E \subseteq I$ a measurable set. 
    Let $f, g \in W^{1,q}(I,\Iq)$. Then:
    \begin{enumerate}
        \item $\mathbf{d}^{1,q}_{E}( f,f )=0$.
        \item $\mathbf{d}^{1,q}_{E}( f,g )=0$ implies $f = g$ on $E$.  
        \item $\mathbf{d}^{1,q}_{E}( f,g )=\mathbf{d}^{1,q}_{E}( g,f )$.
    \end{enumerate}
    In particular, $\dd^{1,q}_I$ is a semimetric on $W^{1,q}(I,\Iq)$.
\end{lemma}

%-----------------------------------------------------------------------------------------------------------------------
\subsection{Convergence in $W^{1,q}(I,\Iq)$}
%-----------------------------------------------------------------------------------------------------------------------

There is a notion of \emph{weak convergence} in $W^{1,q}(I,\Iq)$, see \cite[Definition 2.9]{De-LellisSpadaro11}, 
which is not appropriate for our purpose.
We introduce a stronger notion of convergence based on the semimetric $\dd^{1,q}_I$.

\begin{definition}[Strong convergence]\label[d]{d:convergence}
    Let $f_n, f\in W^{1,q}(I,\Iq)$. We say 
    that $f_n$ \emph{converges to $f$ in $W^{1,q}(I,\Iq)$} as $n \to \infty$,
    and we write $f_n\rightarrow f$,
    if
    \[ 
        \mathbf{d}^{1,q}_{I}( f, f_n )\ra 0 \quad \text { as }  n\ra\infty .
    \] 
\end{definition}

Strong convergence in $W^{1,q}(I,\Iq)$ is equivalent to convergence with respect to the 
topology induced by the metric $\rh^{1,q}_\De$ (see \eqref{eq:Almgren}), for some, equivalently every, 
Almgren embedding $\De$:

\begin{theorem}[{\cite[Theorem 3.11]{Parusinski:2024ab}}]\label[t]{thm:Almgren and convergence}
    Let $\De : \Iq \to \R^N$ be an Almgren embedding.  
    Then $f_n \to f$ in $W^{1,q}(I,\Iq)$ as $n \to \infty$ if and only if  
    \begin{equation*}
        \rh^{1,q}_\De(f,f_n) \to 0 \quad \text{ as } n \to \infty.
    \end{equation*}
\end{theorem}

In particular, the topology induced by the metric $\rh^{1,q}_\De$ 
on $W^{1,q}(I,\Iq)$ does not depend on the choice of the Almgren embedding $\De$. 
We shall see below that this is also true in several variables.

%-----------------------------------------------------------------------------------------------------------------------
\subsection{The topology on $W^{1,q}(U,\cA_d(\C))$ induced by $\rh^{1,q}_\De$ is independent of $\De$}
%-----------------------------------------------------------------------------------------------------------------------

The following theorem is an improved version of \cite[Theorem 10.2]{Parusinski:2024ab}.

\begin{theorem} \label[t]{t:topmult}
    Let $\De^i : \Iq \to \R^{N_i}$, for $i =1,2$, be two Almgren embeddings. 
    Let $U \subseteq \R^m$ be open and bounded and $1 \le q < \infty$.
    Let $f,f_n \in W^{1,q}(U,\Iq)$, for $n\ge 1$. 
    Then, 
    \begin{align*}
        \| \De^1 \o f - \De^1 \o f_n \|_{W^{1,q}(U,\R^{N_1})} &\to 0 
        \intertext{if and only if}  
        \| \De^2\o f - \De^2 \o f_n \|_{W^{1,q}(U,\R^{N_2})} &\to 0
    \end{align*}
    as $n \to \infty$. 
\end{theorem}

\begin{proof}
The map
\[
    \De^2 \o (\De^1)|_{\De^1(\Iq)}^{-1} : \De^1(\Iq) \to \R^{N_2}
\]
is Lipschitz and has a Lipschitz extension $\Ga$ to all of $\R^{N_1}$.
    Thus (see \cite{AmbrosioDalMaso90})  superposition with $\Ga$
    defines a bounded map from $W^{1,q}(U,\R^{N_1})$ to $W^{1,q}(U,\R^{N_2})$, where, setting $C:= \Lip(\Ga)$, for $F \in W^{1,q}(U,\R^{N_1})$, 
    \[
        \|\Ga \o F\|_2 \le C\, \|F\|_2, 
    \]
    because $\Ga(0) = 0$, and 
    \begin{equation} \label{eq:Ga}
        \|\nabla (\Ga  \o F)\|_2 \le C\, \|\nabla F\|_2 
    \end{equation} 
    almost everywhere in $U$.

   Set $F^i := \De^i \o f$ and $F^i_n := \De^i \o f_n$, for $i=1,2$.
   Assume that
   \begin{equation} \label{eq:F1}
    \|F^1-F^1_n\|_{W^{1,q}(U,\R^{N_1})} \to 0 \quad \text{ as } n \to \infty.
\end{equation}

Then $\{\|F^1\|_2^q+ F^1_n\|_2^q : n\ge 1\}$ and $\{\|\p_j F^1\|_2^q+\|\p_j F^1_n\|_2^q : n\ge 1\}$, for $1 \le j \le m$, are uniformly integrable 
subsets of $L^1(U)$. 
Since $F^2 = \Ga \o F^1$ and $F_n^2 = \Ga \o F_n^1$, we may conclude similarly as in the proof of \Cref{l:ui} 
(using \eqref{eq:Ga}) that the sets  
$\{\|F^2 - F^2_n\|_2^q : n \ge 1\}$ and $\{\|\p_j F^2 - \p_j F^2_n\|_2^q : n \ge 1\}$, for $1 \le j \le m$, are uniformly integrable.

Let us assume that $U = I_1 \times \cdots \times I_m$ is an open bounded box with sides parallel to the coordinate axes. Let $j=1$.
   For $x' \in  U'=I_2 \times \cdots \times I_m$ and $i=1,2$, 
   consider
   \begin{align*}
       A^i_n(x') &= \int_{I_1} \big\| \p_1 F^i(x_1,x') -  \p_1 F^i_n (x_1,x')\big\|_2^q\, dx_1, 
       \\
       B^i_n(x') &= \int_{I_1} \big\|  F^i(x_1,x') -  F^i_n (x_1,x')\big\|_2^q\, dx_1.
   \end{align*}
   By \eqref{eq:F1} and Tonelli's theorem,
   \[
       \int_{U'} A^1_n(x')\, dx' \to 0 \quad \text{ and }  \quad \int_{U'} B^1_n(x')\, dx' \to 0 \quad \text{ as } n \to \infty.
   \]
   Thus there is a subsequence $(n_k)$ such that $A^1_{n_k}(x') \to 0$ and $B^1_{n_k}(x') \to 0$ for almost every $x' \in U'$ as $k \to \infty$.
   For each such $x'$,
   \Cref{thm:Almgren and convergence} implies that $A^2_{n_k}(x') \to 0$ and $B^2_{n_k}(x') \to 0$ as $k \to \infty$. 
   
Tonelli's theorem and uniform integrability of the sets    
$\{\|F^2 - F^2_n\|_2^q : n\ge 1\}$ and $\{\|\p_j F^2 - \p_j F^2_n\|_2^q : n\ge 1\}$, for $1 \le j \le m$,  
imply, invoking de la Vall\'ee Poussin's criterion \ref{thm:VP}, that 
$\{A^2_n : n \ge 1\}$ and $\{B^2_n : n \ge 1\}$ are uniformly integrable. (See the proof of \Cref{thm:multhermW}.)

   Then Vitali's convergence theorem \ref{thm:Vitali} and Tonelli's theorem yield that
\[
    \|  F^2 -  F^2_{n_k} \|_{L^{q}(U,\R^{N_2})} \to 0 \quad \text{ and } \quad  \| \p_1 F^2 - \p_1 F^2_{n_k} \|_{L^{q}(U,\R^{N_2})} \to 0 
\]
as $k \to \infty$. Since the partial derivatives $\p_j$, for $2 \le j \le m$, can be treated in the same way, we 
have showed that there is a subsequence $(n_k)$ such that
\[
    \|  F^2 -  F^2_{n_k} \|_{W^{1,q}(U,\R^{N_2})} \to 0 \quad \text{ as } k \to \infty
\]
which implies 
\begin{equation*} 
    \|  F^2 -  F^2_{n} \|_{W^{1,q}(U,\R^{N_2})} \to 0 \quad \text{ as } n \to \infty,
\end{equation*}
in the case that $U = I_1 \times \cdots \times I_m$.

The case that $U$ is a general open bounded subset of $\R^m$ follows from \Cref{l:Vapp} 
(applied to $f_n := \|  F^2 -  F^2_{n} \|_2^q$ and $f_n := \|\p_j F^2   - \p_j F_{n}^2 \|_2^q$), 
since $U$ is a countable union of bounded open boxes with sides parallel to the coordinate axes.
\end{proof}

%-----------------------------------------------------------------------------------------------------------------------
\section{The characteristic map for normal matrices} \label{sec:charn}
%-----------------------------------------------------------------------------------------------------------------------

We are ready to introduce the characteristic map for normal matrices.

\begin{proposition} \label[p]{p:mcharnorm}
    Let $\De : \cA_d(\C) \to \R^N$ be an Almgren embedding.
    Let $1 \le q \le \infty$. 
    Let $U \subseteq \R^m$ be open and bounded.
Then the map
\[
    \eigu : W^{1,q}(U,\on{Norm}(d)) \to W^{1,q}(U,\cA_d(\C)), \quad A \mapsto \La \o A,
\]
is well-defined and bounded, satisfying
\begin{equation*}
    \|\De(\eigu(A)(x))\|_2 \le C\, \|A(x)\|_2 
    \quad \text{ and } \quad 
    \|\nabla (\De(\eigu(A)(x)))\|_2 \le C\, \|\nabla A(x)\|_2 
\end{equation*}
for almost every $x \in U$, where $C$ is the Lipschitz constant of $\De$.
\end{proposition}

\begin{proof}
    The map $\De \o \La : \on{Norm}(d) \to \cA_d(\C) \to \R^N$ is Lipschitz with Lipschitz constant $C=\Lip(\De)$, 
    by \Cref{p:HW}. It admits a Lipschitz extension $L : M_d(\C) \to \R^N$ 
    preserving the Lipschitz constant $C$, by Kirszbraun's theorem.

    It is well-known (see \cite{AmbrosioDalMaso90}) that superposition with $L$ defines a bounded map from $W^{1,q}(U,M_d(\C))$ to $W^{1,q}(U,\R^N)$, where 
    \[
        \|L\o A\|_2 \le C\, \|A\|_2, 
    \]
    because $L(0)= \De(\La(0)) = 0$, and 
    \[
        \|\nabla (L \o A)\|_2 \le C\, \|\nabla A\|_2 
    \] 
    almost everywhere in $U$.
    This implies the  assertion. 
\end{proof}

Let us recall an important relationship between the metric speed of an $AC$-curve in $\cA_d(\C)$ and 
the derivative of any parameterization.

\begin{lemma}[{\cite[Lemma 11.1]{Parusinski:2024ab}}] \label[l]{l:ms}
    Let $\la : I \to \C^d$ be an absolutely continuous curve and $\ga: I \to \cA_d(\C)$ defined by $\ga(x) = [\la(x)]$, for $x \in I$.
    Then the metric speed of $\ga$ is given by
    \[
        |\dot \ga|(x) = \|\la'(x)\|_2 \quad \text{ for almost every } x \in I.
    \]
\end{lemma}

On an interval $I$, we can express the boundedness of $\eigu : W^{1,q}(I,\on{Norm}(d)) \to W^{1,q}(I,\cA_d(\C))$ 
in terms of $\bs_0$ and $\bs_1$, 
introduced in \Cref{def:dd}.

\begin{proposition}
Let $1 \le q < \infty$ and $I \subseteq \R$ a bounded open interval.     
Then, for $A \in W^{1,q}(I,\on{Norm}(d))$,
\begin{align} \label{eq:charnorm1}
    \|\bs_0(\eigu(A),[0])\|_{L^\infty(I)} &\le  |I|^{-1/q}\, \|A\|_{L^q(I,M_d(\C))} + |I|^{1-1/q} \, \|A'\|_{L^q(I,M_d(\C))}, \qquad
    \\ \label{eq:charnorm2}
    \|\bs_1(\eigu(A),[0])\|_{L^q(I)} &\le  \|A'\|_{L^q(I,M_d(\C))}. 
\end{align}
\end{proposition}

\begin{proof}
    Let $A \in W^{1,q}(I,\on{Norm}(d))$. 
    Let $\la \in W^{1,q}(I,\C^d)$ be a parameterization of $\La := \eigu(A)$ (see \Cref{p:Wselection-1}).
    By \eqref{eq:supW}, for all $x \in I$,
    \begin{align*}
     \|\la(x)\|_2 = \|A(x)\|_2 \le    |I|^{-1/q}\, \|A\|_{L^q(I,M_d(\C))} + |I|^{1-1/q} \, \|A'\|_{L^q(I,M_d(\C))},
    \end{align*}
    which implies \eqref{eq:charnorm1}.
    By \Cref{l:ms} and \eqref{eq:HW}, for almost every $x \in I$,
    \begin{align*}   
        \|\la'(x)\|_2 = |\dot \La|(x) \le \|A'(x)\|_2, 
    \end{align*}
    entailing \eqref{eq:charnorm2}.
\end{proof}

%-----------------------------------------------------------------------------------------------------------------------
\section{Eigenvalue stability for normal matrices: one-parameter case} \label{sec:n1}
%-----------------------------------------------------------------------------------------------------------------------

In view of \Cref{thm:Almgren and convergence}, the following theorem is a version of \Cref{thm:eiguW} for $m=1$.
Recall the semimetric $\dd_I^{1,q}$ on $W^{1,q}(I,\cA_d(\C))$, introduced in \Cref{def:dd}.

\begin{theorem} \label[t]{thm:1eiguW}
    Let $1 \le q <\infty$.
    Let $I \subseteq \R$ be a bounded open interval.
    Let $A_n \to A$ in $W^{1,q}(I,\on{Norm}(d))$ as $n\to \infty$.  
    Then 
    \begin{equation*}
        \dd^{1,q}_I(\eigu(A),\eigu(A_n))   \to 0 \quad \text{ as } n \to \infty.
    \end{equation*}
\end{theorem}

\Cref{thm:1eiguW} and \Cref{thm:Almgren and convergence} yield the following corollary. 

\begin{corollary} \label[c]{c:1eiguW}
    Let $1 \le q <\infty$.
    Let $I \subseteq \R$ be a bounded open interval.
    Let $\De : \cA_d(\C) \to \R^N$ be an Almgren embedding.
    Let $A_n \to A$ in $W^{1,q}(I,\on{Norm}(d))$ as $n\to \infty$.  
    Then 
    \[
        \|\De \o \eigu(A) - \De \o \eigu(A_n)\|_{W^{1,q}(I,\R^N)} \to 0 \quad \text{ as } n \to \infty.
    \]
\end{corollary}

The proof of \Cref{thm:1eiguW} comprises \Cref{ssec:preln} and \Cref{ssec:hatd}.

%-----------------------------------------------------------------------------------------------------------------------
\subsection{Preliminary observations and reductions} \label{ssec:preln}
%-----------------------------------------------------------------------------------------------------------------------

Let $1 \le q <\infty$.
Let $I \subseteq \R$ be a bounded open interval and
assume that $A_n \to A$ in $W^{1,q}(I,\on{Norm}(d))$ as $n\to \infty$.  

By \Cref{p:mcharnorm}, $\La := \eigu(A)$ and $\La_n := \eigu(A_n)$ belong to $W^{1,q}(I,\cA_d(\C))$.
Let $\la=(\la_1,\ldots,\la_d) \in W^{1,q}(I,\C^d)$
and $\la_n=(\la_{n,1},\ldots,\la_{n,d}) \in W^{1,q}(I,\C^d)$
be parameterizations of $\La$ and $\La_n$, respectively; see \Cref{p:Wselection-1}.
Then, for $1 \le i \le d$, $n\ge 1$, and
almost every $x \in I$,
\[
    D\la_i(x) = \la_i'(x) \quad \text{ and } \quad  D\la_{n,i}(x) = \la_{n,i}'(x);
\]
see \Cref{d:diff}.

By \Cref{p:HW} and \Cref{l:mat},
\begin{align*}
    \dd_2(\La(x),\La_n(x)) &\le \|A(x)-A_n(x)\|_2 
    \\
                           &\le   |I|^{-1/q}\, \|A-A_n\|_{L^q(I,M_d(\C))} + |I|^{1-1/q} \, \|A'-A_n'\|_{L^q(I,M_d(\C))}
\end{align*}
so that
\begin{equation} \label{eq:s0}
    \|\mathbf{s}_{0}(\La,\La_n)\|_{L^\infty(I)}  
    = \| \dd_2( \La , \La_n ) \|_{L^\infty(I)} \to 0 \quad \text{ as } n \to \infty. 
\end{equation}
Thus, to complete the proof of \Cref{thm:1eiguW}, it suffices to show that
\begin{equation} \label{eq:nts}
    \ddd^{1,q}_{I}(\La,\La_n)  \to 0 \quad \text{ as } n \to \infty,
\end{equation}
where, for each measurable set $E \subseteq I$, we define
\begin{align*} 
    \ddd^{1,q}_{E}(\La,\La_n) &:= \|\bs_1(\La,\La_n)\|_{L^q(E)};
\end{align*}
see \eqref{eq:maxorderings} for the definition of $\bs_1$.

With $A \in W^{1,q} (I, \on{Norm}(d))$ we associate 
$\tilde A:=A-\frac{1}{d}\on{Tr}(A)\I_d \in W^{1,q} (I, \on{Norm}_T(d))$.

\begin{lemma} \label[l]{lem:redTr}
    Let $1 \le q < \infty$. 
    Let $I \subseteq \R$ be a bounded open interval.
    Let $A_n \to A$ in $W^{1,q}(I,\on{Norm}(d))$ as $n \to \infty$. 
    Then:
    \begin{enumerate}
        \item $\tilde A_n \to \tilde A$ in $W^{1,q}(I,\on{Norm}_T(d))$ as $n \to \infty$.
    \end{enumerate}    
    Let $\La, \La_n, \tilde \La, \tilde \La_n : I \to \Iq$ be the curves of unordered eigenvalues 
    of $A,A_n,\tilde  A, \tilde A_n$,
    respectively. Then:
    \begin{enumerate}
        \item[(2)] $\ddd^{1,q}_I(\La,\La_n) \to 0$ if and only if\hspace{1mm} 
            $\ddd^{1,q}_I(\tilde \La, \tilde \La_n) \to 0$ as $n \to \infty$.
    \end{enumerate}
\end{lemma}

\begin{proof}
    (1) This is clear since $A_n \to A$ in $W^{1,q}(I,\on{Norm}(d))$ implies 
    $\on{Tr}(A_n) \to \on{Tr}(A)$ in $W^{1,q}(I,\C)$ as $n \to \infty$.

    (2) 
    By \eqref{eq:s0} and (1), $\|\dd_2(\La,\La_n)\|_{L^\infty(I)} \to 0$ 
    as well as $\|\dd_2(\tilde \La,\tilde \La_n)\|_{L^\infty(I)} \to 0$ as $n \to \infty$.
    Assume that  $\ddd^{1,q}_I(\La,\La_n) \to 0$ as $n \to \infty$.
    By \Cref{thm:Almgren and convergence},
    we have 
    \[
        \|\De \o \La - \De \o \La_n\|_{W^{1,q}(I,\R^N)} \to 0 \quad \text{ as } n \to \infty,
    \]
for every Almgren embedding $\De : \Iq \to \R^N$. 
Let $H : \Iq \to \R^d$ be an Almgren map with associated real linear form $\et$ (see \Cref{def:Amap}).
    Then
    \begin{align*}
        H \o \La - H \o \tilde \La &=  \tfrac{1}{d} \et(\on{Tr}(A))(1,1,\ldots,1),
        \\
        H \o \La_n - H \o \tilde \La_n &=  \tfrac{1}{d} \et(\on{Tr}(A_n))(1,1,\ldots,1).
    \end{align*}
    Thus $\|H \o \tilde \La - H\o \tilde \La_n\|_{W^{1,q}(I,\R^d)} \to 0$ and, consequently, 
    in view of \eqref{eq:AlmgrenDE},
    \[
        \|\De \o \tilde \La - \De \o \tilde \La_n\|_{W^{1,q}(I,\R^N)} \to 0 \quad \text{ as } n \to \infty,
    \]
    which implies $\dd^{1,q}_I(\tilde \La,\tilde \La_n) \to 0$ and hence $\ddd^{1,q}_I(\tilde \La,\tilde \La_n) \to 0$, 
    again by \Cref{thm:Almgren and convergence}. 
    The opposite direction follows from the same arguments.
\end{proof}

By \Cref{lem:redTr}, we may assume that all matrices $A,A_n$ are traceless.

%-----------------------------------------------------------------------------------------------------------------------
\subsection{Proof of \eqref{eq:nts}} \label{ssec:hatd}
%-----------------------------------------------------------------------------------------------------------------------

We will proceed by induction on $d$. If $d=1$ then 
\[
    \ddd^{1,q}_I (\La,\La_n) = \|A' - A_n'\|_{L^q(I,M_1(\C))}  \to 0 \quad \text{ as } n \to \infty.  
\]
So assume $d\ge 2$.

On the zero set $Z_A$ of $A$ we have the following lemma.

\begin{lemma} \label[l]{l:Zn}
    Let $1 \le q < \infty$.
    Let $I \subseteq \R$ be a bounded open interval.
    Assume that $A_n \to A$ in $W^{1,q}(I, \on{Norm}(d))$ as $n \to \infty$.
    Let $\La := \eigu(A)$ and $\La_n := \eigu(A_n)$.
    Then
            \begin{equation} \label{eq:dddZ}
                \ddd^{1,q}_{Z_A}(\La, \La_n) \to 0 \quad \text{ as } n\to \infty. 
            \end{equation}
\end{lemma}

\begin{proof}
    Let $\la, \la_n \in W^{1,q} (I,\C^d)$ be parameterizations of $\La,\La_n$, respectively; see \Cref{p:Wselection-1}.
    Let $E$ be the set of $x \in \on{acc}(Z_A)$, where the derivatives $A'(x)$, $\la'(x)$ and $A_n'(x)$, $\la_n'(x)$ for all $n\ge 1$ exist.
    For each $x \in E$, $A'(x)=0$, $\la'(x) = 0$, and, by \Cref{l:ms} and \eqref{eq:HW},
    \[ 
        \big\| \la_n'(x) \big\|_2  = |\dot \La_n|(x) \le \|A_n'(x)\|_2.
    \]
    Since $E$ has full measure in $Z_A$, we thus get 
    \begin{align*}
        \ddd^{1,q}_{Z_A}([\la], [\la_n]) &= \|\bs_1([\la], [\la_n]) \|_{L^q(Z_A)} = \big\|\|\la_n'\|_2 \big\|_{L^q(Z_A)}
        \\
                                         &\le \big\|\|A_n'\|_2 \big\|_{L^q(Z_A)} = \|A' - A_n'\|_{L^q(Z_A,M_d(\C))}
    \end{align*}
    which implies the assertion.
\end{proof}

Let us fix a uniform unitary block-diagonalization $(\{U_i : \cV_i \to U_d(\C)\}_{i=1}^s,r,\ch)$ for $\on{Norm}_T^0(d)$; see \Cref{d:UBD}.

\begin{lemma} \label[l]{l:rednorm}
    Let $1 \le q < \infty$.
    Let $I \subseteq \R$ be a bounded open interval.     
    Assume that $A_n \to A$ in $W^{1,q} (I, \on{Norm}_T(d))$ as $n \to \infty$.
    Let $x_0 \in I\setminus Z_A$. 
    Then there exist an open interval $J$ with $x_0 \in J \subseteq I \setminus Z_A$, 
    $n_0 \ge 1$, and $i \in \{1,\ldots,s\}$ such that 
    \begin{enumerate}
        \item for all $n \ge n_0$, the curves
            $\ul A := A/\|A\|_2 $ and $\ul A_n := A_n/\|A_n\|_2$ belong to $W^{1,q}(J, \on{Norm}_T^0(d))$ 
            and satisfy
            $\ul A(J), \ul A_n(J) \subseteq \cV_i$; 
        \item on $J$ and for all $n \ge n_0$, we have  
            \begin{align} \label{eq:simBDn}
                U_i^* (\ul A) A \, U_i(\ul A) 
                =   
                \begin{pmatrix}
                    B & 0 \\
                    0 &  C
                \end{pmatrix}, \quad
                U_i^* (\ul A_n) A_n \, U_i(\ul A_n) 
                =   
                \begin{pmatrix}
                    B_n & 0 \\
                    0 &  C_n
                \end{pmatrix},
            \end{align}
            where $B,B_n \in W^{1,q}(J,\on{Norm}(d_1))$ and $C,C_n \in W^{1,q}(J,\on{Norm}(d_2))$ with $d_1+d_2 = d$;
        \item for all $x \in J$ and $n \ge n_0$, we have 
            \begin{equation} \label{eq:sep4}
                \big|\|A_n(x)\|_2\, \mu(x) - \|A(x)\|_2\, \nu_n(x)\big| > \ch\, \|A(x)\|_2 \|A_n(x)\|_2
            \end{equation}
            for all eigenvalues $\mu(x)$ of $B(x)$ and all eigenvalues $\nu_n(x)$ of $C_n(x)$, 
            where $B$ and $C_n$ are defined by \eqref{eq:simBDn};
        \item we have 
            \begin{equation} \label{eq:convBDn}
                \|B - B_n\|_{W^{1,q}(J,M_{d_1}(\C))} \to 0, \quad 
                \|C - C_n\|_{W^{1,q}(J,M_{d_2}(\C))} \to 0 
            \end{equation}
            as $n \to \infty$.  
    \end{enumerate}
\end{lemma}

\begin{proof}
    The proof is almost identical to the one of \Cref{l:red}. 
    (In particular, \eqref{eq:L1clW}--\eqref{eq:clnW} hold 
    and we will use them in the proof of \Cref{l:respect} below.)
    Item (3) follows from \eqref{eq:sep3}.
\end{proof}

The unitary block-diagonalization \eqref{eq:simBDn} will permit us to use the induction hypothesis.

Consider the parameterizations $\la,\la_n \in W^{1,q}(I,\C^d)$ of $\La = \eigu(A)$ and $\La_n = \eigu(A_n)$.
In the situation of \Cref{l:rednorm},
by reordering (independently of $x$) if necessary, we may assume that 
$\La^1 = [\la_1(x),\ldots,\la_{d_1}(x)]$ is the unordered $d_1$-tuple of the eigenvalues of $B(x)$ and 
$\La^2 = [\la_{d_1+1}(x),\ldots,\la_{d}(x)]$ is the unordered $d_2$-tuple of the eigenvalues of $C(x)$, for $x \in J$.
Analogously, we may assume that
$\La^1_n = [\la_{n,1}(x),\ldots,\la_{n,d_1}(x)]$ is the unordered $d_1$-tuple of the eigenvalues of $B_n(x)$ and 
$\La^2_n = [\la_{n,d_1+1}(x),\ldots,\la_{n,d}(x)]$ is the unordered $d_2$-tuple of the eigenvalues of $C_n(x)$, for $x \in J$ and $n \ge n_0$.
Since the reordering is independent of $x$, $\la$ and $\la_n$ remain parameterizations in $W^{1,q}(J,\C^d)$ of 
$\La$ and $\La_n$ on $J$, respectively.

By \Cref{lem:redTr}, \Cref{l:rednorm}, and the induction hypothesis,
\[
    \ddd^{1,q}_J(\La^1,\La^1_n) \to 0 
    \quad \text{ and } \quad 
    \ddd^{1,q}_J(\La^2,\La^2_n) \to 0 
    \quad \text{ as } n \to \infty,
\]
where $\ddd^{1,q}_J([\La^j],[\La^j_n])$ is interpreted in dimension $d_j$, for $j=1,2$.
Thanks to the following lemma, this implies 
\begin{equation} \label{eq:Jnorm}
    \ddd^{1,q}_J(\La,\La_n) \to 0 
    \quad \text{ as } n \to \infty.
\end{equation}

\begin{lemma} \label[l]{l:respect}
    In the situation of \Cref{l:rednorm}, after possibly 
    increasing $n_0$, for all $x \in J$ and $n \ge n_0$,
    the following holds. 
    If $\ta \in \on{S}_d$ satisfies 
    \begin{equation} \label{eq:taid}
        \|\la(x) - \ta \la_n(x)\|_2 = \dd_2([\la(x)],[\la_n(x)])
    \end{equation}
    then $\ta$ respects \eqref{eq:simBDn} in the sense that 
    $\ta(\{1,\ldots,d_1\}) =\{1,\ldots,d_1\}$
   and (consequently) $\ta(\{d_1+1,\ldots,d\}) =\{d_1+1,\ldots,d\}$.
\end{lemma}

\begin{proof}
    By \eqref{eq:sep4}, for all $x \in J$, $n\ge n_0$, $1 \le i \le d_1$, and $d_1+1\le j \le d$,
    \begin{equation} \label{eq:sep5}
        \big|\|A_n(x)\|_2\, \la_i(x) - \|A(x)\|_2\, \la_{n,j}(x)\big| > \ch \, \|A(x)\|_2 \|A_n(x)\|_2.
    \end{equation}
    By \eqref{eq:cl1W}, \eqref{eq:cl2W}, and \eqref{eq:clnW},
    \[
        \|A(x)\|_2 \|A_n(x)\|_2 \ge \frac{1}{8}\, \|A(x_0)\|_2^2
    \]
    so that, by \eqref{eq:sep5}, 
    \begin{align*}
       \frac{\ch}{8}\, \|A(x_0)\|_2^2 &< \big|\|A_n(x)\|_2\, \la_i(x) - \|A(x)\|_2\, \la_{n,j}(x)\big|
       \\
                                      &\le \big|\|A_n(x)\|_2 - \|A(x)\|_2\big| |\la_i(x)| + |\la_i(x) - \la_{n,j}(x)| \|A(x)\|_2.
    \end{align*}
    By \eqref{eq:cl2W}, 
    \[
        |\la_i(x)| \le    \|A(x)\|_2 \le \frac{3}{2} \, \|A(x_0)\|_2.
    \]
    Since $\|A -A_n\|_{L^\infty(I,M_d(\C))} \to 0$ as $n \to \infty$, by \Cref{c:supW}, 
    \[
        \big|\|A_n(x)\|_2 - \|A(x)\|_2\big| \le \frac{\ch}{24}\, \|A(x_0)\|_2
    \]
    for all $x \in J$ if $n$ is large enough. 
    Thus
    \[
       \frac{\ch}{8}\, \|A(x_0)\|_2^2 < \frac{\ch}{16}\, \|A(x_0)\|_2^2 +  \frac{3}{2} \, \|A(x_0)\|_2|\la_i(x) - \la_{n,j}(x)| 
    \]
    which implies
    \begin{equation} \label{eq:lb}
       \frac{\ch}{24}\, \|A(x_0)\|_2 <  |\la_i(x) - \la_{n,j}(x)|,
    \end{equation}
    for all $x \in J$ and all sufficiently large $n$, say for all $n \ge n_1$ where $n_1 \ge n_0$.

    By \eqref{eq:s0}, there exists $n_2 \ge n_1$ such that 
    \begin{equation} \label{eq:ub}
        \dd_2([\la(x)],[\la_n(x)]) < \frac{\ch}{24}\, \|A(x_0)\|_2,
    \end{equation}
    for all $x \in J$ and $n \ge n_2$.
   
    Assume that $\ta \in \on{S}_d$ is such that $\ta(\{1,\ldots,d_1\}) \ne \{1,\ldots,d_1\}$.    
    Then there exist $i, j$ with $1 \le i \le d_1 < j \le d$ and $\ta(i) = j$.

    Fix $x \in J$ and $n \ge n_2$ and 
    suppose that $\ta$ satisfies \eqref{eq:taid}. 
    By \eqref{eq:lb} and \eqref{eq:ub},
    \begin{align*}
        \frac{\ch}{24}\, \|A(x_0)\|_2 < |\la_i(x)- \la_{n,j}(x)| &\le \|\la(x) - \ta \la_n(x)\|_2
      \\
       &= \dd_2([\la(x)],[\la_n(x)]) < \frac{\ch}{24}\, \|A(x_0)\|_2,
    \end{align*}
    a contradiction. The lemma is proved.
\end{proof}

By \Cref{l:Zn} and \eqref{eq:Jnorm}, the interval $I$ 
can be covered by countably many measurable sets $F_i$ such that 
\begin{equation*} 
\ddd^{1,q}_{F_i}(\La,\La_n) \to 0 
    \quad \text{ as } n \to \infty.
\end{equation*}
By \Cref{l:Vapp} (applied to $f_n = \bs_1(\La,\La_n)^q$), we conclude \eqref{eq:nts}. 
Indeed, $\{\bs_1(\La,\La_n)^q : n \ge 1\} \subseteq L^1(I)$ is uniformly integrable
which can be checked following the arguments in the proof of \Cref{l:ui} and noting that 
\[
    \bs_1(\La,\La_n) \le \|\la'\|_2 + \|\la_n'\|_2 \le \|A'\|_2 + \|A_n'\|_2,
\]
almost everywhere in $I$,
by \Cref{l:ms} and \eqref{eq:HW}.

This completes the induction and hence the proof of \Cref{thm:1eiguW}.

%-----------------------------------------------------------------------------------------------------------------------
\subsection{Variants of \Cref{thm:1eiguW}} \label{ssec:variants}
%-----------------------------------------------------------------------------------------------------------------------

\Cref{thm:main1varW} and \Cref{thm:main1var} are versions of  \Cref{thm:main1varWm} and \Cref{thm:main1varm} for $m=1$,
respectively.

\begin{theorem} \label[t]{thm:main1varW}
    Let $1 \le q < \infty$.
    Let $I \subseteq \R$ be a bounded open interval.
    Let $A_n \to A$ in $W^{1,q}(I,\on{Norm}(d))$ as $n\to \infty$.  
    Assume that $\la,\la_n \in W^{1,q}(I, \C^d)$ are parameterizations 
    of the eigenvalues of $A,A_n$, respectively, and that $\lim_{n\to \infty} \la_n(x_0)= \la(x_0)$, for some $x_0 \in I$.
    Then $\la_n \to \la$ in $L^\infty(I,\C^d)$ 
    if and only if
    $\la_n' \to \la'$ in $L^q(I,\C^d)$
    as $n \to \infty$.
\end{theorem}

\begin{proof}
    Let us assume that $\la_n \to \la$ in $L^\infty(I,\C^d)$ as $n \to \infty$.
    The proof (by induction on $d$) that then $\la_n' \to \la'$ in $L^q(I,\C^d)$
    as $n \to \infty$ follows from the arguments in the proof of \Cref{thm:1eiguW} with 
    slight modifications. It is easy to adjust \Cref{lem:redTr} and \Cref{l:Zn}, while 
    \Cref{l:rednorm} can be used unchanged. 
    In the situation of \Cref{l:rednorm}, there exist
    $\mu,\mu_n \in W^{1,q}(J,\C^{d_1})$ and $\nu,\nu_n \in W^{1,q}(J,\C^{d_2})$ 
    such that 
    \begin{align} \label{eq:munu}
        \la = (\mu,\nu)  \quad \text{ and } \quad \la_n = (\mu_n,\nu_n).
    \end{align}
    Since $\la_n \to \la$ in $L^\infty(I,\C^d)$, we have 
    $\mu_n \to \mu$ in $L^\infty(I,\C^{d_1})$ and $\nu_n \to \nu$ in $L^\infty(I,\C^{d_2})$ as $n \to \infty$.
    The induction hypothesis implies that there is countable cover $\{F_i\}$ of $I$ by measurable sets such that, for all $F_i$,  
    \[
        \|\la' -\la_n'\|_{L^q(F_i,\C^d)} \to 0 \quad \text{ as } n \to \infty
    \]
    which implies $\la_n' \to \la'$ in $L^q(I,\C^d)$, thanks to \Cref{l:Vapp}.

    Conversely, assume that $\la_n' \to \la'$ in $L^q(I,\C^d)$ as $n \to \infty$.
    Then, for $x \in I$,  
    \begin{align*}
        \| \la(x) - \la_n(x)\|_2  &= \Big\| \la(x_0) - \la_n(x_0) + \int_{x_0}^x \la'(t) - \la_n'(t) \, dt \Big\|_2
        \\
                                  &\le \| \la(x_0) - \la_n(x_0) \|_2 + \|\la'- \la_n'\|_{L^1(I,\C^d)}
        \\
                                  &\le \| \la(x_0) - \la_n(x_0) \|_2 + |I|^{1-1/q} \|\la'- \la_n'\|_{L^q(I,\C^d)}
    \end{align*}
    which implies $\la_n \to \la$ in $L^\infty(I,\C^d)$ as $n \to \infty$.
\end{proof}

\begin{theorem} \label[t]{thm:main1var}
    Let $I \subseteq \R$ be a bounded open interval.
    Let $A_n \to A$ in $C^{0,1}(\ol I,\on{Norm}(d))$ as $n\to \infty$.  
    Assume that $\la,\la_n : I \to \C^d$ are continuous (thus Lipschitz) parameterizations 
    of the eigenvalues of $A,A_n$, respectively, and that $\lim_{n\to \infty} \la_n(x_0)= \la(x_0)$, for some $x_0 \in I$.
    Then $\la_n \to \la$ in $L^\infty(I,\C^d)$ 
    if and only if
    $\la_n' \to \la'$ almost everywhere in $I$
    as $n \to \infty$.
\end{theorem}

\begin{proof}
    Every continuous parameterization $\la$, $\la_n$ of the eigenvalues of $A$, $A_n$, respectively,
    is actually Lipschitz, by \cite[Proposition 6.3]{RainerN}, \eqref{eq:HW}, and \Cref{l:ms}.  

    If $\la_n' \to \la'$ almost everywhere in $I$, then $\la_n' \to \la'$ in $L^q(I,\C^d)$ as $n \to \infty$, by the dominated convergence theorem.
    So $\la_n \to \la$ in $L^\infty(I,\C^d)$ as $n \to \infty$, by \Cref{thm:main1varW}.

    Assume that $\la_n \to \la$ in $L^\infty(I,\C^d)$ as $n \to \infty$.
    We see that then $\la_n' \to \la'$ almost everywhere 
    as $n \to \infty$  
    by adjusting the arguments in the proof of \Cref{thm:1eiguW}.
    It is easy to adjust \Cref{lem:redTr} and  
    \Cref{l:rednorm}. We replace \Cref{l:Zn} by \Cref{l:ZnLip}.
    Here we have \eqref{eq:munu} with 
    $\mu,\mu_n \in C^{0,1}(\ol J,\C^{d_1})$ and $\nu,\nu_n \in C^{0,1}(\ol J,\C^{d_2})$. Hence,
    $\mu_n \to \mu$ in $L^\infty(I,\C^{d_1})$ and $\nu_n \to \nu$ in $L^\infty(I,\C^{d_2})$ as $n \to \infty$.
    By the induction hypothesis and \Cref{l:ZnLip}, we may conclude that 
    $\la_n' \to \la'$ almost everywhere in $I$
    as $n \to \infty$.
\end{proof}

\begin{lemma} \label[l]{l:ZnLip}
    Let $I \subseteq \R$ be a bounded open interval
    and assume that $A_n \to A$ in $C^{0,1} (\ol I, \on{Norm}(d))$ as $n \to \infty$.
    Let $\la_n \in C^{0,1} (\ol I,\C^d)$ be a parameterization of $\La_n = \eigu(A_n)$.
    Then, for almost every $x \in Z_A$, 
    \begin{equation} \label{eq:ptw4}
        \la_n'(x)  \to  0 \quad \text{ as } n\to \infty. 
    \end{equation}
\end{lemma}

\begin{proof}
    This is a simple modification of the proof of \Cref{l:Zp}.
    By \Cref{p:HW}, \Cref{l:slope1}, and \Cref{l:slope2}, 
    \[
        \sup_{(x,y) \in Z_A^{<2>}} \frac{\dd_2(\La_n(x),\La_n(y))}{|x-y|} \le \sup_{(x,y) \in Z_A^{<2>}} \frac{\|A_n(x)-A_n(y)\|_2}{|x-y|} \to 0
    \]
    as $n \to \infty$.
    Let $E$ be the set of $x \in \on{acc}(Z_A)$, where the derivatives $\la_n'(x)$ for all $n\ge 1$ exist.
    For each $x_0 \in E$ there is a sequence $x_k \to x_0$ with $x_k \in Z_A$ and $x_k \ne x_0$ for all $k \ge 1$. Thus, by \Cref{l:ms}, for fixed $n$, 
    \begin{align*}
        \| \la_n'(x_0) \|_2   = |\dot \La_n|(x_0)
        = \lim_{k \to \infty} \frac{\dd_2(\la_n(x_0), \la_n(x_k))}{|x_0-x_k|}
       \le \sup_{(x,y) \in Z_A^{<2>}} \frac{\dd_2(\la_n(x) , \la_n(y))}{|x-y|}.
    \end{align*}
    This implies \eqref{eq:ptw4}, since $E$ has full measure in $Z_A$.
\end{proof}

%-----------------------------------------------------------------------------------------------------------------------
\subsection{Proofs of \Cref{thm:mainse} and \Cref{cor:main1}} \label{ssec:mse}
%-----------------------------------------------------------------------------------------------------------------------

Let $1 \le q < \infty$.
Let $I \subseteq \R$ be a bounded open interval.
Let $A_n \to A$ in $W^{1,q}(I,\on{Norm}(d))$ as $n\to \infty$.  
Let $\la,\la_n \in W^{1,q} (I, \C^d)$ be parameterizations of the eigenvalues of $A$, $A_n$, respectively.

Fix an ordering of $\on{S}_d$ and let $\ta(x) \in \on{S}_d$ be as in \Cref{def:dd}. 
Then
\begin{align*}
    \big| \|\la'\|_2 - \|\la_n'\|_2\big|
    = \big| \|\la'\|_2 - \|\ta\la_n'\|_2\big|
    \le \|\la' - \ta\la_n'\|_2 \le \bs_1([\la],[\la_n])
\end{align*}
almost everywhere in $I$. By \Cref{thm:1eiguW}, this easily implies \Cref{cor:main1}.

Thanks to \Cref{l:ms},
\Cref{thm:mainse} follows from \Cref{cor:main1} and \eqref{eq:s0}.

%-----------------------------------------------------------------------------------------------------------------------
\section{Eigenvalue stability for normal matrices: multiparameter case} \label{sec:normmult}
%-----------------------------------------------------------------------------------------------------------------------

In this section, we prove \Cref{thm:eiguW}, \Cref{cor:eigu}, \Cref{thm:main1varWm}, and \Cref{thm:main1varm}.

%-----------------------------------------------------------------------------------------------------------------------
\subsection{Proof of \Cref{thm:eiguW}}
%-----------------------------------------------------------------------------------------------------------------------

Let $1 \le q < \infty$. Let $U \subseteq \R^m$ be open and bounded.
Let $\De : \cA_d(\C) \to \R^N$ be an Almgren embedding.
Let $A_n \to A$ in $W^{1,q}(U,\on{Norm}(d))$ as $n \to \infty$.
Set $\La := \eigu(A), \La_n := \eigu(A_n) : U \to \cA_d(\C)$ 
and $F:= \De \o \La, F_n:= \De \o \La_n : U \to \R^N$.
We have to show that 
\begin{equation*}
    \|F - F_n \|_{W^{1,q}(U,\R^N)} \to 0 \quad \text{ as } n \to \infty.
\end{equation*}

By \eqref{p:HW}, for $x \in U$, 
\begin{align*}
  \|F(x) - F_n(x)\|_2 \le \Lip(\De)\,\dd_2(\La(x), \La_n(x)) \le \Lip(\De)\,\|A(x) - A_n(x)\|_2 
\end{align*}
so that 
\begin{equation*}
    \|F - F_n \|_{L^{q}(U,\R^N)} \to 0 \quad \text{ as } n \to \infty.
\end{equation*}
It remains to show 
\begin{equation} \label{eq:pjF}
    \|\p_j F - \p_j F_n \|_{L^{q}(U,\R^N)} \to 0 \quad \text{ as } n \to \infty,
\end{equation}
for all $1 \le j \le m$. It is enough to prove that there is a subsequence $(n_k)$ with this property. 

Let us first assume that $U = I_1 \times \cdots \times I_m$ is an open box with 
sides parallel to the coordinate axes. 
Let $j=1$.
As in the proof of \Cref{thm:multhermW}, 
we conclude that there is a subsequence $(n_k)$ such that 
\[
    \int_{I_1} \|A(x_1,x') - A_{n_k}(x_1,x')\|_2^q\, dx_1  \to 0 
\]
and
\[
    \int_{I_1} \|\p_1 A(x_1,x') - \p_1 A_{n_k}(x_1,x')\|_2^q\, dx_1  \to 0
\]
for almost every $x' \in U'=I_2 \times \cdots \times I_m$ as $k \to \infty$. 
By \Cref{c:1eiguW},
\[
    G_{1,n_k}(x') :=  \int_{I_1} \|\p_1 F(x_1,x') - \p_1 F_{n_k}(x_1,x')\|_2^q\, dx_1  \to 0 \quad \text{ as } k \to \infty
\]
for almost every $x' \in U'$. 

Using \Cref{p:mcharnorm} instead of \Cref{p:mcharHerm}, we see as in the proof of \Cref{thm:multhermW}, that the set 
$\{G_{1,n_k} : k \ge 1\} \subseteq L^1(U')$ is uniformly integrable. 
So Vitali's convergence theorem \ref{thm:Vitali} and Tonelli's theorem imply \eqref{eq:pjF} for $j=1$,
in the case that $U = I_1 \times \cdots \times I_m$. 
For $2 \le j \le m$, the reasoning is analogous.

Let now $U$ be a general open bounded subset of $\R^m$. 
Then \eqref{eq:pjF} follows from \Cref{l:Vapp} (applied to $f_n:=\|\p_j F - \p_j F_{n} \|_2^q$), 
since $\{f_n\}$ is uniformly integrable
which can be seen in analogy to the proof of \Cref{l:ui}, using \Cref{p:mcharnorm}.
This ends the proof of \Cref{thm:eiguW}.

%---------------------------------------------------------------------------------------------
\subsection{Proof of \Cref{cor:eigu}}
%---------------------------------------------------------------------------------------------

The corollary follows from \Cref{thm:eiguW} and Morrey's inequality,
\[
    \|\De \o \eig(A)-  \De \o \eig(A_n)\|_{C^{0,\al}(\ol U,\R^N)} 
    \le C\, \|\De \o \eig(A) -  \De \o \eig(A_n)\|_{W^{1,q}(U,\R^N)},
\]
where $\al = 1-m/q$, $q> m$, and $C=C(m,N,q,U)$.

%-----------------------------------------------------------------------------------------------------------------------
\subsection{Proof of \Cref{thm:main1varWm}}
%-----------------------------------------------------------------------------------------------------------------------

We adapt the proof of \eqref{eq:pjF} with $F,F_n$ replaced by $\la,\la_n$.
Assume first that $U=I_1 \times \cdots \times I_m$ and $j=1$.
By \Cref{thm:main1varW}, we may conclude that there is a subsequence $(n_k)$ of $(n)$ such that 
\[
    G_{1,n_k}(x') :=  \int_{I_1} \|\p_1 \la(x_1,x') - \p_1 \la_{n_k}(x_1,x')\|_2^q\, dx_1  \to 0 
    \quad \text{ as } k \to \infty
\]
for almost every $x' \in U'$. It is easy to check (as before) that $\{G_{1,n_k}: k\ge 1\}$ is uniformly integrable. 
So the assertion follows from Vitali's convergence theorem \ref{thm:Vitali} and Tonelli's theorem.
The case of general $U$ then follows from \Cref{l:Vapp}.

%-----------------------------------------------------------------------------------------------------------------------
\subsection{Proof of \Cref{thm:main1varm}}
%-----------------------------------------------------------------------------------------------------------------------

It follows easily from \Cref{thm:main1var} (applied coordinate-wise).

%-----------------------------------------------------------------------------------------------------------------------
\section{An application for compact self-adjoint operators} \label{sec:cop}
%-----------------------------------------------------------------------------------------------------------------------

Let $H$ be a Hilbert space.
A bounded operator $A \in L(H)$ is \emph{compact} if for each bounded sequence $(v_n) \subseteq H$ the image $(Av_n)$ 
contains a convergent subsequence.

The set $K(H)$ of all compact operators $A \in L(H)$ form a closed linear subspace of $L(H)$ and thus a Banach space 
endowed with the operator norm.

The \emph{resolvent set} of $A$ is by definition the set of $z \in \C$ such that $A -z$ is invertible with 
\emph{resolvent} $R(z) := (A-z)^{-1} \in L(H)$. 
The resolvent set $P(A)$ is an open subset of $\C$, its complement is the \emph{spectrum} of $A$. 
As $A$ is compact, the spectrum of $A$ is a countable set which accumulates at zero, and zero is its only accumulation 
point. Every nonzero point in the spectrum is an eigenvalue of $A$ with finite multiplicity. See \cite[III.6.26]{Kato76}.

Let $A \in K(H)$ be a self-adjoint nonnegative compact operator. 
Let 
\[
    \la_1(A) \ge \la_2(A) \ge \cdots
\]
denote its decreasingly ordered eigenvalues.
Then, by the Courant--Fischer--Weyl min-max principle (see e.g.\ \cite[Chapter IV, Theorem 9.1]{Gohberg:2003aa} or 
\cite[Theorem XIII.1]{Reed:1980aa}), 
\begin{equation} \label{eq:minmax}
    \la_i  = \min_{\substack{M\subseteq H\\\dim M =i-1}} \max_{\substack{\|x\|=1\\ x \bot M}} \langle Ax,x\rangle.
\end{equation}

As a consequence we get the following version of Weyl's theorem (see e.g.\ \cite[Chapter IV, Section 4.9]{Gohberg:2003aa}).

\begin{proposition} \label[p]{p:coWeyl}
    Let $A,B \in K(H)$ be self-adjoint nonnegative compact operators. Then, for all $i \ge 1$,
    \[
        |\la_i(A) - \la_i(B)| \le \|A-B\|_{\on{op}}.
    \]
\end{proposition}

\begin{proof}
    Let $x \in H$ with $\|x\|=1$. Then 
    \[
        |\langle Ax,x \rangle - \langle Bx,x \rangle | = |\langle (A-B)x,x \rangle| \le \|A-B\|_{\on{op}}
    \]
    which means that 
    \[
        \langle Ax,x \rangle \le \langle Bx,x \rangle + \|A-B\|_{\on{op}}
    \]
    and 
    \[
        \langle Bx,x \rangle \le \langle Ax,x \rangle + \|A-B\|_{\on{op}}.
    \]
    Applying \eqref{eq:minmax}, we conclude that 
    \[
        \la_i(A)\le \la_i(B) + \|A-B\|_{\on{op}} \quad \text{ and } \quad \la_i(B)\le \la_i(A) + \|A-B\|_{\on{op}}
    \]
    which gives the statement.
\end{proof}

Let $U \subseteq \R^m$ be a bounded open set.
Consider a Lipschitz family $A \in C^{0,1}(\ol U, K(H))$ of compact operators 
such that each $A(x)$, for $x \in U$, is self-adjoint. 
Fix $x_0 \in U$.
Let $\la$ be an eigenvalue of $A(x_0)$ of multiplicity $d$.
Let $\Ga$ be a simple closed $C^1$ curve in the resolvent set $P(A(x_0))$ 
enclosing only $\la$ among all eigenvalues of $A(x_0)$.
Then there is an open neighborhood $V$ of $x_0$ such that 
$\Ga$ lies in $P(A(x))$, for $x \in V$,
and $A(x)$, for $x \in V$, has precisely $d$ eigenvalues (counted with multiplicities) 
in the interior of $\Ga$; see \cite[IV.2.16, IV.3.15, IV.3.16]{Kato76}.
Moreover, for $x \in V$ and $z$ close to any $z_0 \in P(A(x_0))$, the operator $A(x)-z$ is invertible with
resolvent $R(A(x),z) = (A(x) - z)^{-1} \in L(H)$.
We conclude that the map
\[
    (x,z) \mapsto R(A(x),z) = (A(x) - z)^{-1} \in L(H)
\]
is of class $C^{0,1}$ in $x$ and holomorphic in $z$, since inversion is analytic on $L(H)$.
Consequently, 
\[
   \Pi:  x \mapsto - \frac{1}{2\pi i} \int_\Ga R(A(x),z)\, dz = - \frac{1}{2\pi i} \int_\Ga (A(x)-z)^{-1}\, dz \in L(H)
\]
defines a Lipschitz family $\Pi$ of projections onto the direct sum of eigenspaces of the corresponding 
eigenvalues in the interior of $\Ga$.
We have that $A(x)$ commutes with $\Pi(x)$, and the spectrum of $A(x)$ inside $\Ga$ coincides 
with the spectrum of $A(x)\Pi(x) = \Pi(x)A(x)\Pi(x) \in L(\Pi(x)H)$; see \cite[III.6.17]{Kato76}. 

The image of $x \mapsto \Pi(x)$ describes a $d$-dimensional Lipschitz subbundle of the trivial bundle $U \times H \to U$. 
Indeed, choose $v_1,\ldots,v_d \in H$ such that the vectors $\Pi(x_0)v_i$ span $\Pi(x_0)H$. 
This remains true locally for $x$ near $x_0$.
By the Gram--Schmidt orthonormalization procedure (which is real analytic), 
we obtain a local orthonormal Lipschitz frame of the bundle.  
In this frame, $\Pi(x)A(x)\Pi(x)$ is given by a Hermitian $d\times d$ matrix parameterized in a Lipschitz way by $x$. 

For the reader's convenience, we restate \Cref{t:copintro}:

\begin{theorem} \label[t]{t:cop}
    Let $U \subseteq \R^m$ be a bounded open set. Let $H$ be a Hilbert space and $A, A_n \in C^{0,1}(\ol U,K(H))$, 
    for $n\ge 1$, Lipschitz families of compact self-adjoint nonnegative operators. Assume that 
    $A(x)$ is positive definite for all $x \in U$ and 
    $A_n \to A$ in $C^{0,1}(\ol U,K(H))$ as $n\to \infty$.
    Then:
    \begin{enumerate}
        \item The decreasingly ordered eigenvalues $\la_i(A)$ and $\la_i(A_n)$, for all $i$ and large enough $n$, belong to $C^{0,1}(\ol U)$. 
        \item  For all $1 \le j \le m$ and almost every $x \in U$,
            \[
                \lim_{n \to \infty} \p_j (\la_i(A_n))(x) = \p_j (\la_i(A))(x), \quad i =1,2,\ldots 
            \]
        \item For every $1 \le q < \infty$, 
            \[
                \lim_{n \to \infty} \|\la_i(A) - \la_i(A_n)\|_{W^{1,q}(U)} = 0, \quad i= 1,2,\ldots
            \]
    \end{enumerate}
\end{theorem}

\begin{proof}
    (1) 
    Since $A(x)$ is positive definite for all $x \in U$,
    all eigenvalues $\la_i(A(x))$ are positive for all $x \in U$. 
    By \Cref{p:coWeyl}, this also holds for $A_n$ if $n$ is large enough. 
    Then \Cref{p:coWeyl} implies that each $\la_i(A)$ and $\la_i(A_n)$, for large enough $n$, is Lipschitz 
    and belongs to $C^{0,1}(\ol U)$.
    
    (2)
    Fix $x_0 \in U$ and $i \ge 1$. 
    Let $\Ga$ be a simple closed $C^1$ curve in $P(A(x_0))$ enclosing only $\la_i(A(x_0))$ among all eigenvalues of $A(x_0)$. 
    As above, by \cite[IV.2.16, IV.3.15, IV.3.16]{Kato76}, we see that there is an open neighborhood $V$ of $x_0$ 
    and $n_0 \ge 1$ such that 
    $\Ga$ lies in $P(A(x))$ and $P(A_n(x))$, for $x \in V$ and $n \ge n_0$,
    and $A(x)$ and $A_n(x)$, for $x \in V$ and $n \ge n_0$, have precisely $d$ eigenvalues (counted with multiplicities) 
    in the interior of $\Ga$. 
    By \Cref{p:coWeyl},   
    $\la_i(A_n(x))$ is among those eigenvalues, if $n_0$ is large enough.
    
    The discussion before the theorem shows 
    that the $d$ eigenvalues of $A(x)$ as well as the $d$ eigenvalues of $A_n(x)$ 
    in the interior of $\Ga$ are precisely the $d$ eigenvalues of $B(x) := \Pi(x)A(x)\Pi(x)$ and $B_n(x) :=\Pi_n(x)A_n(x)\Pi_n(x)$, 
    given by a Lipschitz family of Hermitian $d \times d$ matrices, respectively. 
    Moreover, $B_n \to B$ in $C^{0,1}(\ol V, \on{Herm}(d))$ as $n \to \infty$, after slightly shrinking $V$ if necessary.
    Thus, \Cref{thm:mptw} implies that, for all $1 \le j \le m$ and almost every $x \in J$, we 
    have
    \[
        \lim_{n \to \infty} \p_j(\la_i(A_n))(x) = \p_j (\la_i(A))(x). 
    \]
    Since $U$ can be covered by countably many such open subsets $V$, this holds for almost every $x \in U$.

    (3) By \Cref{p:coWeyl}, for each $i \ge 1$ and $n \ge 1$, 
    \[
        |\la_i(A_n)|_{C^{0,1}(\ol U)} \le |A_n|_{C^{0,1}(\ol U,K(H))} \le L,
    \]
    for a positive constant $L$ independent of $n$ and $i$. Thus, for each $1 \le j \le m$ and almost every $x \in U$,
    \[
        |\p_j (\la_i(A_n))(x)| \le L.
    \]
    Consequently, in view of (2), the dominated convergence theorem implies that, 
    for each $1 \le j \le m$ and every $1\le q < \infty$,
    \[
        \|\p_j(\la_i(A)) - \p_j(\la_i(A_n))\|_{L^{q}(U)} \to 0 \quad \text{ as } n \to \infty.
    \]
    By \Cref{p:coWeyl}, we have 
    \[
                \|\la_i(A) - \la_i(A_n)\|_{L^\infty(U)} \to 0  \quad \text{ as } n \to \infty
    \]
    and (3) is proved.
\end{proof}

%---------------------------------------------------------------------------------------------
\subsection*{Acknowledgement}
%---------------------------------------------------------------------------------------------

We are grateful to Antonio Lerario for posing the question about the continuity of the solution map for polynomials 
which led us to study the continuity of the characteristic map for Hermitian and normal matrices.

This research was funded in whole or 
in part by the Austrian Science Fund (FWF) DOI 10.55776/PAT1381823.
For open access purposes, the authors have applied a CC BY public copyright license to any author-accepted manuscript version arising from this submission.

%\bibliography{../../references/biblio}

\begin{thebibliography}{AKLM98}

\bibitem[AKLM98]{AKLM98}
D.~Alekseevsky, A.~Kriegl, M.~Losik, and P.~W. Michor, \emph{Choosing roots of polynomials smoothly}, Israel J. Math. \textbf{105} (1998), 203--233.

\bibitem[Alm00]{Almgren00}
F.~J. Almgren, Jr., \emph{Almgren's big regularity paper}, World Scientific Monograph Series in Mathematics, vol.~1, World Scientific Publishing Co. Inc., River Edge, NJ, 2000. 

\bibitem[ADM90]{AmbrosioDalMaso90}
L.~Ambrosio and G.~Dal~Maso, \emph{A general chain rule for distributional derivatives}, Proc. Amer. Math. Soc. \textbf{108} (1990), no.~3, 691--702.

\bibitem[AGS08]{Ambrosio:2008aa}
L.~Ambrosio, N.~Gigli, and G.~Savar{\'e}, \emph{Gradient flows in metric spaces and in the space of probability measures}, 2nd ed., Basel: Birkh{\"a}user, 2008 (English).

\bibitem[BDM83]{BhatiaDavisMcIntosh83}
R.~Bhatia, C.~Davis, and A.~McIntosh, \emph{Perturbation of spectral subspaces and solution of linear operator equations}, Linear Algebra Appl. \textbf{52/53} (1983), 45--67.

\bibitem[Bha97]{Bhatia97}
R.~Bhatia, \emph{Matrix analysis}, Graduate Texts in Mathematics, vol. 169, Springer-Verlag, New York, 1997.

\bibitem[BCR98]{BCR} J. Bochnak, M. Coste,  M.-F. Roy, Real algebraic geometry, \textit{Ergebnisse der Mathematik und ihrer Grenzgebiete (3) [Results in Mathematics and Related Areas (3)]}, 
36. Springer-Verlag, Berlin, 1998.

\bibitem[Bog07]{Bogachev:2007aa}
V.~I. Bogachev, \emph{Measure theory. {Vol}. {I} and {II}}, Berlin: Springer, 2007 (English).

\bibitem[Bro79]{Bronshtein79}
M.~D. Bronshtein, \emph{Smoothness of roots of polynomials depending on parameters}, Sibirsk. Mat. Zh. \textbf{20} (1979), no.~3, 493--501, 690, English transl. in Siberian Math. J. \textbf{20} (1980), 347--352.

\bibitem[Coh84]{cohn} P. M. Cohn, Puiseux's theorem revisited,  \emph{J. Pure Appl. Algebra} \textbf{31} (1984) 1-4; Corrigendum:  \emph{J. Pure Appl. Algebra}, \textbf{52}, (1988), 197-198.

\bibitem[DLS11]{De-LellisSpadaro11}
C.~De~Lellis and E.~N. Spadaro, \emph{{$Q$}-valued functions revisited}, Mem. Amer. Math. Soc. \textbf{211} (2011), no.~991, vi+79. 

\bibitem[Ego11]{Egoroff:1911aa}
D.~Th. Egoroff, \emph{Sur les suites de fonctions mesurables.}, C. R. Acad. Sci., Paris \textbf{152} (1911), 244--246 (French).

\bibitem[GGK03]{Gohberg:2003aa}
I.~Gohberg, S.~Goldberg, and M.~A. Kaashoek, \emph{Basic classes of linear operators}, Birkh{\"a}user Basel, 2003.

\bibitem[HW53]{Hoffman:1953aa}
A.~J. Hoffman and H.~W. Wielandt, \emph{The variation of the spectrum of a normal matrix}, Duke Math. J. \textbf{20} (1953), 37--39 (English).

\bibitem[Huh01]{Huhtanen:2001aa}
M.~Huhtanen, \emph{A stratification of the set of normal matrices}, SIAM Journal on Matrix Analysis and Applications \textbf{23} (2001), no.~2, 349--367.

\bibitem[Kat76]{Kato76}
T.~Kato, \emph{Perturbation theory for linear operators}, second ed., Grundlehren der Mathematischen Wissenschaften, vol. 132, Springer-Verlag, Berlin, 1976.

\bibitem[KM03]{KM03}
A.~Kriegl and P.~W. Michor, \emph{Differentiable perturbation of unbounded operators}, Math. Ann. \textbf{327} (2003), no.~1, 191--201.

\bibitem[KMR11]{KMRp}
A.~Kriegl, P.~W. Michor, and A.~Rainer, \emph{{D}enjoy--{C}arleman differentiable perturbation of polynomials and unbounded operators}, Integr. Equ. Oper. Theory \textbf{71} (2011), no.~3, 407--416.

\bibitem[KMR12]{KMR}
\bysame, \emph{{M}any parameter {H}\"older perturbation of unbounded operators}, Math. Ann. \textbf{353} (2012), 519--522.

\bibitem[L\"ow34]{Lowner:1934aa}
K.~L\"owner, \emph{\"Uber monotone {M}atrixfunktionen}, Math. Z. \textbf{38} (1934), no.~1, 177--216. 

\bibitem[MSZ03]{MalySwansonZiemer03}
J.~Mal{{\'y}}, D.~Swanson, and W.~P. Ziemer, \emph{The co-area formula for {S}obolev mappings}, Trans. Amer. Math. Soc. \textbf{355} (2003), no.~2, 477--492 (electronic). 

\bibitem[MM72]{MarcusMizel72}
M.~Marcus and V.~J. Mizel, \emph{Absolute continuity on tracks and mappings of {S}obolev spaces}, Arch. Rational Mech. Anal. \textbf{45} (1972), 294--320.

\bibitem[MM79]{Marcus:1979aa}
\bysame, \emph{Every superposition operator mapping one {S}obolev space into another is continuous}, J. Functional Analysis \textbf{33} (1979), no.~2, 217--229. 

\bibitem[Mus91]{Musina:1991aa}
R.~Musina, \emph{On the continuity of the {N}emitsky operator induced by a {L}ipschitz continuous map}, Proc. Amer. Math. Soc. \textbf{111} (1991), no.~4, 1029--1041. 


\bibitem[PR15]{ParusinskiRainerHyp}
A.~Parusi{{\'n}}ski and A.~Rainer, \emph{A new proof of {B}ronshtein's theorem}, J. Hyperbolic Differ. Equ. \textbf{12} (2015), no.~4, 671--688.

\bibitem[PR16]{ParusinskiRainerAC}
\bysame, \emph{Regularity of roots of polynomials}, Ann.\ Sc.\ Norm.\ Super.\ Pisa Cl.\ Sci.\ (5) \textbf{16} (2016), 481--517.

\bibitem[PR18]{ParusinskiRainer15}
\bysame, \emph{Optimal {S}obolev regularity of roots of polynomials}, Ann.\ Sci.\ \'Ec.\ Norm.\ Sup\'er.\ (4) \textbf{51} (2018), no.~5, 1343--1387, doi:10.24033/asens.2376.

\bibitem[PR20a]{Parusinski:2020aa}
\bysame, \emph{Selections of bounded variation for roots of smooth polynomials}, Sel. Math. New Ser. \textbf{26} (2020), no.~13, doi:10.1007/s00029-020-0538-z.

\bibitem[PR25]{Parusinski:2023ab}
\bysame, \emph{Perturbation theory of polynomials and linear operators}, Handbook of Geometry and Topology of Singularities, vol. VII, pp.~121--202, Springer, 2025.

\bibitem[PR24a]{Parusinski:2024aa}
\bysame, \emph{Continuity of the solution map for hyperbolic polynomials}, https://arxiv.org/abs/2410.01321.

\bibitem[PR24b]{Parusinski:2024ab}
\bysame, \emph{On the continuity of the solution map for polynomials}, https://arxiv.org/abs/2410.01326.

\bibitem[PR20]{Parusinski:2020ac}
A.~Parusi{\'{n}}ski and G.~Rond, \emph{Multiparameter perturbation theory of matrices and linear operators}, Transactions of the American Mathematical Society \textbf{373} (2020), no.~4, 2933--2948.

\bibitem[Rai09]{RainerAC}
A.~Rainer, \emph{Perturbation of complex polynomials and normal operators}, Math. Nachr. \textbf{282} (2009), no.~12, 1623--1636.

\bibitem[Rai13]{RainerN}
\bysame, \emph{Perturbation theory for normal operators}, Trans. Amer. Math. Soc. \textbf{365} (2013), no.~10, 5545--5577. 

\bibitem[{Rai}22]{Rainer:2021vk}
\bysame, \emph{{Roots of G\r{a}rding hyperbolic polynomials}}, Proc. Amer. Math. Soc. \textbf{150} (2022), no.~6, 2433--2446.

\bibitem[RS80]{Reed:1980aa}
M.~Reed and B.~Simon, \emph{Methods of modern mathematical physics. {I}}, second ed., Academic Press, Inc. [Harcourt Brace Jovanovich, Publishers], New York, 1980, Functional analysis. \MR{751959}

\bibitem[Rel37]{Rellich37}
F.~Rellich, \emph{St{\"o}rungstheorie der {S}pektralzerlegung, {I}}, Math. Ann. \textbf{113} (1937), 600--619.

\bibitem[Wey12]{Weyl12}
H.~Weyl, \emph{{Das asymptotische Verteilungsgesetz der Eigenwerte linearer partieller Differentialgleichungen (mit einer Anwendung auf die Theorie der Hohlraumstrahlung)}}, Math. Ann. \textbf{71} (1912), 441--479 (German).

\bibitem[Zur93]{zurro} M.A. Zurro, The Abhyankar-Jung theorem revisited, \emph{J. Pure Appl. Algebra} \textbf{90} (1993), 275-282.

\end{thebibliography}
%\bibliographystyle{amsalpha}

\def\cprime{$'$}
\providecommand{\bysame}{\leavevmode\hbox to3em{\hrulefill}\thinspace}
\providecommand{\MR}{\relax\ifhmode\unskip\space\fi MR }
\providecommand{\MRhref}[2]{%
  \href{http://www.ams.org/mathscinet-getitem?mr=#1}{#2}
}
\providecommand{\href}[2]{#2}

\end{document}